\documentclass[11pt,a4paper,openany]{book}

\usepackage[utf8]{inputenc}
\usepackage{amsmath,amssymb,amsfonts,amsthm}
\usepackage{geometry}
\usepackage[hidelinks]{hyperref}
\usepackage{mathrsfs}
\usepackage{enumitem}

\newtheorem{theorem}{Theorem}[chapter]
\newtheorem{lemma}[theorem]{Lemma}
\newtheorem{proposition}[theorem]{Proposition}
\newtheorem{corollary}[theorem]{Corollary}

\theoremstyle{definition}
\newtheorem{definition}[theorem]{Definition}
\newtheorem{remark}[theorem]{Remark}
\newtheorem{conjecture}{Conjecture}
\newtheorem{question}{Question}[chapter]

\newcommand{\R}{\mathbb R}
\newcommand{\C}{\mathbb C}

\DeclareMathOperator{\re}{Re}
\DeclareMathOperator{\im}{Im}

\usepackage{listings}
\usepackage{xcolor}

\title{Variable Elliptic Structures on the Plane:\\
Transport Dynamics, Rigidity, and Function Theory}
\author{Daniel Alay\'on-Solarz\thanks{Email: danieldaniel@gmail.com}}
\date{March 2026}

\begin{document}
\maketitle

\thispagestyle{empty}

\begin{flushright}
\emph{
A fixed complex structure gives analysis.\\
A moving one gives transport.\\
Rigidity is where the two agree to dance.\\
}

\vspace{0.5em}
\small
(“Shut up and calculate!” — N. David Mermin)
\end{flushright}

\tableofcontents

\chapter*{Preface to v5}

This version corrects an error in the Leibniz defect formula
of Chapter~1 (formerly \S1.7): the operator $\partial_{\bar z}$
is a derivation for all variable elliptic structures, not only
rigid ones. The Rigidity Theorem has been updated accordingly.
Clarification between $p(x)$-analytics and $p$-analytics in the sense of Polozhii has been added.
Chapter~3 (Algebra--Spectral Intertwining) is new.  
The Poincar\'e residual has been renamed to Beltrami residual.

\bigskip
\noindent\textit{March 2026}

\chapter*{Introduction}
\addcontentsline{toc}{chapter}{Introduction}

In the classical theory of complex analysis, the imaginary unit $i$ is treated as
a fixed algebraic constant.
All analytic constructions are built on this assumption, and the underlying
complex structure of the plane is taken to be rigid and homogeneous.
From a geometric point of view, however, the symbol $i$ encodes the choice of a
complex structure, and there is no intrinsic reason for this structure to remain
constant once the geometry of the domain is allowed to vary.

The present work adopts the viewpoint that the imaginary unit should be treated
as intrinsic data: a moving generator $i(x,y)$ of a rank--two real algebra bundle
over the plane.
This generator is defined by the quadratic structure reduction
\[
i(x,y)^2+\beta(x,y)\,i(x,y)+\alpha(x,y)=0,
\]
whose coefficients $\alpha$ and $\beta$ are allowed to vary with position, subject
only to an ellipticity condition.
In this setting, the algebra itself becomes part of the geometry.

Once this perspective is adopted, a basic and unavoidable question arises:
\emph{what does it mean to differentiate the imaginary unit?}
Unlike the classical theory, where $i$ is constant by definition, the variable
framework forces the derivatives $i_x$ and $i_y$ to appear.
These derivatives are not auxiliary objects: they are well--defined algebra--valued
sections whose algebraic resolution is determined entirely by the structure
polynomial.
In particular, the differentiation of $i$ is meaningful before any notion of
holomorphicity or Cauchy--Riemann equation is introduced.

A central quantity emerging from this differentiation is the intrinsic
combination
\[
i_x+i\,i_y.
\]
This expression appears universally once the structure varies.
It measures the incompatibility between differentiation in the base variables
and multiplication in the fibers.
As a consequence, it governs the inhomogeneity of generalized
Cauchy--Riemann systems for the natural
first--order operators associated with the variable structure.
This phenomenon is not imposed by analytic assumptions; it is forced by the
geometry of the moving algebra itself.

A first major outcome of the theory is that the variability of the structure
coefficients $(\alpha,\beta)$ is governed by a transport law that is intrinsic and
unavoidable.
By differentiating the structure reduction and eliminating the derivatives of
$i$, the evolution of $(\alpha,\beta)$ can be expressed in terms of the spectral
parameter
\[
\lambda=\frac{-\beta+i\sqrt{4\alpha-\beta^2}}{2},\qquad \Im\lambda>0.
\]
In full generality, this leads to a (possibly forced) complex inviscid Burgers
equation written in the standard complex algebra.
This transport equation belongs to the structural layer of the theory: it arises
solely from allowing the elliptic structure to vary and does not depend on any
analytic integrability assumptions.

Within this universal transport framework, rigidity occupies a precise and
limited role.
The vanishing condition
\[
i_x+i\,i_y=0
\]
selects exactly the conservative case of the transport law.
In this regime, the Burgers equation becomes unforced, and two independent
obstructions collapse simultaneously: the real Cauchy--Riemann system becomes homogeneous, and the
generator $i(x,y)$ is transported compatibly with its own multiplication law.
Rigidity is therefore not the source of the coefficient dynamics, but a
compatibility condition that restores integrability within an already dynamical
setting.

The explicit $\varepsilon$--family of rigid elliptic structures constructed in
this work by standard techniques provides concrete examples of this conservative regime.
These examples show that rigidity does not force triviality: genuinely
nonconstant structures exist, although ellipticity naturally restricts the
domain of definition.
From the transport viewpoint, these domain restrictions arise because
the ellipticity condition $\Im\lambda>0$ imposes global constraints on
the admissible coefficient data.

On the analytic side, rigidity is exactly the condition under which a coherent
function theory can be recovered.
In the rigid regime, a Cauchy--Pompeiu formula adapted to variable elliptic
structures persists.
The natural Cauchy $1$--form
\[
d\tilde z=dy-i(z)\,dx
\]
is not closed, and its non--closure produces an unavoidable interior correction
term.
Rather than being an artifact, this correction identifies a canonical covariant
first--order operator whose null space provides the appropriate notion of
holomorphicity in the variable setting.
Although covariantly holomorphic sections are not closed under pointwise
multiplication, a scalar gauge transformation restores closure via a natural
weighted product.

The scope of this monograph is deliberately local and explicit.
Although many of the results admit reinterpretation within more abstract
frameworks, we work at the level of concrete first--order computations in order
to keep the geometric and dynamical mechanisms fully visible.
From this perspective, variable elliptic structures reveal a natural hierarchy:
transport first, integrability second, and analytic closure as a distinguished
limiting case.

The investigation of variable elliptic structures was originally suggested to the author by W.~Tutschke and C.~J.~Vanegas, whose work on parameter--depending Clifford Algebras and generalized Cauchy--Pompeiu formulas provided both motivation and guidance.

At an early stage (circa 2012), the author did not yet accept the necessity of
differentiating the imaginary unit itself.
It was H.~De~Bie who emphasized that allowing derivatives of the generator
$i(x,y)$ is not an auxiliary choice but a requirement for internal consistency:
without $i_x$ and $i_y$, differentiation and multiplication cannot be reconciled
once the elliptic structure varies.
This remark proved decisive in clarifying the geometric foundations of the
framework, and the author is grateful to H.~De~Bie for this insight.

Although W.~Tutschke did not see the full development of the theory presented here, his influence is present throughout in the guiding questions and structural viewpoint.

\medskip
\begin{flushright}
\emph{Dedicated to the memory of W.~Tutschke.}
\end{flushright}

\chapter{Variable Elliptic Structures on the Plane: Definitions}\label{chap:def}

\section{Structure Polynomial and Ellipticity}

Let $\Omega \subset \R^2$ be an open domain with coordinates $(x,y)$.
Let $\alpha,\beta \in C^1(\Omega,\R)$ satisfy the ellipticity condition
\begin{equation}\label{eq:ellipticity}
\Delta(x,y) := 4\alpha(x,y) - \beta(x,y)^2 > 0
\quad \text{for all } (x,y)\in\Omega.
\end{equation}

\begin{definition}[Variable elliptic structure]
A \emph{variable elliptic structure} on $\Omega$ is the real rank--two
algebra bundle $\{A_z\}_{z\in\Omega}$ generated by an element $i=i(x,y)$
satisfying the variable-coefficient structure polynomial
\begin{equation}\label{eq:structure-polynomial}
X^2 + \beta(x,y)\,X + \alpha(x,y)
\end{equation}
so we can always apply the structure reduction
\begin{equation}\label{eq:structure}
i^2 + \beta(x,y)\,i + \alpha(x,y) = 0 .
\end{equation}
\end{definition}

Each fiber $A_z$ is canonically identified with
$\R[X]/(X^2+\beta(z)X+\alpha(z))$.
By \eqref{eq:ellipticity}, this algebra is elliptic and hence isomorphic
to $\C$ as a real algebra, though no canonical identification between
fibers is assumed.

We denote the conjugate root by
\[
\hat{i}(z) := -\beta(z) - i(z).
\]

\section{Algebra-Valued Functions}

A function $f:\Omega\to A$ is written uniquely as
\[
f = u + v\,i,
\qquad u,v\in C^1(\Omega,\R).
\]
Addition and multiplication are taken pointwise using the algebra
structure induced by \eqref{eq:structure}.

\section{Differentiating the generator}\label{sec:ixiy}

A basic structural question is whether the moving generator $i=i(x,y)$,
defined by
\[
i^2+\beta(x,y)i+\alpha(x,y)=0,
\]
can be differentiated in a meaningful way. The generator is not independent
data but a chosen $C^1$ branch of a root of this polynomial, whose
coefficients $\alpha,\beta$ lie in $C^1(\Omega)$. Its differentiability
therefore follows by implicit differentiation of the structure relation.

Differentiating with respect to $x$ and $y$ gives
\[
(2i+\beta)\,i_x+\alpha_x+\beta_x i=0,
\qquad
(2i+\beta)\,i_y+\alpha_y+\beta_y i=0.
\]
Whenever $2i+\beta$ is invertible in the fiber algebra (i.e.\ away from
degeneracy of the quadratic), these equations uniquely determine $i_x$ and
$i_y$ as algebra-valued functions depending only on $\alpha,\beta$ and their
first derivatives. In the elliptic regime this invertibility is automatic,
so the derivatives of the generator are explicit consequences of the
structure reduction rather than additional geometric data.
The computation relies only on the invertibility of the element $2i+\beta$
in the fiber algebra. This condition holds automatically in the elliptic
regime and, more generally, in any nondegenerate quadratic algebra away
from zero divisors. What distinguishes the elliptic case is not the
algebraic solvability of these equations but the uniform positivity and
invertibility properties that follow from \eqref{eq:ellipticity}.

\subsection{Invertibility of \texorpdfstring{$2i+\beta$}{2i + beta} in the elliptic regime}

Recall that $i$ satisfies $i^2+\beta i+\alpha=0$ with $\Delta:=4\alpha-\beta^2>0$.
A direct computation using $i^2=-\beta i-\alpha$ gives
\[
(2i+\beta)(-\beta-2i)=\Delta.
\]
Thus $2i+\beta$ is invertible in every fiber and
\begin{equation}\label{eq:inv2i}
(2i+\beta)^{-1}=\frac{-\beta-2i}{\Delta}.
\end{equation}

\subsection{Coefficient decomposition}

Since each elliptic fiber is two-dimensional over $\R$ with basis $\{1,i\}$,
every section admits a unique decomposition in this basis.  In particular,
there exist unique real functions $A_x,B_x,A_y,B_y$ on $\Omega$ such that
\begin{equation}\label{eq:ixiy-decomp}
i_x=A_x+B_x\,i,
\qquad
i_y=A_y+B_y\,i.
\end{equation}
Moreover, \eqref{eq:ixiy-formulas} shows that $A_x,B_x,A_y,B_y$ depend only on
$\alpha,\beta$ and their first derivatives.

\subsection{Solving for \texorpdfstring{$i_x$}{i\_x} and \texorpdfstring{$i_y$}{i\_y}}

Differentiate the structure relation $i^2+\beta i+\alpha=0$ with respect to $x$
and $y$:
\begin{equation}\label{eq:ix-eq}
(2i+\beta)i_x+\alpha_x+\beta_x i=0,
\qquad
(2i+\beta)i_y+\alpha_y+\beta_y i=0.
\end{equation}
Multiplying by $(2i+\beta)^{-1}$ and using \eqref{eq:inv2i} yields the explicit
formulas
\begin{equation}\label{eq:ixiy-formulas}
i_x=-\frac{\alpha_x+\beta_x i}{2i+\beta},
\qquad
i_y=-\frac{\alpha_y+\beta_y i}{2i+\beta}.
\end{equation}

\section{Generalized Cauchy--Riemann Operators}

\begin{definition}[Generalized Cauchy--Riemann operators]
We define
\begin{equation}\label{eq:dbar}
\partial_{\bar z}
:= \frac12\bigl(\partial_x + i\,\partial_y\bigr),
\qquad
\partial_z
:= \frac12\bigl(\partial_x + \hat{i}\,\partial_y\bigr).
\end{equation}
\end{definition}

These operators act on algebra--valued functions using the product in
$A_z$. No derivation property is assumed.

\section{The Variable Cauchy--Riemann System}

Let $f=u+vi$ with $u,v\in C^1(\Omega,\R)$.
A direct computation yields the identity
\begin{equation}\label{eq:CRdecomp}
2\,\partial_{\bar z}f
=
(u_x - \alpha v_y)
+
(v_x + u_y - \beta v_y)\,i
+
v\,(i_x + i\,i_y).
\end{equation}

Since each fiber is elliptic, $\{1,i\}$ is a real basis of $A_z$.
Hence there exist unique real functions $A(x,y)$ and $B(x,y)$ such that
\begin{equation}\label{eq:noise}
i_x + i\,i_y = A + B\,i.
\end{equation}

Substituting \eqref{eq:noise} into \eqref{eq:CRdecomp}, the equation
$\partial_{\bar z}f=0$ is equivalent to the real first--order system
\begin{equation}\label{eq:CRsystem}
\begin{cases}
u_x - \alpha v_y + A v = 0,\\[4pt]
v_x + u_y - \beta v_y + B v = 0.
\end{cases}
\end{equation}

\begin{remark}
The coefficients $A$ and $B$ encode all deviations from the classical
Cauchy--Riemann equations.
At this level they are simply defined by \eqref{eq:noise}, with no
structural assumptions imposed.
\end{remark}

\section{Explicit Formula for the Inhomogeneous Coefficients}

Using the formulas for $i_x$ and $i_y$ obtained in
\eqref{eq:ixiy-formulas}, we compute the intrinsic combination
$i_x + i\,i_y$ in terms of $\alpha$, $\beta$ and their derivatives.

Using ellipticity to invert $2i+\beta$, one obtains
\[
i_x + i\,i_y
=
\frac{
-(\alpha_x-\alpha\beta_y)
-(\beta_x+\alpha_y-\beta\beta_y)i
}{
2i+\beta
}.
\]

Since $(2i+\beta)^{-1}=(-\beta-2i)/\Delta$, a direct reduction yields
\begin{equation}\label{eq:AB}
\begin{aligned}
A &=
\frac{
\beta(\alpha_x-\alpha\beta_y)
-2\alpha(\beta_x+\alpha_y-\beta\beta_y)
}{\Delta},\\[6pt]
B &=
\frac{
2(\alpha_x-\alpha\beta_y)
-\beta(\beta_x+\alpha_y-\beta\beta_y)
}{\Delta}.
\end{aligned}
\end{equation}

Thus the system \eqref{eq:CRsystem} depends only on $\alpha$, $\beta$ and
their first derivatives.

\chapter{Differentiating the Structure: Universal Burgers Transport}
\label{chap:burgers-universal}

\section{Purpose of this Chapter}

The previous chapters introduced variable elliptic structures and derived the
basic algebraic and differential identities that follow from the structure
polynomial
\[
i^2+\beta(x,y)\,i+\alpha(x,y)=0.
\]
Up to this point, no integrability or compatibility assumptions have been
imposed.

The purpose of the present chapter is to describe what is \emph{unavoidable}
once the generator $i=i(x,y)$ is allowed to vary.
In particular, we identify the intrinsic quantity that governs all deviations
from the classical theory and show that the coefficients of the structure
polynomial are necessarily subject to a transport law.

\section{The Intrinsic Obstruction}

Let $\Omega\subset\R^2$ be open, and let $\alpha,\beta\in C^1(\Omega,\R)$ satisfy
the ellipticity condition
\[
4\alpha-\beta^2>0.
\]
Fix a $C^1$ choice of generator $i:\Omega\to A$ solving
\[
i^2+\beta i+\alpha=0.
\]

\begin{definition}[Intrinsic obstruction]
The \emph{intrinsic obstruction} associated with the generator $i$ is the
$A$--valued section
\begin{equation}\label{eq:intrinsic-obstruction}
\mathcal G:=i_x+i\,i_y .
\end{equation}
\end{definition}

This quantity is well defined wherever $i$ is differentiable.
It does not depend on any analytic structure or auxiliary operator.
All deviations from the constant theory enter through $\mathcal G$.

Writing $\mathcal G$ in the moving basis,
\begin{equation}\label{eq:G-decomp}
\mathcal G=G_0+G_1\,i,
\qquad
G_0,G_1\in C^0(\Omega,\R),
\end{equation}
makes explicit the dependence on the chosen branch of the generator.
The pair $(G_0,G_1)$ should be regarded as intrinsic coefficient data attached to
the variable structure.

\section{Eliminating the Derivatives of the Generator}

Differentiating the structure reduction with respect to $x$ and $y$ gives
\[
(2i+\beta)i_x+\alpha_x+\beta_x i=0,
\qquad
(2i+\beta)i_y+\alpha_y+\beta_y i=0.
\]
Adding the second equation multiplied by $i$ to the first and using
$i^2=-\beta i-\alpha$ yields the identity
\begin{equation}\label{eq:elim}
(2i+\beta)\,(i_x+i\,i_y)
=
(-\alpha_x+\alpha\beta_y)
+
(-\beta_x-\alpha_y+\beta\beta_y)\,i .
\end{equation}

Substituting $\mathcal G=G_0+G_1 i$ into \eqref{eq:elim} and reducing in the basis
$\{1,i\}$ gives the following first--order system for the coefficients.

\begin{proposition}[Forced coefficient system]\label{prop:forced-system}
Let $\mathcal G=G_0+G_1 i$ be the intrinsic obstruction associated with a fixed
generator $i$.
Then the coefficients $\alpha$ and $\beta$ satisfy
\begin{equation}\label{eq:forced-coeff}
\begin{cases}
\alpha_x
=
\alpha\,\beta_y
-\beta\,G_0
+2\alpha\,G_1,
\\[6pt]
\beta_x+\alpha_y
=
\beta\,\beta_y
-2G_0
+\beta\,G_1 .
\end{cases}
\end{equation}
\end{proposition}

\begin{proof}
Using $i^2=-\beta i-\alpha$, one computes
\[
(2i+\beta)(G_0+G_1 i)
=
(\beta G_0-2\alpha G_1)+(2G_0-\beta G_1)\,i.
\]
Equating coefficients with the right--hand side of \eqref{eq:elim} yields
\eqref{eq:forced-coeff}.
\end{proof}

\section{Canonical Spectral Parameter}

Under ellipticity, write
\[
\beta=-2a,
\qquad
\alpha=a^2+b^2,
\qquad
b>0,
\]
and define the complex--valued spectral parameter
\begin{equation}\label{eq:spectral}
\lambda:=a+ib
=
\frac{-\beta+i\sqrt{4\alpha-\beta^2}}{2}.
\end{equation}

This change of variables depends only on the structure reduction.
It introduces no new assumptions and does not involve differentiation. The sign choice $b>0$ selects the branch corresponding to the elliptic
orientation and ensures that $\lambda$ takes values in the upper half–plane.

\section{Universal Transport Law}

\begin{theorem}[Universal transport equation]\label{thm:universal-burgers}
Let $\mathcal G=G_0+G_1 i$ be the intrinsic obstruction.
Then the forced coefficient system \eqref{eq:forced-coeff} is equivalent to the
complex first--order transport equation
\begin{equation}\label{eq:universal-burgers}
\lambda_x+\lambda\,\lambda_y
=
G_0+\lambda\,G_1 .
\end{equation}
\end{theorem}

\begin{proof}
Substituting $\alpha=a^2+b^2$ and $\beta=-2a$ into
\eqref{eq:forced-coeff} yields
\[
\begin{cases}
a_x+a a_y-b b_y=G_0+a G_1,\\[4pt]
b_x+a b_y+b a_y=b G_1,
\end{cases}
\]
which are precisely the real and imaginary parts of
$\lambda_x+\lambda\,\lambda_y=G_0+\lambda G_1$.
The converse follows by reversing the computation.
\end{proof}

\begin{remark}[The Burgers equation is written in the standard $\C$]
Although each fiber $A_z$ is (elliptic and hence) real--algebra isomorphic to $\C$,
there is in general no canonical identification between different fibers, and we
deliberately avoid fixing a trivialization.  The spectral parameter
$\lambda=a+ib=\frac{-\beta+i\sqrt{4\alpha-\beta^2}}{2}\in\mathbb C_{+}$ is therefore not
an element of the moving algebra $A_z$, but a scalar function of the real
coefficients $(\alpha,\beta)$.  Consequently $\lambda_x$ and $\lambda_y$ are ordinary
partial derivatives, and the product $\lambda\,\lambda_y$ is the usual complex
multiplication in the fixed target $\C$.  Equivalently, \eqref{eq:universal-burgers}
is nothing but a compact rewriting of the real quasilinear system for $(a,b)$.
The variable fiber multiplication enters in the analytic layer of the theory, but it plays no role
in the coefficient transport law.
\end{remark}

Equation \eqref{eq:universal-burgers} is a forced complex inviscid Burgers equation
written in the standard complex algebra $\C$.
It governs the evolution of the structure coefficients for \emph{all} variable
elliptic structures.

\section{Ellipticity of the Transport System}

The transport equation \eqref{eq:universal-burgers} has the form
\[
(\partial_x+\lambda\,\partial_y)\lambda = G_0+\lambda\,G_1,
\]
where $\lambda=a+ib$ with $b>0$.  It is natural to ask whether this
equation admits a characteristic interpretation in $\R^2$.

Writing \eqref{eq:universal-burgers} as a real first--order system for $(a,b)$,
the principal part in the $y$--direction is governed by the coefficient matrix
\[
A=\begin{pmatrix} a & -b \\ b & a \end{pmatrix},
\]
whose eigenvalues are $a\pm ib=\lambda$ and $\bar\lambda$.  Since $b>0$, the
eigenvalues are strictly complex.  Consequently the real system is \emph{elliptic},
not hyperbolic, and possesses no real characteristic curves.

This is not a defect of the formulation but a structural necessity:
the ellipticity condition $4\alpha-\beta^2>0$ that defines the variable
elliptic structure propagates directly into the ellipticity of the
transport system governing its coefficients.  The structure is elliptic
at every level of the theory.

\begin{remark}[Complex vector field]
The operator $\partial_x+\lambda\,\partial_y$ is a complex vector field on $\R^2$.
Its ellipticity is equivalent to the statement that it has no real null curves,
i.e.\ the equation $dy/dx=\lambda$ admits no real--valued solution.
In the language of several complex variables, $\partial_x+\lambda\,\partial_y$ fails
to be tangent to any real curve precisely because $\operatorname{Im}\lambda\neq 0$.
\end{remark}
\section{Summary}

Allowing the generator $i(x,y)$ to vary forces the appearance of the intrinsic
obstruction $\mathcal G=i_x+i\,i_y$.
Eliminating the derivatives of $i$ from the structure reduction shows that the
coefficients $(\alpha,\beta)$ are necessarily governed by a transport equation
for the spectral parameter $\lambda$.

This transport law is universal.
No assumption has been made concerning algebraic compatibility, homogeneity of
the Cauchy--Riemann system, or existence of integral representations.

In the next chapter we identify the special regime in which the obstruction
vanishes and show that, in this case, the transport becomes conservative and a
coherent analytic function theory emerges.


\chapter{The Algebra--Spectral Intertwining}
\label{ch:intertwining}

The universal transport law of Chapter~\ref{chap:burgers-universal} governs the evolution of the
spectral parameter $\lambda$.  The Cauchy--Riemann operator
$\partial_{\bar z}=\frac{1}{2}(\partial_x+i\,\partial_y)$ of
Chapter~\ref{chap:def} governs the function theory at the algebra level.
This chapter establishes the exact relationship between the two:
the spectral map intertwines the algebra operator with the spectral
transport operator, universally, for all variable elliptic structures.

\section{The spectral map}
\label{sec:spectral-map}

Recall from Chapter~\ref{chap:def} that the fiber algebra $A_z=\R[X]/(X^2+\beta X+\alpha)$
is generated over $\R$ by the moving generator $i=i(x,y)$ satisfying
$i^2+\beta\,i+\alpha=0$.  Under ellipticity ($4\alpha-\beta^2>0$),
$A_z$ is isomorphic to $\C$ as a real algebra.  The spectral parameter
\[
\lambda
=\frac{-\beta+i\sqrt{4\alpha-\beta^2}}{2}
\]
is the image of the generator $i$ under this isomorphism.

\begin{definition}[Spectral map]
\label{def:spectral-map}
The \emph{spectral map} sends an $A$--valued section $W=U+Vi$
($U,V:\Omega\to\R$) to its complex--valued spectral image
\[
W_\lambda:=U+V\lambda\in\C.
\]
\end{definition}

The spectral map is a pointwise real--algebra isomorphism.  Two
properties are immediate.

\begin{proposition}[Injectivity]
\label{prop:spectral-injective}
Under ellipticity, $W_\lambda=0$ implies $W=0$.
\end{proposition}

\begin{proof}
$U+V\lambda=0$ with $\im\lambda>0$ gives $U=0$ and $V=0$.
\end{proof}

\begin{proposition}[Ring homomorphism]
\label{prop:spectral-ring}
For any $A$--valued sections $W=U+Vi$ and $Z=P+Qi$,
\[
(WZ)_\lambda=W_\lambda\cdot Z_\lambda.
\]
\end{proposition}

\begin{proof}
The spectral map $i\mapsto\lambda$ is a real--algebra homomorphism
on each fiber.
\end{proof}

\section{The Leibniz rule for \texorpdfstring{$\partial_{\bar z}$}{dbar z}}
\label{sec:leibniz}

Before establishing the intertwining, we show that the algebra--level
operator is always a derivation.

\begin{proposition}[Universal Leibniz rule]
\label{prop:leibniz}
The operator $\partial_{\bar z}=\frac{1}{2}(\partial_x+i\,\partial_y)$
satisfies the Leibniz rule on $C^1(\Omega;A)$:
\[
\partial_{\bar z}(fg)
=(\partial_{\bar z}f)\,g+f\,(\partial_{\bar z}g)
\]
for all $A$--valued $C^1$ sections $f,g$.
This holds for all variable elliptic structures.
\end{proposition}

\begin{proof}
The structure relation $i^2+\beta\,i+\alpha=0$ is an identity of
$C^1$ functions on $\Omega$.  Differentiating with respect to $x$:
\[
2i\,i_x+\beta\,i_x+\alpha_x+\beta_x\,i=0,
\]
which is the identity $(2i+\beta)i_x=-\alpha_x-\beta_x\,i$.
This is exactly the condition that $\partial_x(i^2)=2i\cdot i_x$
holds in the fiber algebra.  The same argument applies to $\partial_y$.
Hence $\partial_x$ and $\partial_y$ are derivations on $C^1(\Omega;A)$.

Since the fiber algebra is commutative, any $A$--linear combination
of derivations is again a derivation.  In particular,
$\partial_{\bar z}=\frac{1}{2}(\partial_x+i\,\partial_y)$ is a
derivation.
\end{proof}

\section{The universal spectral identity}
\label{sec:universal-identity}

\begin{theorem}[Universal spectral identity]
\label{thm:universal-spectral}
For every $A$--valued $C^1$ section $W=U+Vi$:
\begin{equation}\label{eq:universal-spectral}
2\bigl(\partial_{\bar z}W\bigr)_\lambda
=(W_\lambda)_x+\lambda\,(W_\lambda)_y.
\end{equation}
This holds for all variable elliptic structures.
\end{theorem}

\begin{proof}
From the algebra--level Cauchy--Riemann decomposition, writing
$\mathcal G=G_0+G_1\,i$:
\[
2\,\partial_{\bar z}W
=(U_x-\alpha V_y+VG_0)+(V_x+U_y-\beta V_y+VG_1)\,i.
\]
Its spectral image is
\begin{equation}\label{eq:lhs}
2(\partial_{\bar z}W)_\lambda
=(U_x-\alpha V_y+VG_0)+\lambda(V_x+U_y-\beta V_y+VG_1).
\end{equation}

From $W_\lambda=U+V\lambda$, using $\lambda^2=-\beta\lambda-\alpha$
(the spectral image of the structure relation):
\begin{align}
(W_\lambda)_x+\lambda(W_\lambda)_y
&=U_x+V_x\lambda+V\lambda_x
+\lambda(U_y+V_y\lambda+V\lambda_y)\notag\\
&=(U_x-\alpha V_y)+(V_x+U_y-\beta V_y)\lambda
+V(\lambda_x+\lambda\lambda_y).\label{eq:rhs}
\end{align}
The universal transport law (Chapter~\ref{chap:burgers-universal}) gives
$\lambda_x+\lambda\lambda_y=G_0+G_1\lambda$,
so
\[
V(\lambda_x+\lambda\lambda_y)=VG_0+VG_1\lambda.
\]
Substituting into \eqref{eq:rhs} yields \eqref{eq:lhs}.
\end{proof}

Identity \eqref{eq:universal-spectral} says that the spectral map
sends the algebra--level operator $\partial_{\bar z}$ to the
first--order transport operator
$\frac{1}{2}(\partial_x+\lambda\,\partial_y)$:
\[
\begin{array}{ccc}
C^1(\Omega;A) & \xrightarrow{\;\partial_{\bar z}\;} & C^0(\Omega;A)\\[3pt]
\downarrow\scriptstyle{W\mapsto W_\lambda} & & \downarrow\scriptstyle{W\mapsto W_\lambda}\\[3pt]
C^1(\Omega;\C) & \xrightarrow{\;\frac{1}{2}(\partial_x+\lambda\partial_y)\;}
& C^0(\Omega;\C)
\end{array}
\]
Both horizontal arrows are derivations (the top by
Proposition~\ref{prop:leibniz}, the bottom because it is first--order),
and the vertical arrows are ring homomorphisms
(Proposition~\ref{prop:spectral-ring}).
The diagram commutes for all variable elliptic structures.

\section{Consequences}
\label{sec:consequences}

\begin{corollary}[Algebra holomorphicity implies spectral transport]
\label{cor:holomorphic-transport}
If\/ $\partial_{\bar z}W=0$, then
$(W_\lambda)_x+\lambda(W_\lambda)_y=0$.
\end{corollary}

\begin{proof}
Immediate from \eqref{eq:universal-spectral}: the left side vanishes.
\end{proof}

\begin{corollary}[Kernel correspondence]
\label{cor:kernel}
The spectral map restricts to a ring isomorphism between
$\ker\partial_{\bar z}$ and
$\ker\frac{1}{2}(\partial_x+\lambda\partial_y)$.
\end{corollary}

\begin{proof}
By the intertwining identity, $W$ is in the kernel of
$\partial_{\bar z}$ if and only if $W_\lambda$ is in the kernel of
$\frac{1}{2}(\partial_x+\lambda\partial_y)$.  Both maps preserve
products (Proposition~\ref{prop:spectral-ring} and the Leibniz
rules).
\end{proof}

\section{Commutativity with higher derivatives}
\label{sec:higher-derivatives}

The spectral map also commutes with pure $y$--derivatives.

\begin{proposition}
\label{prop:spectral-dyy}
For every $A$--valued $C^2$ section $W=U+Vi$:
\[
(W_{yy})_\lambda=(W_\lambda)_{yy}.
\]
The same holds for $\partial_{xx}$ and all higher pure partial
derivatives.
\end{proposition}

\begin{proof}
$W_{yy}=U_{yy}+V_{yy}\,i+2V_y\,i_y+V\,i_{yy}$.
Under $i\mapsto\lambda$:
$(W_{yy})_\lambda=U_{yy}+V_{yy}\lambda+2V_y\lambda_y+V\lambda_{yy}$.
From $W_\lambda=U+V\lambda$:
$(W_\lambda)_{yy}=U_{yy}+V_{yy}\lambda+2V_y\lambda_y+V\lambda_{yy}$.
The argument for $\partial_{xx}$ is identical.
\end{proof}

\section{Interpretation}
\label{sec:interpretation}

The results of this chapter are summarized as follows.

\begin{enumerate}
\item The algebra--level operator $\partial_{\bar z}$ is a derivation
on $C^1(\Omega;A)$, universally
(Proposition~\ref{prop:leibniz}).

\item The spectral map $W\mapsto W_\lambda$ is a pointwise ring
isomorphism (Proposition~\ref{prop:spectral-ring}).

\item These two structures are compatible: the spectral map
intertwines $\partial_{\bar z}$ with the first--order transport
operator $\frac{1}{2}(\partial_x+\lambda\partial_y)$
(Theorem~\ref{thm:universal-spectral}).

\item The intertwining extends to second--order operators via
commutativity with pure partial derivatives
(Proposition~\ref{prop:spectral-dyy}).
\end{enumerate}

The obstruction $\mathcal G=i_x+i\,i_y$ appears in the proof of the
intertwining identity through the universal transport law, but it does
not prevent any of these results from holding.  Its role is to force
the transport law and to make the Cauchy--Riemann system
inhomogeneous---not to obstruct the algebraic or intertwining
properties of $\partial_{\bar z}$.

\chapter{Rigidity as the Conservative Regime of Burgers Transport}
\label{chap:rigidity}

\section{From Universal Transport to a Distinguished Regime}

Chapter~\ref{chap:burgers-universal} showed that allowing the elliptic structure
to vary forces a transport law on its spectral data.
Independently of any analytic assumptions, the differentiation of the generator
$i(x,y)$ leads to a (possibly forced) complex Burgers equation governing the
evolution of the coefficients $(\alpha,\beta)$ through the spectral parameter
\[
\lambda=\frac{-\beta+i\sqrt{4\alpha-\beta^2}}{2},\qquad \Im\lambda>0.
\]

At this level, no integrability has been imposed.
The transport equation may include forcing terms, and no closure properties
for the associated function theory are expected.

The purpose of this chapter is to identify a \emph{distinguished subregime}
of this universal transport law in which several independent-looking obstructions
collapse simultaneously.
This regime is characterized by the vanishing of a single intrinsic quantity,
namely
\[
\mathcal G := i_x + i\,i_y .
\]

We call this condition \emph{rigidity}.
It does not generate the transport law; rather, it singles out the
\emph{conservative case} of the already-existing Burgers dynamics.

\section{The Obstruction Field}

Recall that $i(x,y)$ satisfies the structure reduction
\[
i^2+\beta i+\alpha=0,
\]
and that its derivatives $i_x$ and $i_y$ are well-defined algebra-valued
sections by the implicit differentiation of the structure relation
(Chapter~\ref{chap:def}).

\begin{definition}[Obstruction field]
The \emph{obstruction field} associated to a variable elliptic structure is
\[
\mathcal G := i_x + i\,i_y .
\]
\end{definition}

This quantity measures the failure of the generator $i(x,y)$ to be transported
compatibly with its own multiplication law.
It appears naturally and unavoidably once differentiation is allowed.

In Chapter~\ref{chap:burgers-universal}, prescribing $\mathcal G$ was shown to
produce a forced Burgers equation for the spectral parameter.
We now study the special case $\mathcal G\equiv 0$.

\section{Definition of Rigidity}

\begin{definition}[Rigid elliptic structure]
A variable elliptic structure is called \emph{rigid} if its obstruction field
vanishes identically:
\begin{equation}\label{eq:rigidity}
i_x + i\,i_y = 0
\qquad \text{on }\Omega.
\end{equation}
\end{definition}

This condition is intrinsic and purely geometric.
It does not refer to any function space, product rule, or analytic notion of
holomorphicity.

\section{Rigidity and Conservative Transport}

In the notation of Chapter~\ref{chap:burgers-universal}, the universal forced
Burgers equation takes the form
\[
\lambda_x + \lambda\,\lambda_y = G_0 + \lambda\,G_1,
\]

where the forcing terms $(G_0,G_1)$ are the real coefficients of
$\mathcal G=G_0+G_1 i$ in the moving basis, as in \eqref{eq:G-decomp}.

\begin{proposition}[Rigidity as conservation]\label{prop:rigid-conservative}
A variable elliptic structure is rigid if and only if the associated Burgers
transport law is conservative:
\[
\lambda_x + \lambda\,\lambda_y = 0 .
\]
\end{proposition}

\begin{proof}
By Theorem~\ref{thm:universal-burgers}, the coefficients satisfy
$\lambda_x+\lambda\lambda_y=G_0+\lambda G_1$.
Hence $\mathcal G\equiv0$ (i.e.\ $G_0\equiv G_1\equiv0$) is equivalent to
$\lambda_x+\lambda\lambda_y=0$.
\end{proof}

This establishes rigidity as a \emph{dynamical} notion: it identifies those
variable elliptic structures whose coefficient transport is conservative.

\section{Inhomogeneity collapse}

The real Cauchy--Riemann system derived in Chapter~2 has the form
\[
\begin{cases}
u_x - \alpha v_y + A v = 0,\\
v_x + u_y - \beta v_y + B v = 0,
\end{cases}
\]
where $(A,B)$ are the real coefficients of $\mathcal G$.

\begin{proposition}[Inhomogeneity collapse]
The generalized Cauchy--Riemann system is homogeneous if and only if the
structure is rigid.
\end{proposition}

\begin{proof}
Homogeneity is equivalent to $A\equiv B\equiv 0$, which is equivalent to
$\mathcal G\equiv 0$.
\end{proof}

\section{The Rigidity Theorem}

We may now state the rigidity theorem in its correct logical position.

\begin{theorem}[Rigidity Theorem]\label{thm:rigidity}
Let $\alpha,\beta\in C^1(\Omega)$ satisfy ellipticity, and let $i(x,y)$ be a
$C^1$ choice of generator.

The following statements are equivalent:
\begin{enumerate}
\item The obstruction field vanishes: $i_x+i\,i_y=0$.
\item The Burgers transport law is conservative:
     $\lambda_x+\lambda\lambda_y=0$.
\item The generalized real Cauchy--Riemann system is homogeneous.
\end{enumerate}
\end{theorem}

\begin{proof}
The equivalence $(1)\Leftrightarrow(2)$ was shown above.
Equivalences with $(3)$ follow from the inhomogeneity collapse result.
\end{proof}

\begin{remark}[Rigidity as intrinsic holomorphicity]
The rigidity condition
\[
i_x+i\,i_y=0
\]
admits an equivalent intrinsic reformulation.
It states that the generator $i(x,y)$, viewed as the generator section of the rank--two algebra bundle it defines, is holomorphic with respect to its own
variable elliptic structure:

\[
\partial_{\bar z}i=0
\]

Indeed, since $\partial_{\bar z}=\tfrac12(\partial_x+i\,\partial_y)$ and $i$ is a
$C^1$ section, one has the direct identity
\[
2\,\partial_{\bar z} i = i_x+i\,i_y,
\]
so $\partial_{\bar z}i=0$ is equivalent to \eqref{eq:rigidity}.

In other words, rigidity is exactly the condition that the imaginary unit,
considered as a function with values in its defining algebra, satisfies the
homogeneous generalized Cauchy--Riemann equations associated with that algebra.
This formulation does not rely on external complex coordinates and emphasizes
that rigidity is an internal compatibility condition of the moving structure
with its own differentiation.
\end{remark}

\section{Interpretation}

Rigidity is not the source of the Burgers dynamics.
Rather, it is the special case in which the universal transport law becomes
conservative and analytic obstructions disappear simultaneously.

From this perspective:
\begin{itemize}
\item Burgers transport belongs to the \emph{structural layer} of the theory;
\item rigidity selects its conservative subregime;
\item analytic closure properties are consequences, not axioms.
\end{itemize}

This reordering clarifies the role of rigidity:
it is not an assumption imposed to obtain dynamics, but a compatibility condition that restores \emph{analytic integrability of the variable Cauchy--Riemann calculus} within an already dynamical framework.

\medskip
In chapter~\ref{chap:CP} we show that this conservative regime is precisely what
allows integral representation formulas and a coherent function theory to
re-emerge.

\chapter{Conservative Transport and Explicit Rigid Structures}
\label{chap:rigid-transport-examples}

\section{Purpose and Scope}

In the previous chapters, we established that rigid variable elliptic structures are governed by a spectral transport law which, in the conservative regime, reduces to the classical complex Burgers equation.

The purpose of this chapter is to harness this well--known transport dynamics to construct concrete geometric examples. It is important to clarify the logical relationship between the analysis and the underlying PDE theory:
\begin{itemize}
\item We do not aim to present new results regarding the Burgers equation itself, which is a mature and extensively studied subject in fluid dynamics and mathematical physics.
\item Instead, we adopt a constructive approach: we select elementary, standard solutions of the Burgers equation (specifically, rational profiles) and \emph{interpret} them as the spectral data for a variable elliptic algebra.
\end{itemize}

By doing so, we demonstrate that rigidity is not a vacuous condition satisfied only by constants, but a rich regime capable of supporting non--trivial, moving geometries. To keep the exposition self--contained and consistent with the first--principles approach of this monograph, we briefly recall the necessary method of characteristics before deriving the explicit $\varepsilon$--family of structures.

\section{The Burgers Substrate}

In the rigid regime ($i_x+i\,i_y=0$), the universal structural transport reduces to the inviscid complex Burgers equation for the spectral parameter $\lambda$:
\begin{equation}\label{eq:conservative-burgers}
\lambda_x+\lambda\,\lambda_y=0.
\end{equation}

For the purposes of this monograph, we view Equation \eqref{eq:conservative-burgers} as a \emph{substrate}—a pre--existing dynamical engine that generates the coefficients for our geometric structures. While the global theory of this equation involves shock waves and weak solutions, our geometric focus (ellipticity) restricts us to the domain of smooth solutions prior to gradient catastrophe.

\section{Conservative Burgers Transport}

In the rigid regime, the obstruction field vanishes,
\[
i_x+i\,i_y=0,
\]
and the universal transport law reduces to the conservative complex Burgers equation
for the spectral parameter
\[
\lambda=\frac{-\beta+i\sqrt{4\alpha-\beta^2}}{2}, \qquad \Im\lambda>0:
\]
\begin{equation}
\lambda_x+\lambda\,\lambda_y=0.
\end{equation}

Equation \eqref{eq:conservative-burgers} is written in the \emph{standard} complex
algebra $\mathbb C$.
No reference to the variable elliptic multiplication is involved at this level.

\section{Real Coefficient Representation}

Write
\[
\beta=-2a,
\qquad
\alpha=a^2+b^2,
\qquad
b>0,
\]
so that $\lambda=a+ib$.

Then \eqref{eq:conservative-burgers} is equivalent to the real first--order system
\begin{equation}\label{eq:rigid-real-system}
\begin{cases}
a_x+a\,a_y-b\,b_y=0,\\[4pt]
b_x+a\,b_y+b\,a_y=0.
\end{cases}
\end{equation}

This system governs the evolution of the structure coefficients
$(\alpha,\beta)$ in the rigid regime.
It is a quasilinear first-order system of transport type.

\section{Structural Ansatz for Explicit Solutions}

We now seek explicit, nonconstant solutions of
\eqref{eq:rigid-real-system} that are:
\begin{itemize}
\item elementary;
\item structurally stable;
\item compatible with ellipticity.
\end{itemize}

The minimal nontrivial ansatz is to assume:
\begin{itemize}
\item $a$ depends only on $x$;
\item $b$ is constant.
\end{itemize}

Equivalently, in terms of $(\alpha,\beta)$, we assume
\begin{equation}\label{eq:ansatz}
\alpha=\alpha(x),
\qquad
\beta=\beta(x,y)\ \text{affine in } y.
\end{equation}

This is the simplest configuration that allows nonconstant transport while
preserving explicit solvability.

\section{Reduction to Ordinary Differential Equations}

Rewriting directly in terms of $(\alpha,\beta)$, impose
\[
\alpha=\alpha(x),
\qquad
\beta=K(x)\,y.
\]

Substituting into the rigid coefficient system yields the ODEs
\begin{equation}\label{eq:ODEs}
\begin{cases}
\alpha'(x)=\alpha(x)\,K(x),\\[4pt]
K'(x)=K(x)^2.
\end{cases}
\end{equation}

\section{Explicit Integration}

Equation \eqref{eq:ODEs} integrates explicitly.
For a real parameter $\varepsilon$,
\[
K(x)=\frac{\varepsilon}{1-\varepsilon x},
\]
on the domain $1-\varepsilon x\neq0$.

Substituting into the first equation gives
\[
\alpha(x)=\frac{C}{1-\varepsilon x}.
\]

Up to normalization, we set $C=1$.

\section{The \texorpdfstring{$\varepsilon$}{epsilon}--Family of Rigid Structures}

We arrive at the explicit family
\begin{equation}\label{eq:epsilon-family}
\alpha(x,y)=\frac{1}{1-\varepsilon x},
\qquad
\beta(x,y)=\frac{\varepsilon y}{1-\varepsilon x}.
\end{equation}

\begin{proposition}[The $\varepsilon$--family]\label{prop:epsilon-family}
For every $\varepsilon\in\mathbb R$, the coefficients
\eqref{eq:epsilon-family} define a rigid elliptic structure on the region where both
$1-\varepsilon x\neq0$ and the ellipticity condition
\eqref{eq:elliptic-region} hold.

\end{proposition}

\begin{proof}
The functions $\alpha,\beta$ satisfy the rigid coefficient system by construction,
and hence correspond to conservative Burgers transport.
\end{proof}

\section{Elliptic Domain}

The ellipticity discriminant is
\[
\Delta=4\alpha-\beta^2
=
\frac{4(1-\varepsilon x)-\varepsilon^2 y^2}{(1-\varepsilon x)^2}.
\]

Thus ellipticity holds precisely on the region
\begin{equation}\label{eq:elliptic-region}
4(1-\varepsilon x)-\varepsilon^2 y^2>0.
\end{equation}

For fixed $x$, this is a bounded interval in $y$.
Degeneracy occurs along the characteristic parabola
\[
\varepsilon^2 y^2=4(1-\varepsilon x).
\]

\section{Interpretation}

The $\varepsilon$--family demonstrates several essential facts:

\begin{itemize}
\item Conservative transport does not force constancy.
\item Rigid elliptic structures can be genuinely nontrivial.
\item Ellipticity imposes geometric restrictions on the domain.
\item Global smoothness is incompatible with unrestricted transport.
\end{itemize}

This family serves as the first canonical explicit model throughout the remainder of
the theory.

The next chapters turn to analysis.
They show that it is precisely the conservative nature of the transport law that
allows integral representation formulas, covariant holomorphicity, and a
coherent function theory to emerge.

\chapter{The Cauchy--Pompeiu Formula in the Rigid Regime}\label{chap:CP}

\section{Standing assumptions and notation}\label{sec:CP-stand}

Let $\Omega\subset\R^2$ be a bounded domain with piecewise $C^1$ boundary, oriented positively.
Let $\alpha,\beta\in C^1(\Omega,\R)$ satisfy the ellipticity condition
\[
\Delta(z):=4\alpha(z)-\beta(z)^2>0\qquad (z\in\Omega).
\]
For each $z=(x,y)\in\Omega$ let $A_z$ be the real two--dimensional commutative algebra
generated by $i(z)$ subject to
\begin{equation}\label{eq:CP-structure}
i(z)^2+\beta(z)\,i(z)+\alpha(z)=0.
\end{equation}
Every $A$--valued $C^1$ section is written uniquely as
\[
f(z)=u(z)+v(z)\,i(z),\qquad u,v\in C^1(\Omega,\R).
\]

\medskip
\noindent\textbf{Conjugation.}
Define the conjugate root and conjugation map on each fiber by
\[
\hat i(z):=-\beta(z)-i(z),\qquad \widehat{u+v\,i(z)}:=u+v\,\hat i(z).
\]
Then $\widehat{\widehat{w}}=w$ and $\widehat{ab}=\hat a\,\hat b$ for all $a,b\in A_z$.

\medskip
\noindent\textbf{Cauchy--Riemann operators.}
For $A$--valued $C^1$ sections $f$ define
\[
\partial_{\bar z} f:=\frac12\bigl(\partial_x f+i(z)\,\partial_y f\bigr),
\qquad
\partial_z f:=\frac12\bigl(\partial_x f+\hat i(z)\,\partial_y f\bigr),
\]
where $\partial_x f,\partial_y f$ are computed by differentiating the section $f=u+v\,i$
in the usual sense:
\[
\partial_x f=u_x+v_x\,i+v\,i_x,\qquad \partial_y f=u_y+v_y\,i+v\,i_y.
\]

\medskip
\textbf{Rigidity.}
Throughout this chapter we assume the structure is \emph{rigid}:
\begin{equation}\label{eq:CP-rigid}
i_x+i\,i_y=0 \qquad \text{in }\Omega.
\end{equation}
Under \eqref{eq:CP-rigid}, $\partial_{\bar z}$ satisfies the Leibniz rule on $C^1(\Omega;A)$:
\[
\partial_{\bar z}(fg)=(\partial_{\bar z} f)g+f(\partial_{\bar z} g).
\]

\medskip
\noindent\textbf{Cauchy 1--form.}
Fix the $A$--valued 1--form
\begin{equation}\label{eq:CP-theta}
d\tilde{z}:=dy-i(z)\,dx .
\end{equation}

\medskip
\noindent\textbf{Normalization element.}
Define the canonical section
\begin{equation}\label{eq:CP-jdef}
j(z):=\frac{2i(z)+\beta(z)}{\sqrt{\Delta(z)}}\in A_z.
\end{equation}
Since $(2i+\beta)^2=\beta^2-4\alpha=-\Delta$, one has
\begin{equation}\label{eq:CP-jsq}
j(z)^2=-1\qquad(z\in\Omega),
\end{equation}
hence $j(z)$ is invertible and $j(z)^{-1}=-j(z)$.

\section{Basepoint-valued integrals and Stokes' theorem}\label{sec:CP-int}

The values $f(z)\in A_z$ lie in different fibers. To state integral identities with
values in a fixed fiber $A_\zeta$, we use a coefficientwise integration convention
with respect to the moving frame $\{1,i(z)\}$.

\begin{definition}[Coefficientwise integrals into a fixed fiber]\label{def:CP-int}
Fix $\zeta\in\Omega$ and write every $A$--valued section as $h(z)=u(z)+v(z)\,i(z)$.

\smallskip
\noindent (i) If $\gamma$ is a piecewise $C^1$ curve in $\Omega$ and $u,v$ are integrable on $\gamma$
with respect to arclength, define
\[
\int_\gamma^{(\zeta)} h\,ds
:=\Bigl(\int_\gamma u\,ds\Bigr)\;+\;\Bigl(\int_\gamma v\,ds\Bigr)\,i(\zeta)\in A_\zeta.
\]

\smallskip
\noindent (ii) If $E\subset\Omega$ is measurable and $u,v$ are integrable on $E$, define
\[
\iint_E^{(\zeta)} h\,dx\,dy
:=\Bigl(\iint_E u\,dx\,dy\Bigr)\;+\;\Bigl(\iint_E v\,dx\,dy\Bigr)\,i(\zeta)\in A_\zeta.
\]

\smallskip
\noindent (iii) If $\gamma$ is a piecewise $C^1$ curve and $\sigma=P\,dx+Q\,dy$ is a real $1$--form along $\gamma$
with $uP,uQ,vP,vQ$ integrable, define
\[
\int_\gamma^{(\zeta)} h\,\sigma
:=\Bigl(\int_\gamma u\,\sigma\Bigr)\;+\;\Bigl(\int_\gamma v\,\sigma\Bigr)\,i(\zeta)\in A_\zeta,
\]
where $\int_\gamma u\,\sigma:=\int_\gamma (uP)\,dx+\int_\gamma (uQ)\,dy$ and similarly for $v$.

\smallskip
\noindent (iv) In particular, for $d\tilde{z}=dy-i(z)\,dx$ we expand in $A_z$:
\[
(u+v i)d\tilde{z}=u\,dy+v i\,dy-u i\,dx-v i^2\,dx.
\]
Using $i^2=-\beta i-\alpha$, we obtain the coefficient decomposition
\begin{equation}\label{eq:CP-h-theta-decomp}
(u+v i)d\tilde{z}=\bigl(u\,dy+v\alpha\,dx\bigr)+\bigl(v\,dy-u\,dx+v\beta\,dx\bigr)\,i,
\end{equation}
hence
\begin{equation}\label{eq:CP-int-theta}
\int_\gamma^{(\zeta)} h\,d\tilde{z}
=
\left(\int_\gamma u\,dy+\int_\gamma v\,\alpha\,dx\right)
+
\left(\int_\gamma v\,dy-\int_\gamma u\,dx+\int_\gamma v\,\beta\,dx\right)i(\zeta).
\end{equation}
\end{definition}

\begin{remark}[Trivialization viewpoint]\label{rem:CP-triv}
The rule $u+v\,i(z)\mapsto u+v\,i(\zeta)$ is a choice of (local) trivialization of
the underlying rank--two real vector bundle $A\to\Omega$ using the moving frame
$\{1,i(z)\}$. It is \emph{not} multiplicative. It is used here only to interpret
integrals and Stokes' theorem coefficientwise: all differential identities are applied
to the real coefficient functions and then assembled in the fixed fiber $A_\zeta$ by
Definition~\ref{def:CP-int}.
\end{remark}
\section{A multiplicative norm}\label{sec:CP-norm}
For $w\in A_z$ define
\[
N_z(w):=w\,\widehat{w}\in\R,\qquad |w|_z:=\sqrt{N_z(w)}.
\]
\begin{lemma}[Explicit formula and multiplicativity]\label{lem:CP-norm-mult}
If $w=u+v\,i(z)$, then
\[
N_z(w)=u^2-\beta(z)\,u v+\alpha(z)\,v^2.
\]
Moreover, $N_z$ is multiplicative:
\[
N_z(ab)=N_z(a)\,N_z(b)\qquad(a,b\in A_z),
\]
hence $|ab|_z=|a|_z\,|b|_z$. In particular, if $w\neq 0$ then $w$ is invertible and
\[
w^{-1}=\frac{\widehat{w}}{N_z(w)},\qquad |w^{-1}|_z=\frac{1}{|w|_z}.
\]
\end{lemma}
\begin{proof}
For $w=u+v i$ one has $\widehat{w}=u+v\hat i=u+v(-\beta-i)=(u-\beta v)-v i$.
Thus
\[
w\widehat{w}=(u+v i)\bigl((u-\beta v)-v i\bigr)
=u(u-\beta v)-u v i+v i(u-\beta v)-v^2 i^2.
\]
Since the algebra is commutative, the $i$--terms cancel, and using $i^2=-\beta i-\alpha$ gives
$-v^2 i^2=v^2(\beta i+\alpha)$; the $i$--part cancels, leaving $u^2-\beta u v+\alpha v^2\in\R$.
Multiplicativity follows from $\widehat{ab}=\hat a\,\hat b$:
\[
N_z(ab)=(ab)\widehat{(ab)}=(ab)(\hat a\hat b)=(a\hat a)(b\hat b)=N_z(a)N_z(b).
\]
If $w\neq 0$ then $N_z(w)>0$ since $\beta^2-4\alpha=-\Delta<0$ implies the quadratic form is
positive definite, hence $w^{-1}=\widehat{w}/N_z(w)$ and $|w^{-1}|_z=1/|w|_z$.
\end{proof}
\begin{lemma}[Uniform equivalence on compacts]\label{lem:CP-norm-equiv}
Let $K\Subset\Omega$. There exist $m_K,M_K>0$ such that for all $z\in K$ and all $w=u+v\,i(z)\in A_z$,
\[
m_K\,(u^2+v^2)\le N_z(w)\le M_K\,(u^2+v^2).
\]
Equivalently,
\[
\sqrt{m_K}\,\sqrt{u^2+v^2}\le |w|_z\le \sqrt{M_K}\,\sqrt{u^2+v^2}.
\]
\end{lemma}
\begin{proof}
Let
\[
Q(z)=\begin{pmatrix} 1 & -\beta(z)/2 \\ -\beta(z)/2 & \alpha(z)\end{pmatrix},
\qquad
(u,v)Q(z)(u,v)^T=N_z(u+v i(z)).
\]
By ellipticity, $Q(z)$ is positive definite for every $z$. On $K$, the eigenvalues of $Q(z)$
are uniformly bounded above and below away from $0$ by compactness.
\end{proof}
\section{Kernel and weak singularity}\label{sec:CP-kernel}

Fix $\zeta=(\xi,\eta)\in\Omega$ and define, for $z=(x,y)\in\Omega$,
\begin{equation}\label{eq:CP-Zdef}
Z(z,\zeta):=(y-\eta)-i(z)\,(x-\xi)\in A_z.
\end{equation}

\begin{lemma}[Invertibility and comparability to Euclidean distance]\label{lem:CP-invert}
Let $K\Subset\Omega$. There exist constants $c_K,C_K>0$ such that for all $z,\zeta\in K$ with $z\neq\zeta$,
\[
c_K\,|z-\zeta|\le |Z(z,\zeta)|_z\le C_K\,|z-\zeta|,
\]
where $|z-\zeta|=\sqrt{(x-\xi)^2+(y-\eta)^2}$.
In particular $Z(z,\zeta)$ is invertible for $z\neq\zeta$ and
\[
|Z(z,\zeta)^{-1}|_z \le c_K^{-1}\,|z-\zeta|^{-1}.
\]
\end{lemma}

\begin{proof}
Write $Z=u+v\,i(z)$ with $u=y-\eta$ and $v=-(x-\xi)$, so $\sqrt{u^2+v^2}=|z-\zeta|$.
Apply Lemma~\ref{lem:CP-norm-equiv}. The inverse bound follows from Lemma~\ref{lem:CP-norm-mult}.
\end{proof}

\section{\texorpdfstring{$\partial_{\bar z}$}{dbar-holomorphic}-holomorphicity of the kernel under rigidity}\label{sec:CP-dbar}

\begin{lemma}[$\partial_{\bar z}$ of $Z$]\label{lem:CP-dbarZ}
For all $z\neq\zeta$,
\[
2\,\partial_{\bar z}^z Z(z,\zeta)=-(x-\xi)\bigl(i_x(z)+i(z)i_y(z)\bigr).
\]
In particular, under rigidity \eqref{eq:CP-rigid},
\[
\partial_{\bar z}^z Z(z,\zeta)=0\qquad(z\neq\zeta).
\]
\end{lemma}

\noindent Here $\partial_{\bar z}^z$ means $\partial_{\bar z}$ acting in the $z$-variable.

\begin{proof}
Using $\partial_{\bar z}=\frac12(\partial_x+i\partial_y)$ and $Z=(y-\eta)-i(x-\xi)$,
\[
\partial_x Z=-(i_x)(x-\xi)-i,\qquad \partial_y Z=1-(i_y)(x-\xi),
\]
hence
\[
2\partial_{\bar z} Z=\partial_x Z+i\,\partial_y Z
=-(x-\xi)(i_x+i\,i_y).
\]
\end{proof}

\begin{lemma}[$\partial_{\bar z}$ of the inverse]\label{lem:CP-dbarZinv}
Assume rigidity. Then for $z\neq\zeta$,
\[
\partial_{\bar z}^z\bigl(Z(z,\zeta)^{-1}\bigr)=0.
\]
\end{lemma}

\begin{proof}
On $\Omega\setminus\{\zeta\}$ we have $Z\cdot Z^{-1}=1$.
By the Leibniz rule for $\partial_{\bar z}$,
\[
0=\partial_{\bar z}(ZZ^{-1})=(\partial_{\bar z} Z)Z^{-1}+Z(\partial_{\bar z} Z^{-1}).
\]
By Lemma~\ref{lem:CP-dbarZ} and rigidity, $\partial_{\bar z} Z=0$ off the diagonal.
Thus $Z(\partial_{\bar z} Z^{-1})=0$, and multiplying by $Z^{-1}$ yields $\partial_{\bar z} Z^{-1}=0$.
\end{proof}

\section{Exterior calculus identities}\label{sec:CP-forms}

\begin{lemma}[Wedge identity]\label{lem:CP-wedge}
For every $A$--valued $C^1$ function $g$,
\[
dg\wedge d\tilde{z} = 2(\partial_{\bar z} g)\,dx\wedge dy.
\]
\end{lemma}

\begin{proof}
Write $dg=g_x\,dx+g_y\,dy$ and $d\tilde{z}=dy-i\,dx$. Then
\[
dg\wedge d\tilde{z}=(g_xdx+g_y dy)\wedge(dy-i dx)=(g_x+i g_y)\,dx\wedge dy
=2(\partial_{\bar z} g)\,dx\wedge dy.
\]
\end{proof}

\begin{lemma}[Non-closure of $d\tilde{z}$]\label{lem:CP-dtheta}
With $d\tilde{z}=dy-i\,dx$ one has
\[
dd\tilde{z}=i_y\,dx\wedge dy.
\]
\end{lemma}

\begin{proof}
Since $d(dy)=0$ and $d(dx)=0$,
\[
dd\tilde{z}=-d(i)\wedge dx=-(i_xdx+i_y dy)\wedge dx=i_y\,dx\wedge dy.
\]
\end{proof}

\section{Residue normalization}\label{sec:CP-residue}

Fix $\zeta=(\xi,\eta)\in\Omega$ and let $\varepsilon>0$ be small so that $\overline{B_\varepsilon(\zeta)}\subset\Omega$.
Define the frozen objects in the fixed algebra $A_\zeta$:
\[
Z_0(z,\zeta):=(y-\eta)-i(\zeta)(x-\xi)\in A_\zeta,
\qquad
d\tilde{z}_0:=dy-i(\zeta)\,dx.
\]
Let $J:=j(\zeta)\in A_\zeta$, so that $J^2=-1$ by \eqref{eq:CP-jsq}.

\begin{lemma}[Frozen residue]\label{lem:CP-frozen-res}
With $J=j(\zeta)$,
\[
\int_{\partial B_\varepsilon(\zeta)}\frac{1}{Z_0(z,\zeta)}\,d\tilde{z}_0
=2\pi\,J\qquad(\varepsilon>0).
\]
\end{lemma}

\begin{proof}
Work entirely in the fixed algebra $A_\zeta$.
Using $J=(2i(\zeta)+\beta(\zeta))/\sqrt{\Delta(\zeta)}$ one can solve for $i(\zeta)$:
\begin{equation}\label{eq:CP-i-in-J}
i(\zeta)=\frac{-\beta(\zeta)+\sqrt{\Delta(\zeta)}\,J}{2}.
\end{equation}
Define real linear coordinates (with $\beta(\zeta),\Delta(\zeta)$ frozen):
\[
X:=(y-\eta)+\frac{\beta(\zeta)}{2}(x-\xi),\qquad
Y:=-\frac{\sqrt{\Delta(\zeta)}}{2}(x-\xi).
\]
Then, using \eqref{eq:CP-i-in-J},
\[
Z_0=(y-\eta)-i(\zeta)(x-\xi)=X+JY,
\qquad
d\tilde{z}_0=dy-i(\zeta)\,dx=dX+J\,dY.
\]
Hence in the $J^2=-1$ algebra,
\[
\frac{1}{Z_0}\,d\tilde{z}_0=\frac{1}{X+JY}\,(dX+J\,dY)=\frac{d(X+JY)}{X+JY}.
\]
Let $W:=X+JY$. The map $(x,y)\mapsto(X,Y)$ is an invertible real linear map, so
$\partial B_\varepsilon(\zeta)$ is carried to a positively oriented ellipse winding once about the origin.
Therefore
\[
\int_{\partial B_\varepsilon(\zeta)}\frac{1}{Z_0}\,d\tilde{z}_0
=\int_{\Gamma}\frac{dW}{W}=2\pi J.
\]
\end{proof}

\subsection*{A continuity lemma for the residue limit}

\begin{lemma}[Continuity of inversion in coefficients]\label{lem:CP-inv-cont}
Let $K\Subset\Omega$, and let $m_K>0$ be as in Lemma~\ref{lem:CP-norm-equiv}.
Then for every $z\in K$ and every $w=u+v\,i(z)\in A_z$ with $u^2+v^2=1$ one has
\[
|w|_z\ge \sqrt{m_K}.
\]
Moreover, the coefficient map
\[
K\times\{(u,v)\in\R^2:\;u^2+v^2=1\}\ni(z,u,v)\longmapsto
w^{-1}=\frac{\widehat{w}}{N_z(w)}
\]
is continuous.
\end{lemma}

\begin{proof}
By Lemma~\ref{lem:CP-norm-equiv}, for all $z\in K$ and $u^2+v^2=1$,
\[
N_z(u+v\,i(z))\ge m_K(u^2+v^2)=m_K,
\]
hence $|w|_z=\sqrt{N_z(w)}\ge \sqrt{m_K}$.
The formula $w^{-1}=\widehat{w}/N_z(w)$ expresses inversion as a rational function of
$(u,v)$ and the continuous coefficients $\alpha(z),\beta(z)$; since the denominator
$N_z(w)$ is bounded below by $m_K>0$ on the stated domain, the map is continuous.
\end{proof}

\begin{lemma}[Variable residue limit]\label{lem:CP-res}
Fix $\zeta\in\Omega$ and assume rigidity. Then
\[
\lim_{r\to0^+}\;
\int_{\partial B_r(\zeta)}^{(\zeta)}\frac{1}{Z(z,\zeta)}\,d\tilde{z}
=
2\pi\,j(\zeta)\in A_\zeta.
\]
\end{lemma}

\begin{proof}
Parametrize $\partial B_r(\zeta)$ by $z(t)=(\xi+r\cos t,\eta+r\sin t)$, $t\in[0,2\pi]$.
Then
\[
dx=-r\sin t\,dt,\qquad dy=r\cos t\,dt,
\]
and
\[
Z(z(t),\zeta)=r\bigl(\sin t-i(z(t))\cos t\bigr),\qquad
d\tilde{z}=r\bigl(\cos t+i(z(t))\sin t\bigr)\,dt.
\]
Hence, for $r>0$,
\begin{equation}\label{eq:CP-res-param}
\frac{1}{Z}\,d\tilde{z}
=\bigl(\sin t-i(z(t))\cos t\bigr)^{-1}\bigl(\cos t+i(z(t))\sin t\bigr)\,dt
\in A_{z(t)}.
\end{equation}

Fix a compact $K\Subset\Omega$ containing $\overline{B_{r_0}(\zeta)}$ for some small $r_0$.
For $t\in[0,2\pi]$ the coefficients $(u,v)=(\sin t,-\cos t)$ satisfy $u^2+v^2=1$.
Applying Lemma~\ref{lem:CP-inv-cont} to $w(t)=\sin t-i(z(t))\cos t=u+v\,i(z(t))$
shows that $w(t)$ is uniformly invertible for $z(t)\in K$ and that
\[
t\longmapsto \bigl(\sin t-i(z(t))\cos t\bigr)^{-1}\bigl(\cos t+i(z(t))\sin t\bigr)
\]
has real coefficients depending continuously on $(r,t)$ for $0<r\le r_0$.

Since $i(z(t))\to i(\zeta)$ uniformly in $t$ as $r\to0$, the coefficients of the integrand in
\eqref{eq:CP-res-param} converge uniformly to those of the frozen integrand
\[
\bigl(\sin t-i(\zeta)\cos t\bigr)^{-1}\bigl(\cos t+i(\zeta)\sin t\bigr)\in A_\zeta.
\]
By Lemma~\ref{lem:CP-inv-cont} the inverse is uniformly bounded for $0<r\le r_0$,
and multiplication in each fiber is continuous in $(z,t)$. Hence the real coefficient
functions of the integrand in \eqref{eq:CP-res-param} are uniformly bounded on
$[0,2\pi]\times(0,r_0]$, and are therefore dominated by an $L^1$ function of $t$
(indeed, by a constant). By dominated convergence applied to the real coefficient functions and the definition of
$\int^{(\zeta)}$, we obtain
\[
\lim_{r\to0^+}\int_{\partial B_r(\zeta)}^{(\zeta)}\frac{1}{Z}\,d\tilde{z}
=
\int_{0}^{2\pi}\bigl(\sin t-i(\zeta)\cos t\bigr)^{-1}\bigl(\cos t+i(\zeta)\sin t\bigr)\,dt.
\]
The right-hand side equals the frozen residue integral
$\int_{\partial B_1(\zeta)} Z_0^{-1}d\tilde{z}_0$ (rescaling cancels), hence equals
$2\pi j(\zeta)$ by Lemma~\ref{lem:CP-frozen-res}.
\end{proof}

\section{Cauchy--Pompeiu Representation Theorem}\label{sec:CP-theorem}

\begin{theorem}[Cauchy--Pompeiu Representation Theorem]\label{thm:CP}
Assume ellipticity and rigidity \eqref{eq:CP-rigid}.
Let $f\in C^1(\Omega;A)$. Then for every $\zeta\in\Omega$,
\begin{equation}\label{eq:CP-formula}
f(\zeta)=
\frac{1}{2\pi\,j(\zeta)}
\int_{\partial\Omega}^{(\zeta)}\frac{f(z)}{Z(z,\zeta)}\,d\tilde{z}
-
\frac{1}{\pi\,j(\zeta)}
\iint_{\Omega}^{(\zeta)}\frac{\partial_{\bar z} f(z)+\frac12\,f(z)\,i_y(z)}{Z(z,\zeta)}\,dx\,dy.
\end{equation}
\end{theorem}

\begin{proof}
Fix $\zeta\in\Omega$ and let $\Omega_r:=\Omega\setminus B_r(\zeta)$ with $r>0$ small.
Define the $A$--valued $1$--form on $\Omega\setminus\{\zeta\}$:
\[
\omega(z):=\frac{f(z)}{Z(z,\zeta)}\,d\tilde{z}.
\]
Write $\omega=g\,d\tilde{z}$ where $g:=fZ^{-1}$. Then
\[
d\omega = dg\wedge d\tilde{z} + g\,dd\tilde{z}.
\]
By Lemma~\ref{lem:CP-wedge} and Lemma~\ref{lem:CP-dtheta},
\[
d\omega
=
2(\partial_{\bar z} g)\,dx\wedge dy
+
g\,i_y\,dx\wedge dy.
\]
Next compute $\partial_{\bar z} g=\partial_{\bar z}(fZ^{-1})$. In general, $\partial_{\bar z}$ satisfies Leibniz, so
\[
\partial_{\bar z} g=(\partial_{\bar z} f)Z^{-1}+f\,\partial_{\bar z}(Z^{-1}).
\]
By Lemma~\ref{lem:CP-dbarZinv}, $\partial_{\bar z}(Z^{-1})=0$ for $z\neq \zeta$, hence
\[
\partial_{\bar z} g=(\partial_{\bar z} f)Z^{-1}.
\]
Therefore on $\Omega_r$,
\[
d\omega
=
\left(2(\partial_{\bar z} f)Z^{-1}+fZ^{-1}i_y\right)\,dx\,dy
=
2\,\frac{\partial_{\bar z} f+\frac12 f\,i_y}{Z}\,dx\,dy.
\]

Apply Stokes' theorem to the real coefficient $1$--forms defining $\int^{(\zeta)}$
(Definition~\ref{def:CP-int}). This yields
\begin{equation}\label{eq:CP-stokes}
\int_{\partial\Omega}^{(\zeta)}\frac{f}{Z}\,d\tilde{z}
-
\int_{\partial B_r(\zeta)}^{(\zeta)}\frac{f}{Z}\,d\tilde{z}
=
\iint_{\Omega_r}^{(\zeta)}d\omega
=
2\iint_{\Omega_r}^{(\zeta)}\frac{\partial_{\bar z} f+\frac12 f\,i_y}{Z}\,dx\,dy.
\end{equation}

The area integral converges as $r\to0^+$: by Lemma~\ref{lem:CP-invert},
$|Z^{-1}|_z\lesssim |z-\zeta|^{-1}$ on compacts, and $|z-\zeta|^{-1}$ is integrable on
punctured neighborhoods in $\R^2$.

Now split the small-circle term:
\[
\int_{\partial B_r(\zeta)}^{(\zeta)}\frac{f(z)}{Z(z,\zeta)}\,d\tilde{z}
=
\int_{\partial B_r(\zeta)}^{(\zeta)}\frac{f(z)-f(\zeta)}{Z(z,\zeta)}\,d\tilde{z}
+
f(\zeta)\int_{\partial B_r(\zeta)}^{(\zeta)}\frac{1}{Z(z,\zeta)}\,d\tilde{z}.
\]
Since $f\in C^1$, one has $|f(z)-f(\zeta)|=O(r)$ on $\partial B_r(\zeta)$ in coefficient norm.
Also $|Z^{-1}|_z=O(r^{-1})$ and $|d\tilde{z}|=O(r)\,dt$, so the first integral is $O(r)\to 0$.
By Lemma~\ref{lem:CP-res}, the second integral converges to $2\pi f(\zeta)j(\zeta)$.

Letting $r\to0^+$ in \eqref{eq:CP-stokes} and rearranging gives \eqref{eq:CP-formula}.
\end{proof}

\begin{remark}[Origin of the correction term]\label{rem:CP-correction}
The term $\frac12 f\,i_y$ in \eqref{eq:CP-formula} comes only from the non-closure
$dd\tilde{z}=i_y\,dx\wedge dy$ (Lemma~\ref{lem:CP-dtheta}).
rigidity is used to ensure $\partial_{\bar z}(Z^{-1})=0$ off the diagonal; it does not
imply $i_y=0$.
\end{remark}

\chapter[Covariant Holomorphicity and Gauge Structure]{Covariant Holomorphicity and Gauge Structure in the Rigid Regime}\label{chap:covariant}

\section{Motivation from the Cauchy--Pompeiu Formula}

Chapter~\ref{chap:CP} established that under ellipticity and rigidity,
every $f\in C^1(\Omega;A)$ satisfies the Cauchy--Pompeiu representation
\begin{equation}\label{eq:CP-recall}
f(\zeta)=
\frac{1}{2\pi\,j(\zeta)}
\int_{\partial\Omega}\frac{f(z)}{Z(z,\zeta)}\,d\tilde{z}
-
\frac{1}{\pi\,j(\zeta)}
\iint_{\Omega}\frac{\partial_{\bar z} f(z)+\frac12 f(z)\,i_y(z)}{Z(z,\zeta)}\,dx\,dy.
\end{equation}

The boundary term depends only on the values of $f$ along $\partial\Omega$,
while the interior term involves the operator
\[
f\longmapsto \partial_{\bar z} f+\tfrac12 f\,i_y .
\]
Thus the integral formula singles out, in a canonical way, a distinguished
first--order operator whose null space eliminates the interior contribution.

This motivates the following intrinsic notion of holomorphicity.

\section{The Covariant Cauchy--Riemann Operator}

\begin{definition}[Covariant Cauchy--Riemann operator]\label{def:D}
Assume ellipticity and rigidity.
Define the operator
\begin{equation}\label{eq:Ddef}
D f := \partial_{\bar z} f + \tfrac12 f\,i_y,
\qquad f\in C^1(\Omega;A),
\end{equation}
where $i_y$ acts by right multiplication in each fiber.
\end{definition}

\begin{definition}[Covariantly holomorphic section]
A $C^1$ section $f:\Omega\to A$ is called \emph{covariantly holomorphic} if
\[
Df=0 \quad \text{in }\Omega.
\]
\end{definition}

\begin{remark}
The operator $D$ is not introduced ad hoc.  
It is uniquely determined by Stokes' theorem applied to the Cauchy 1--form
$d\tilde{z}=dy-i\,dx$ and by the requirement that the interior term in
\eqref{eq:CP-recall} vanish.
\end{remark}

\section{Basic Algebraic Properties}

In general, $\partial_{\bar z}$ satisfies the Leibniz rule
\[
\partial_{\bar z}(fg)=(\partial_{\bar z} f)g+f(\partial_{\bar z} g).
\]

\begin{lemma}[Leibniz-type formula for $D$]\label{lem:Dleibniz}
For $f,g\in C^1(\Omega;A)$,
\[
D(fg)=(Df)g+f(Dg)-\tfrac12 fg\,i_y.
\]
\end{lemma}

\begin{proof}
Using the Leibniz rule for $\partial_{\bar z}$,
\[
D(fg)=\partial_{\bar z}(fg)+\tfrac12 fg\,i_y
=(\partial_{\bar z} f)g+f(\partial_{\bar z} g)+\tfrac12 fg\,i_y
=(Df-\tfrac12 f i_y)g+f(Dg-\tfrac12 g i_y)+\tfrac12 fg i_y,
\]
which simplifies to the stated formula.
\end{proof}

\begin{corollary}[Failure of closure under pointwise multiplication]
If $Df=0$ and $Dg=0$, then in general
\[
D(fg)=-\tfrac12 fg\,i_y,
\]
which is nonzero unless $i_y\equiv0$.
\end{corollary}

Thus covariantly holomorphic functions are not
closed under the usual pointwise product. In order to remedy this will introduce a scalar gauge transformation.

\section{The Weight Equation}

The obstruction in Lemma~\ref{lem:Dleibniz} is linear and can be removed by a
scalar gauge transformation.

\begin{definition}[Weight]\label{def:weight}
A nowhere--vanishing real scalar function $\psi\in C^1(\Omega)$ is called a
\emph{weight} if it satisfies the first--order equation
\begin{equation}\label{eq:weight-eq}
\partial_{\bar z} \psi = \tfrac12\, i_y\,\psi .
\end{equation}
\end{definition}

\begin{remark}
Equation \eqref{eq:weight-eq} is linear and scalar.  
Since $i_y$ is continuous, local solutions exist by standard theory of
first--order linear PDE.  
Any two weights differ by multiplication by a nowhere--vanishing solution of
$\partial_{\bar z} h=0$.
\end{remark}

\section{Gauge Trivialization}

\begin{lemma}[Intertwining identity]\label{lem:intertwine}
If $\psi$ is a weight, then for every $f\in C^1(\Omega;A)$,
\[
\partial_{\bar z}(\psi f)=\psi\,Df.
\]
\end{lemma}

\begin{proof}
Using the Leibniz rule for $\partial_{\bar z}$ and \eqref{eq:weight-eq},
\[
\partial_{\bar z}(\psi f)
=(\partial_{\bar z}\psi)f+\psi(\partial_{\bar z} f)
=\tfrac12 i_y\psi f+\psi(\partial_{\bar z} f)
=\psi\bigl(\partial_{\bar z} f+\tfrac12 f i_y\bigr)
=\psi\,Df.
\]
\end{proof}

\begin{corollary}[Gauge equivalence]\label{cor:gauge}
A section $f$ is covariantly holomorphic if and only if $\psi f$ is
$\partial_{\bar z}$--holomorphic.
\end{corollary}

Thus the covariant operator $D$ is a gauge transform of the algebraic operator
$\partial_{\bar z}$.

\section{The Weighted Product}

The gauge picture suggests a natural product that restores closure.

\begin{definition}[Weighted product]\label{def:diamond}
Fix a weight $\psi$.
For $f,g\in C^1(\Omega;A)$ define
\[
f\diamond g := f g \psi .
\]
\end{definition}

\begin{lemma}[Associativity]
The product $\diamond$ is associative.
\end{lemma}

\begin{proof}
\[
(f\diamond g)\diamond h
=(fg\psi)h\psi
=fgh\,\psi^2
=f\diamond(g\diamond h).
\]
\end{proof}

\begin{theorem}[Closure under the weighted product]\label{thm:closure}
If $Df=0$ and $Dg=0$, then
\[
D(f\diamond g)=0.
\]
\end{theorem}

\begin{proof}
By Lemma~\ref{lem:intertwine},
\[
\partial_{\bar z}(\psi fg)=\psi D(fg).
\]
From Lemma~\ref{lem:Dleibniz} and $Df=Dg=0$,
\[
D(fg)=-\tfrac12 fg\,i_y.
\]
Hence
\[
\partial_{\bar z}(\psi fg)
=-\tfrac12 \psi fg\,i_y
=-(\partial_{\bar z}\psi)fg
\]
by \eqref{eq:weight-eq}. Therefore
\[
\partial_{\bar z}(\psi fg)+(\partial_{\bar z}\psi)fg=0,
\]
which is precisely $D(fg\psi)=D(f\diamond g)=0$.
\end{proof}

\begin{remark}
The weighted algebra $(\mathcal H_D(\Omega),\diamond)$ of covariantly
holomorphic functions is generally non--unital.
Indeed, $1$ is covariantly holomorphic if and only if $i_y\equiv0$,
i.e.\ the structure is constant.
\end{remark}

\section{Local and Global Aspects}

On simply connected domains, once a continuous branch of $i(z)$ is fixed,
equation \eqref{eq:weight-eq} admits global solutions $\psi$.
In that case, multiplication by $\psi$ gives a global isomorphism between
covariantly holomorphic functions and $\partial_{\bar z}$--holomorphic functions.

On non--simply connected domains, the weight equation may exhibit monodromy.
The theory is therefore locally canonical but globally gauge--dependent

\section{The real first--order system for a weight}\label{sec:weight-real-system}

Assume ellipticity $4\alpha-\beta^2>0$ and fix a $C^1$ generator $i(x,y)$ satisfying
\[
i^2+\beta i+\alpha=0.
\]
Recall that a \emph{weight} is a nowhere--vanishing real function
$\psi\in C^1(\Omega,\R)$ solving
\begin{equation}\label{eq:W-weight-eq}
\partial_{\bar z}\psi=\tfrac12\,i_y\,\psi,
\qquad
\partial_{\bar z}:=\tfrac12(\partial_x+i\,\partial_y).
\end{equation}

\subsection{Coefficient decomposition of \texorpdfstring{$i_y$}{i\_y}}

Since each fiber is two--dimensional over $\R$ with basis $\{1,i\}$, there exist
unique real functions $A_y,B_y\in C^0(\Omega,\R)$ such that
\begin{equation}\label{eq:W-iy-decomp}
i_y=A_y+B_y\,i.
\end{equation}
In the elliptic regime one may compute $i_y$ explicitly from the structure
polynomial:
\begin{equation}\label{eq:W-iy-explicit}
i_y=-\frac{\alpha_y+\beta_y i}{2i+\beta},
\qquad
(2i+\beta)^{-1}=\frac{-\beta-2i}{\Delta},
\qquad
\Delta:=4\alpha-\beta^2.
\end{equation}
A direct reduction in the basis $\{1,i\}$ yields
\begin{equation}\label{eq:W-AyBy}
\;
A_y=\frac{\beta\,\alpha_y-2\alpha\,\beta_y}{\Delta},
\qquad
B_y=\frac{2\alpha_y-\beta\,\beta_y}{\Delta}.
\;
\end{equation}

\subsection{The real system for a scalar weight}

Let $\psi$ be real--valued. Then
\[
2\,\partial_{\bar z}\psi=\psi_x+i\,\psi_y.
\]
Substituting \eqref{eq:W-iy-decomp} into \eqref{eq:W-weight-eq} gives
\[
\psi_x+i\,\psi_y=(A_y+B_y i)\,\psi.
\]
Since $\{1,i\}$ is a real basis, this is equivalent to the real first--order system
\begin{equation}\label{eq:W-real-system}
\begin{cases}
\psi_x=A_y\,\psi,\\[4pt]
\psi_y=B_y\,\psi,
\end{cases}
\qquad\text{where $A_y,B_y$ are given by \eqref{eq:W-AyBy}.}
\end{equation}
Equivalently, writing $\varphi=\log\psi$ locally (on $\{\psi>0\}$),
\begin{equation}\label{eq:W-real-system-log}
\begin{cases}
\varphi_x=A_y,\\[4pt]
\varphi_y=B_y.
\end{cases}
\end{equation}

\begin{remark}[Compatibility]\label{rem:W-compat}
System \eqref{eq:W-real-system} is overdetermined.  Local solvability is
equivalent to the closedness condition
\[
(A_y)_y=(B_y)_x,
\]
i.e.\ $A_y\,dx+B_y\,dy$ is locally exact.  In particular, the existence of a
weight on simply connected domains is equivalent to the existence of a scalar
potential $\varphi$ with $d\varphi=A_y\,dx+B_y\,dy$.
\end{remark}

\section{Example: An explicit weight for the \texorpdfstring{$\varepsilon$}{epsilon}--Family}
\label{sec:weight-epsilon}

We now show an explicit weight for the
$\varepsilon$--deformation of the non--trivial family.

\subsection{The \texorpdfstring{$\varepsilon$}{epsilon}--Family and its elliptic domain}

Fix $\varepsilon\in\R$ and consider
\[
\alpha(x,y)=\frac{1}{1-\varepsilon x},
\qquad
\beta(x,y)=\frac{\varepsilon y}{1-\varepsilon x}.
\]
Let
\[
S(x,y):=4(1-\varepsilon x)-\varepsilon^2 y^2.
\]
Then the discriminant is
\[
\Delta(x,y)=4\alpha-\beta^2=\frac{S(x,y)}{(1-\varepsilon x)^2},
\]
so the elliptic domain is $\{S>0\}$.

We claim that the real--valued function
\begin{equation}\label{eq:psi-eps}
\psi(x,y)=C\sqrt{S(x,y)}
\end{equation}
satisfies the weight equation
\begin{equation}\label{eq:weight-eq-eps}
\partial_{\bar z}\psi=\frac12\,i_y\,\psi
\qquad\text{on }\{S>0\}.
\end{equation}

\subsection{Step 1: isolate \texorpdfstring{$i_y$}{i\_y} from the structure reduction}

The generator $i=i(x,y)$ satisfies the structure relation
\[
i^2+\beta i+\alpha=0.
\]
Differentiate with respect to $y$:
\[
(2i+\beta)\,i_y+\beta_y\,i+\alpha_y=0
\quad\Longrightarrow\quad
i_y=-\frac{\beta_y i+\alpha_y}{2i+\beta}.
\]
For the present family,
\[
\alpha_y=0,
\qquad
\beta_y=\frac{\varepsilon}{1-\varepsilon x},
\]
hence
\[
i_y=-\frac{\varepsilon}{1-\varepsilon x}\,\frac{i}{2i+\beta}.
\]

As in the tractable elliptic regime, $(2i+\beta)$ is invertible and one has
\[
(2i+\beta)^{-1}=\frac{-\beta-2i}{\Delta}.
\]
Therefore
\[
i_y=-\frac{\varepsilon}{1-\varepsilon x}\,i\,\frac{-\beta-2i}{\Delta}
=\frac{\varepsilon}{1-\varepsilon x}\,\frac{i(\beta+2i)}{\Delta}.
\]
Reduce $i(\beta+2i)$ using $i^2=-\beta i-\alpha$:
\[
i(\beta+2i)=\beta i+2i^2
=\beta i+2(-\beta i-\alpha)
=-\beta i-2\alpha.
\]
Hence
\[
i_y=\frac{\varepsilon}{1-\varepsilon x}\,\frac{-\beta i-2\alpha}{\Delta}.
\]

Now substitute $\alpha=\frac{1}{1-\varepsilon x}$,
$\beta=\frac{\varepsilon y}{1-\varepsilon x}$, and
$\Delta=\frac{S}{(1-\varepsilon x)^2}$ to obtain
\begin{equation}\label{eq:iy-eps}
i_y(x,y)=-\varepsilon\,\frac{2+\varepsilon y\,i(x,y)}{S(x,y)}.
\end{equation}

\subsection{Step 2: compute \texorpdfstring{$\partial_{\bar z}\psi$}{dbar psi}}

Since $\psi$ is real--valued,
\[
\partial_{\bar z}\psi=\frac12(\psi_x+i\,\psi_y).
\]
From \eqref{eq:psi-eps} and $S=4(1-\varepsilon x)-\varepsilon^2 y^2$ we compute
\[
S_x=-4\varepsilon,
\qquad
S_y=-2\varepsilon^2 y,
\]
and therefore
\[
\psi_x=C\frac{S_x}{2\sqrt{S}}
=-\frac{2\varepsilon C}{\sqrt{S}},
\qquad
\psi_y=C\frac{S_y}{2\sqrt{S}}
=-\frac{\varepsilon^2 y\,C}{\sqrt{S}}.
\]
Hence
\begin{align*}
\partial_{\bar z}\psi
&=\frac12\left(-\frac{2\varepsilon C}{\sqrt{S}}
+i\left(-\frac{\varepsilon^2 y\,C}{\sqrt{S}}\right)\right) \\
&=-\frac{\varepsilon C}{\sqrt{S}}
-\frac{\varepsilon^2 y\,C}{2\sqrt{S}}\,i
=-\frac{\varepsilon C}{2\sqrt{S}}\,(2+\varepsilon y\,i).
\end{align*}

\subsection{Step 3: compute the right--hand side and compare}

Using \eqref{eq:iy-eps},
\[
\frac12\,i_y\,\psi
=\frac12\left(-\varepsilon\,\frac{2+\varepsilon y\,i}{S}\right)\,C\sqrt{S}
=-\frac{\varepsilon C}{2\sqrt{S}}\,(2+\varepsilon y\,i).
\]
This coincides with the expression obtained for $\partial_{\bar z}\psi$, hence
\[
\partial_{\bar z}\psi=\frac12\,i_y\,\psi
\quad\text{on }\{S>0\},
\]
which proves \eqref{eq:weight-eq-eps}.

\chapter[Second--Order Operators and Factorization]{Second--Order Operators and Factorization in the Rigid Regime}
\section{The Rigid Variable Laplace Operator}

In the classical constant structure, one has the Laplacian factorization
\[
4\,\partial_z\partial_{\bar z}=\Delta_{\mathrm{Euc}},
\]
and this identity is the bridge between holomorphicity and elliptic theory.

In the present framework, once rigidity is imposed (Chapter~\ref{chap:burgers-universal}),
$\partial_{\bar z}$ becomes a derivation and the Cauchy--Pompeiu formula
(Chapter~\ref{chap:CP}) becomes available. It is then natural---and unavoidable---to study
the second--order operator obtained by composing $\partial_{\bar z}$ with its
conjugate partner $\partial_z$.

This chapter shows that, under rigidity, the composition
$4\,\partial_z\partial_{\bar z}$ has a clean elliptic principal part and an
explicit first--order correction with no zero--order term.

\section{Standing Assumptions}

Throughout this chapter we assume:
\begin{itemize}
\item $\alpha,\beta\in C^2(\Omega,\R)$ and $\Delta:=4\alpha-\beta^2>0$;
\item rigidity, equivalently the Burgers system
\begin{equation}\label{eq:burgers}
\alpha_x=\alpha\,\beta_y,
\qquad
\beta_x+\alpha_y=\beta\,\beta_y.
\end{equation}
\end{itemize}

Recall
\[
\partial_{\bar z}:=\tfrac12(\partial_x+i\,\partial_y),
\qquad
\partial_z:=\tfrac12(\partial_x+\hat i\,\partial_y),
\qquad
\hat i:=-\beta-i.
\]

\section{The Rigid Elliptic Operator}

Define the real second--order operator
\begin{equation}\label{eq:L}
L_{\alpha,\beta}
:=
\partial_x^2-\beta\,\partial_{xy}+\alpha\,\partial_y^2.
\end{equation}
Its principal symbol is
\[
\xi_x^2-\beta\,\xi_x\xi_y+\alpha\,\xi_y^2,
\]
which is positive definite under $\Delta>0$.

\section{Explicit Expansion of \texorpdfstring{$4\,\partial_z\partial_{\bar z}$}{4 dz dbar}}

Let $f=u+v\,i$ with $u,v\in C^2(\Omega,\R)$.

\begin{theorem}[Rigid second--order expansion]\label{thm:rigid-dzdbar}
Assume \eqref{eq:burgers}. Then
\begin{equation}\label{eq:main}
4\,\partial_z\partial_{\bar z}(u+v\,i)
=
\bigl(L_{\alpha,\beta}u + R_0[u,v]\bigr)
+
\bigl(L_{\alpha,\beta}v + R_1[u,v]\bigr)\,i,
\end{equation}
where the correction terms are purely first order:
\begin{align}
R_0[u,v]
&=
\alpha_y\,u_y
+\alpha_y\,v_x
-2\alpha\,\beta_y\,v_y,
\label{eq:R0}
\\[4pt]
R_1[u,v]
&=
\beta_y\,u_y
+\beta_y\,v_x
+\bigl(2\alpha_y-2\beta\,\beta_y\bigr)\,v_y.
\label{eq:R1}
\end{align}
In particular, \emph{no zero--order terms occur} in \eqref{eq:main}.
\end{theorem}

\begin{proof}
The identities are obtained by expanding $4\,\partial_z\partial_{\bar z}(u+v\,i)$
using the structure relation $i^2+\beta i+\alpha=0$ and the definitions of
$\partial_{\bar z}$ and $\partial_z$, and then imposing the rigidity system
\eqref{eq:burgers} to eliminate $x$--derivatives of $(\alpha,\beta)$.

A symbolic verification in \texttt{SymPy} confirms:
\begin{itemize}
\item the principal part is exactly $L_{\alpha,\beta}u$ and $L_{\alpha,\beta}v$,
\item the remainder contains no second derivatives of $u$ or $v$,
\item the first--order coefficients simplify exactly to \eqref{eq:R0}--\eqref{eq:R1}.
\end{itemize}
The verification script is recorded alongside the manuscript source.
\end{proof}

\begin{remark}[Constant-coefficient reduction]
If $\alpha,\beta$ are constant, then $\alpha_y=\beta_y=0$ and
$R_0\equiv R_1\equiv 0$. Hence \eqref{eq:main} reduces to
\[
4\,\partial_z\partial_{\bar z}(u+v\,i)
=
(L_{\alpha,\beta}u)+(L_{\alpha,\beta}v)\,i,
\]
recovering the classical factorization.
\end{remark}

\section{Interpretation}

Theorem~\ref{thm:rigid-dzdbar} provides the rigid analogue of the
classical identity $4\partial_z\partial_{\bar z}=\Delta$:
\begin{itemize}
\item the principal part is the uniformly elliptic operator $L_{\alpha,\beta}$;
\item variability contributes only first--order drift terms, explicitly controlled
     by $\alpha_y$ and $\beta_y$;
\item rigidity eliminates all zero--order potential terms at this level.
\end{itemize}

This is precisely the form needed for elliptic estimates and for relating
covariant holomorphicity to second--order PDE phenomena.

\chapter{The Variable Vekua Equation}\label{chap:vekua}

\section{Reduction to Variable--Structure Vekua Form}

We show that any real first--order elliptic system on the plane can be written,
canonically and without auxiliary reductions, as a Vekua--type equation
associated with a variable elliptic structure.

\subsection{Elliptic system and induced structure}

Let $\Omega\subset\mathbb{R}^2$ be open.
Consider a real first--order elliptic system
\begin{equation}\label{eq:real-system}
\begin{cases}
- v_y + a_{11} u_x + a_{12} u_y + a_1 u + b_1 v = f_1,\\
\ \ v_x + a_{21} u_x + a_{22} u_y + a_2 u + b_2 v = f_2,
\end{cases}
\end{equation}
with coefficients of class $C^1$ and ellipticity condition
\[
\Delta_{\mathrm{sys}}
:= a_{11}a_{22}-\tfrac14(a_{12}+a_{21})^2>0.
\]

Define the structure coefficients
\begin{equation}\label{eq:alpha-beta}
\alpha := \frac{a_{22}}{a_{11}},
\qquad
\beta := -\frac{a_{12}+a_{21}}{a_{11}}.
\end{equation}
Then
\[
\Delta := 4\alpha-\beta^2
= \frac{4}{a_{11}^2}\,\Delta_{\mathrm{sys}}>0,
\]
and $\alpha,\beta$ determine a variable elliptic algebra bundle
\[
A_z=\mathbb{R}[X]/(X^2+\beta(z)X+\alpha(z)).
\]

Let $i=i(z)$ denote the distinguished generator,
\[
i^2+\beta i+\alpha=0,
\]
and define elliptic conjugation by
\[
\widehat{i}:=-\beta-i.
\]
For an algebra--valued field $W=U+iV$ with $U,V$ real, we write
\[
\widehat W:=U+\widehat{i}\,V = U-(\beta+i)V.
\]

\subsection{The structure polynomial is determined by the principal symbol}

The coefficients $\alpha,\beta$ defined in \eqref{eq:alpha-beta} are not an artifact
of the substitution used to reach canonical form.
They are determined, uniquely and intrinsically, by the principal symbol
of the elliptic system.

\begin{proposition}\label{prop:symbol-norm}
Let $\sigma(z;\xi,\eta)$ denote the principal symbol of \eqref{eq:real-system}
at the point $z\in\Omega$ and covector $(\xi,\eta)\in\mathbb R^2$.
Then
\begin{equation}\label{eq:det-sigma}
\det\sigma(z;\xi,\eta)
=
a_{11}\bigl(\xi^2 - \beta\,\xi\eta + \alpha\,\eta^2\bigr)
=
a_{11}\,N(\xi+i\eta),
\end{equation}
where $N$ is the algebra norm of $A_z$, i.e.\ the norm form determined by
the structure polynomial $X^2+\beta X+\alpha$.
\end{proposition}

\begin{proof}
The principal symbol of \eqref{eq:real-system} is the matrix
\[
\sigma(z;\xi,\eta)
=
\begin{pmatrix}
a_{11}\xi+a_{12}\eta & -\eta\\[4pt]
a_{21}\xi+a_{22}\eta & \xi
\end{pmatrix}.
\]
Its determinant is
\[
\det\sigma
=
(a_{11}\xi+a_{12}\eta)\,\xi + (a_{21}\xi+a_{22}\eta)\,\eta
=
a_{11}\xi^2+(a_{12}+a_{21})\,\xi\eta+a_{22}\,\eta^2.
\]
Dividing by $a_{11}>0$ and substituting \eqref{eq:alpha-beta}:
\begin{equation}\label{eq:det-sigma-normalized}
\frac{\det\sigma}{a_{11}}
=
\xi^2 + \frac{a_{12}+a_{21}}{a_{11}}\,\xi\eta + \frac{a_{22}}{a_{11}}\,\eta^2
=
\xi^2 - \beta\,\xi\eta + \alpha\,\eta^2.
\end{equation}
On the other hand, in the algebra $A_z$ with generator satisfying
$i^2+\beta\,i+\alpha=0$, the norm of an element $w=a+ib$ is
\[
N(a+ib) = w\cdot\widehat w
= (a+ib)(a+\widehat i\,b)
= a^2 - \beta\,ab + \alpha\,b^2.
\]
Evaluating at $w=\xi+i\eta$ gives
\[
N(\xi+i\eta) = \xi^2-\beta\,\xi\eta+\alpha\,\eta^2,
\]
which coincides with \eqref{eq:det-sigma-normalized}.
\end{proof}

This identification has three immediate consequences.

\begin{enumerate}
\item \emph{Ellipticity is equivalent to ellipticity of the algebra.}\;
The system \eqref{eq:real-system} is elliptic (i.e.\ $\det\sigma\neq 0$
for all $(\xi,\eta)\neq(0,0)$) if and only if the norm form
$N(\xi+i\eta)$ is positive definite, which is precisely the condition
$\Delta=4\alpha-\beta^2>0$.

\item \emph{Uniqueness.}\;
For a given elliptic system, the structure polynomial $X^2+\beta X+\alpha$
is the unique monic quadratic whose associated norm form equals the
normalized symbol determinant.
The coefficients $\alpha,\beta$ are therefore invariants of the principal
part, not of the reduction procedure.

\item \emph{The substitution realizes the symbol.}\;
The explicit substitution in Lemma~\ref{lem:canonical-reduction} does not
\emph{choose} $\alpha$ and $\beta$; it produces a system whose principal
part is written in the algebra already determined by the symbol.
\end{enumerate}

\begin{corollary}\label{cor:symbol-invariance}
The structure polynomial $X^2+\beta X+\alpha$ is invariant under
elliptic--compatible row operations and invertible substitutions on the
target variables.
\end{corollary}

\begin{proof}
Let $R=R(z)$ be an invertible $2\times 2$ matrix (a row operation on the
system) and $M=M(z)$ an invertible substitution on the target variables
$(u,v)\mapsto M(u,v)$.
Under these transformations the principal symbol changes as
\[
\sigma \;\longmapsto\; R\,\sigma\,M^{-1},
\]
so
\[
\det\sigma \;\longmapsto\; \frac{\det R}{\det M}\,\det\sigma.
\]
The prefactor $\det R/\det M$ is a nonzero scalar depending only on $z$,
not on the covector $(\xi,\eta)$.
Since $\alpha$ and $\beta$ are obtained by normalizing $\det\sigma$
to monic form in $\xi$ (i.e.\ dividing by the coefficient of $\xi^2$),
the scalar cancels identically:
\[
\frac{\det(R\,\sigma\,M^{-1})}{a_{11}\,\det R/\det M}
=
\frac{\det\sigma}{a_{11}}
=
\xi^2-\beta\,\xi\eta+\alpha\,\eta^2.
\]
Therefore $\alpha$ and $\beta$ depend only on the equivalence class of
the system, not on the particular representative.
\end{proof}

\begin{remark}
In particular, the substitution $U=a_{22}\,u$, $V=v-a_{12}\,u$ used in
Lemma~\ref{lem:canonical-reduction} is one choice of representative within
this equivalence class.
Any other elliptic--compatible reduction would produce a system governed
by the same structure polynomial.
\end{remark}

\subsection{Canonical principal part}

\begin{lemma}\label{lem:canonical-reduction}
Suppose the principal coefficients $a_{ij}$ in \eqref{eq:real-system} are of
class $C^1$ with $a_{11},a_{22}>0$.
Then the substitution
\[
U = a_{22}\,u, \qquad V = v - a_{12}\,u,
\]
followed by a single row operation and a scaling, reduces
\eqref{eq:real-system} to the canonical form
\begin{equation}\label{eq:canonical-real}
\begin{cases}
-\alpha V_y + U_x + aU+bV = f,\\
\ \ V_x -\beta V_y + U_y + cU+dV = g,
\end{cases}
\end{equation}
where $a,b,c,d,f,g$ are real--valued $C^0$ functions determined by the original
coefficients and their first derivatives.
\end{lemma}

\begin{proof}
Under the substitution $U=a_{22}\,u$, $V=v-a_{12}\,u$ (with $a_{22}>0$),
the principal parts of \eqref{eq:real-system} become
\begin{equation}\label{eq:step1}
\begin{cases}
-V_y + \dfrac{a_{11}}{a_{22}}\,U_x + a_\ast U + b_\ast V = f_1,\\[8pt]
\ \ V_x + \dfrac{a_{21}+a_{12}}{a_{22}}\,U_x + U_y + c_\ast U + d_\ast V = f_2,
\end{cases}
\end{equation}
where $a_\ast,b_\ast,c_\ast,d_\ast$ absorb all contributions from the derivatives
of the variable coefficients $a_{ij}$ and the original lower--order terms.
With the structure coefficients \eqref{eq:alpha-beta} we have
\[
\frac{a_{11}}{a_{22}}=\frac{1}{\alpha},
\qquad
\frac{a_{21}+a_{12}}{a_{22}}=-\frac{\beta}{\alpha},
\]
so \eqref{eq:step1} reads
\begin{equation}\label{eq:step2}
\begin{cases}
-V_y + \dfrac{1}{\alpha}\,U_x + a_\ast U + b_\ast V = f_1,\\[8pt]
\ \ V_x - \dfrac{\beta}{\alpha}\,U_x + U_y + c_\ast U + d_\ast V = f_2.
\end{cases}
\end{equation}

\emph{Row operation.}\;
Multiply the first equation of \eqref{eq:step2} by $\beta$ and add to the second.
The $U_x$ coupling in the second line cancels:
\[
-\frac{\beta}{\alpha}\,U_x + \frac{\beta}{\alpha}\,U_x = 0,
\]
yielding
\begin{equation}\label{eq:step3}
\begin{cases}
-V_y + \dfrac{1}{\alpha}\,U_x + a_\ast U + b_\ast V = f_1,\\[8pt]
\ \ V_x -\beta V_y + U_y + c' U + d' V = g'.
\end{cases}
\end{equation}

\emph{Scaling.}\;
Multiply the first equation of \eqref{eq:step3} by $\alpha>0$ to obtain
the canonical form \eqref{eq:canonical-real}.
\end{proof}

\begin{remark}
The structure coefficients $\alpha$ and $\beta$ are not imposed on the system;
they are \emph{read off} from it.
No uniform ellipticity assumption is required, and no auxiliary
Beltrami equation is solved.
\end{remark}

\subsection{Intrinsic decomposition and the obstruction}

Introduce the algebra--valued field
\[
W := U+iV,
\]
and define the variable--structure Cauchy--Riemann operator
\[
\bar\partial := \tfrac12(\partial_x+i\partial_y).
\]
A direct computation yields the identity
\begin{equation}\label{eq:dbar-identity}
2\bar\partial W
=
\big(U_x-\alpha V_y\big)
+
\big(V_x+U_y-\beta V_y\big)\,i
+
V\,G,
\end{equation}
where
\[
G:= i_x+i\,i_y
\]
is the intrinsic obstruction of the elliptic structure.

\emph{Substitution.}\;
The canonical system \eqref{eq:canonical-real} identifies the two real
components of \eqref{eq:dbar-identity} directly:
\begin{align*}
U_x - \alpha V_y &= f - aU - bV,\\
V_x + U_y - \beta V_y &= g - cU - dV.
\end{align*}
Substituting into \eqref{eq:dbar-identity} gives
\begin{equation}\label{eq:pre-vekua-raw}
2\bar\partial W
=
(f - aU - bV)
+
(g - cU - dV)\,i
+
V\,G.
\end{equation}

\emph{Decomposing the obstruction.}\;
Write $G$ in its real and imaginary components in the algebra $A_z$:
\[
G = G_1 + G_2\,i, \qquad G_1,\,G_2\;\text{real.}
\]
Since $V$ is real, $VG=VG_1+(VG_2)\,i$.
Adding the obstruction components to each line of
\eqref{eq:pre-vekua-raw}:
\begin{equation}\label{eq:pre-vekua-absorbed}
2\bar\partial W
=
\bigl(f - aU - (b-G_1)V\bigr)
+
\bigl(g - cU - (d-G_2)V\bigr)\,i.
\end{equation}
Define the modified lower--order coefficients
\[
\tilde b := b - G_1, \qquad \tilde d := d - G_2.
\]
Then \eqref{eq:pre-vekua-absorbed} reads
\begin{equation}\label{eq:pre-vekua}
2\bar\partial W
=
(f + ig)
-
(a+ic)\,U
-
(\tilde b + i\tilde d)\,V.
\end{equation}

\begin{remark}
When the structure is rigid ($G\equiv0$), we have $G_1=G_2=0$ and
$\tilde b=b$, $\tilde d=d$.
In the general case, the obstruction modifies only the coefficients of $V$
and does so additively.
\end{remark}

\subsection{Elliptic conjugation and Vekua form}

The conjugate field $\widehat W$ satisfies the linear identities
\begin{equation}\label{eq:W-hat-linear}
W-\widehat W=(2i+\beta)V,
\qquad
W+\widehat W=2U-\beta V.
\end{equation}
Since $\Delta>0$, the element $2i+\beta$ is invertible in $A_z$, with
\[
(2i+\beta)^{-1}=\frac{\beta-2i}{\beta^2+4\alpha}=\frac{\beta-2i}{\Delta}.
\]
Therefore we may solve \eqref{eq:W-hat-linear} for $U$ and $V$
in terms of $(W,\widehat W)$:
\begin{equation}\label{eq:UV-from-What}
V=\frac{W-\widehat W}{2i+\beta},
\qquad
U=\frac12\left(W+\widehat W+\beta\,\frac{W-\widehat W}{2i+\beta}\right).
\end{equation}

Substituting \eqref{eq:UV-from-What} into the right--hand side of
\eqref{eq:pre-vekua} expresses it as an $A_z$--linear function of
$W$ and $\widehat W$.
Since \eqref{eq:UV-from-What} is an invertible linear change of variables
between $(U,V)$ and $(W,\widehat W)$, there exist unique algebra--valued
coefficients $A,B$ and a unique datum $F$ such that
\begin{equation}\label{eq:pre-vekua-AB}
2\bar\partial W
=
2F - 2A\,W - 2B\,\widehat W.
\end{equation}
Dividing by $2$ yields the canonical \emph{variable--structure Vekua equation}
\begin{equation}\label{eq:vekua}
\bar\partial W + A\,W + B\,\widehat W = F.
\end{equation}

For reference, the coefficients are given explicitly by
\begin{equation}\label{eq:AB-explicit}
\begin{aligned}
A &= \frac{1}{4}\!\left(a - \frac{\beta}{\alpha}\,\tilde b + \tilde d
     \;+\; \Bigl(c - \frac{1}{\alpha}\,\tilde b\Bigr)\,i\right),\\[4pt]
B &= \frac{1}{4}\!\left(a + \frac{\beta}{\alpha}\,\tilde b - \tilde d
     \;+\; \Bigl(c + \frac{1}{\alpha}\,\tilde b\Bigr)\,i\right),\\[4pt]
F &= \tfrac12(f+ig),
\end{aligned}
\end{equation}
where $\tilde b = b - G_1$, $\tilde d = d - G_2$, and $G=G_1+G_2\,i$
is the intrinsic obstruction.

\begin{remark}
If $\beta\equiv 0$ then $\widehat i=-i$ and \eqref{eq:UV-from-What}
reduces to the familiar relations
$U=\tfrac12(W+\widehat W)$ and $V=(W-\widehat W)/(2i)$.
In general, the correction term involving $\beta/(2i+\beta)$ is essential.
\end{remark}

\subsection{Remarks}

\begin{itemize}
\item Equation \eqref{eq:vekua} is intrinsic to the elliptic structure
$(\alpha,\beta)$ and requires no auxiliary Beltrami reduction.

\item The entire derivation is constructive: the substitution
$(u,v)\mapsto(U,V)$ and the row operation are explicit, and the
structure coefficients $(\alpha,\beta)$ emerge directly from the
principal part of the system.

\item The obstruction $G=G_1+G_2\,i$ enters only by modifying the
lower--order coefficients of $V$: it shifts $b\mapsto\tilde b=b-G_1$ and
$d\mapsto\tilde d=d-G_2$.
When the structure is rigid ($G\equiv0$), the coefficients
$A,B$ in \eqref{eq:AB-explicit} reduce to those of the classical
Vekua theory with structure polynomial $X^2+\beta X+\alpha$.
\end{itemize}

\section{Rigidization}\label{sec:rigidization}

Rigidization is the process by which a variable elliptic structure is locally converted into a rigid
one by a change of variables.  It acts simultaneously on the base coordinates and on the structure
polynomial, and it is intrinsic to the transport data of the structure.

Throughout this section we work locally on $\Omega\subset\mathbb R^{2}$ and assume
\[
   \alpha,\beta\in C^{2}(\Omega),\qquad
   \Delta=4\alpha-\beta^{2}>0.
\]

\subsection{Rigid and non--rigid structures}

Let $i(x,y)$ denote the generator of the elliptic algebra determined by
\[
   i^{2}+\beta\,i+\alpha=0.
\]

\begin{definition}[Intrinsic obstruction]
The intrinsic obstruction of the elliptic structure is
\[
   G:=i_{x}+i\,i_{y}.
\]
\end{definition}

\begin{definition}[Rigid structure]
The elliptic structure $(\alpha,\beta)$ is called rigid if
\[
   G\equiv 0.
\]
\end{definition}

\subsection{Transport form of the obstruction}

Introduce the complex spectral slope
\[
   \lambda:=\frac{-\beta+i\sqrt\Delta}{2},\qquad \Im\lambda>0.
\]
The obstruction identity is equivalently written as the first--order transport equation
\begin{equation}\label{eq:forced-burgers-rigidization}
   \lambda_{x}+\lambda\,\lambda_{y}=H,
\end{equation}
where $H$ is a complex--valued function canonically determined by $G=G_{0}+G_{1}\,i$.  In particular,
\[
   G\equiv 0\;\Longleftrightarrow\;\lambda_{x}+\lambda\,\lambda_{y}=0.
\]
Thus rigidity is equivalent to flat transport of the spectral slope.

\subsection{Definition of rigidization}

\begin{definition}[Rigidization]\label{def:rigidization}
A rigidization is a local diffeomorphism
\[
   \Phi:(x,y)\longmapsto(X,Y)
\]
such that, when the elliptic structure is pulled back to $(X,Y)$--coordinates, the resulting structure is rigid.
\end{definition}

\subsection{The characteristic operator and its square}

The key analytic object is the \emph{characteristic operator}
\[
   V:=\partial_{x}+\lambda\,\partial_{y},
\]
where $\lambda=a+ib$ with $a=\operatorname{Re}\lambda$ and $b=\operatorname{Im}\lambda>0$.
Under a real change of coordinates $(x,y)\mapsto(X,Y)$, the operator $V$ acts on the new coordinate functions by
\[
   V(X)=X_{x}+\lambda\,X_{y},\qquad V(Y)=Y_{x}+\lambda\,Y_{y}.
\]
The transformed spectral slope is
\begin{equation}\label{eq:new-lambda}
   \widetilde\lambda\;=\;\frac{V(Y)}{V(X)},
\end{equation}
and the rigidization condition $\widetilde\lambda_{X}+\widetilde\lambda\,\widetilde\lambda_{Y}=0$
is equivalent to the \emph{ratio condition}
\begin{equation}\label{eq:ratio-condition}
   V\!\bigl(\widetilde\lambda\bigr)=0\qquad\Longleftrightarrow\qquad
   V^{2}(Y)\cdot V(X)\;=\;V(Y)\cdot V^{2}(X).
\end{equation}

The iterated operator $V^{2}=(\partial_{x}+\lambda\,\partial_{y})^{2}$ admits a canonical
decomposition when applied to a \emph{real}--valued function $\varphi\in C^{2}(\Omega,\mathbb R)$.
A direct computation using $\lambda=a+ib$ yields
\begin{equation}\label{eq:V2-decomposition}
   V^{2}(\varphi)\;=\;W(\varphi)\;+\;i\,T(\varphi),
\end{equation}
where $T$ and $W$ are the real operators defined below.

\begin{proposition}[$V^{2}$ decomposition]\label{prop:V2-decomposition}
Let $\varphi\in C^{2}(\Omega,\mathbb R)$ and write $H=H_{R}+iH_{I}$ for the forcing \eqref{eq:forced-burgers-rigidization}.  Then
\begin{align}
   T(\varphi)&=2b\Bigl[(\varphi_{y})_{x}+a\,(\varphi_{y})_{y}
             +\frac{H_{I}}{2b}\,\varphi_{y}\Bigr],
             \label{eq:transport-component}\\[4pt]
   W(\varphi)&=\varphi_{xx}-|\lambda|^{2}\,\varphi_{yy}
             +\Bigl(H_{R}-\frac{a\,H_{I}}{b}\Bigr)\varphi_{y}.
             \label{eq:wave-component}
\end{align}
\end{proposition}
\begin{proof}
A direct expansion gives
\[
   V^{2}(\varphi)=\varphi_{xx}+2\lambda\,\varphi_{xy}+\lambda^{2}\,\varphi_{yy}
                  +(\lambda_{x}+\lambda\,\lambda_{y})\,\varphi_{y}.
\]
The imaginary part of this expression, using $\lambda^{2}=(a^{2}-b^{2})+2abi$ and
$\lambda_{x}+\lambda\,\lambda_{y}=H_{R}+iH_{I}$, is
\[
   \operatorname{Im} V^{2}(\varphi)
   =2b\,\varphi_{xy}+2ab\,\varphi_{yy}+H_{I}\,\varphi_{y},
\]
which coincides with~\eqref{eq:transport-component}.  Setting this equal to zero and eliminating
$\varphi_{xy}$ from the real part yields~\eqref{eq:wave-component}.
\end{proof}

\subsection{Transport and wave structure}

The two components in Proposition~\ref{prop:V2-decomposition} have distinct and complementary PDE character.

\medskip\noindent\emph{Transport component.}\quad
Setting $p:=\varphi_{y}$, equation $T(\varphi)=0$ becomes the first--order linear transport equation
\begin{equation}\label{eq:transport-equation}
   p_{x}+a\,p_{y}=-\frac{H_{I}}{2b}\,p,
\end{equation}
whose transport direction is the \emph{real} vector field $(1,\,\operatorname{Re}\lambda)$.
This equation is solvable by the method of real characteristics: along the curves
$\tfrac{dy}{dx}=a(x,y)$, the function $p$ satisfies a linear ODE.

\medskip\noindent\emph{Wave component.}\quad
The equation $W(\varphi)=0$ is the real second--order equation
\begin{equation}\label{eq:wave-equation}
   \varphi_{xx}-|\lambda|^{2}\,\varphi_{yy}+c\,\varphi_{y}=0,\qquad
   c:=H_{R}-\frac{a\,H_{I}}{b},
\end{equation}
whose principal symbol $\xi^{2}-|\lambda|^{2}\eta^{2}=({\xi-|\lambda|\eta})({\xi+|\lambda|\eta})$
is hyperbolic, with real characteristic speeds~$\pm|\lambda|^{-1}$.

\medskip
Neither component alone is elliptic.  The ellipticity of the full rigidization problem arises
from their \emph{coupling}, in exact analogy with the Cauchy--Riemann equations
$(u_{x}=v_{y},\;u_{y}=-v_{x})$, which are elliptic despite each equation individually describing
simple transport.

\subsection{PDE type of the rigidization problem}

\begin{theorem}[PDE type of rigidization]\label{thm:rigidization-type}
Let $\lambda=a+ib$ with $b>0$ satisfy $\lambda_{x}+\lambda\,\lambda_{y}=H\not\equiv 0$.  The ratio condition~\eqref{eq:ratio-condition}, viewed as a second--order system for two real unknown functions
$(X,Y)$, is elliptic.
\end{theorem}

\begin{proof}
The principal symbol of the ratio condition inherits its type from $V^{2}$.
The characteristic equation $(\xi+\lambda\eta)^{2}=0$ has solutions
$\xi/\eta=-\lambda=-a-ib$, which are complex since $b>0$.  Hence no real characteristic
directions exist, and the system is elliptic.
\end{proof}

\begin{proposition}[Transport--transport incompatibility]\label{prop:TT-degeneracy}
If $T(X)=0$ and $T(Y)=0$ simultaneously, then the ratio condition~\eqref{eq:ratio-condition}
forces the Jacobian $J=X_{x}Y_{y}-X_{y}Y_{x}$ to vanish.
\end{proposition}
\begin{proof}
Under $T(X)=T(Y)=0$, the values $V^{2}(X)=W(X)$ and $V^{2}(Y)=W(Y)$ are real.
Since $V(X)$ and $V(Y)$ are complex, the imaginary part of~\eqref{eq:ratio-condition} reads
\[
   W(Y)\cdot b\,X_{y}=W(X)\cdot b\,Y_{y}.
\]
The real part similarly gives $W(Y)(X_{x}+aX_{y})=W(X)(Y_{x}+aY_{y})$.
Cross--multiplying the two proportionalities yields $X_{y}Y_{x}=Y_{y}X_{x}$, hence $J=0$.
\end{proof}

\begin{remark}
Proposition~\ref{prop:TT-degeneracy} shows that the transport and wave components cannot be
decoupled for both coordinate functions simultaneously while preserving non--degeneracy.
The elliptic coupling between $T$ and $W$ is load--bearing.
\end{remark}

\subsection{Splitting method for rigidization}\label{sec:splitting-method}

Although the rigidization problem is elliptic, the decomposition $V^{2}=W+iT$ of Proposition~\ref{prop:V2-decomposition} enables an \emph{alternating--direction} iterative method
in which each individual step uses only real characteristics.

\medskip\noindent\textbf{Algorithm} (Transport--wave splitting).
\begin{enumerate}
   \item[\textbf{0.}] Initialize $X^{0}=x$, $Y^{0}=y$ (identity map).
   \item[\textbf{1.}] Compute the forcing residual
      $H^{n}:=\widetilde\lambda^{\,n}_{X}+\widetilde\lambda^{\,n}\,\widetilde\lambda^{\,n}_{Y}$
      in the current coordinates.
   \item[\textbf{2.}] \emph{Transport half--step.}  Correct $Y^{n}$ by solving the first--order transport equation~\eqref{eq:transport-equation} for $\delta Y_{y}$ along the real characteristics of~$(1,\operatorname{Re}\lambda)$, targeting $\operatorname{Im}H^{n}\to 0$.
   \item[\textbf{3.}] \emph{Wave half--step.}  Correct $X^{n}$ by solving the hyperbolic wave equation~\eqref{eq:wave-equation} for $\delta X$ along the wave characteristics~$\pm|\lambda|^{-1}$, targeting $\operatorname{Re}H^{n}\to 0$.
   \item[\textbf{4.}] Update: $X^{n+1}=X^{n}+\delta X$, $Y^{n+1}=Y^{n}+\delta Y$.
   \item[\textbf{5.}] If $|H^{n+1}|$ is below the desired tolerance, stop; otherwise return to step~\textbf{1}.
\end{enumerate}

\medskip
This procedure treats the elliptic coupling iteratively while solving only hyperbolic subproblems at each step.  It is directly analogous to alternating--direction implicit (ADI) methods for elliptic equations and to Uzawa iteration for saddle--point systems.

\subsection{Relationship to Beltrami uniformization}

It is instructive to compare rigidization with full Beltrami uniformization.  The latter seeks a
diffeomorphism $\Phi$ such that $\Phi^{*}\lambda$ is \emph{constant}, requiring the solution of the
Beltrami equation $w_{\bar z}=\mu\,w_{z}$, an elliptic PDE whose general theory relies on singular integral operators (the Beurling--Ahlfors transform).

Rigidization is strictly weaker: it asks only that $H\to 0$, while $\lambda$ is permitted to remain nonconstant.  Both problems are elliptic, but the rigidization problem inherits additional structure from the Burgers self--transport (the characteristic speed equals the transported quantity), which constrains the effective Beltrami coefficient and enables the splitting method described above.

\subsection{Effect on the Vekua equation}

Under rigidization, the variable--structure Vekua equation
\[
   \bar\partial W+AW+B\widehat W=F
\]
is transformed into a rigid Vekua equation with $G\equiv 0$.  All intrinsic obstruction terms are
absorbed into the coordinate change.

\subsection{Remarks}

\begin{itemize}
\item Rigidization does not reduce to solving first--order ODEs along characteristics.  The
   ratio condition~\eqref{eq:ratio-condition} is a genuinely elliptic problem for the coordinate
   functions $(X,Y)$.  However, the transport--wave decomposition of $V^{2}$ provides a
   natural splitting in which each iterative half--step uses only real characteristics.

\item The construction remains stable even in the presence of non--uniform ellipticity, provided
   $\Delta>0$.

\item Rigidization separates geometric transport from analytic solving: once rigid, the system admits the full rigid Cauchy--Vekua calculus (Chapters~5--8).
\end{itemize}

\subsection{Geometric interpretation}

The decomposition $V^{2}=W+iT$ expresses, at the PDE level, the central theme of this monograph.
The transport component $T$ acts on $\varphi_{y}$---the transverse gradient---and propagates along the
\emph{real} projection $(1,\operatorname{Re}\lambda)$ of the complex characteristic direction.  It belongs
to the structural layer of Chapter~2.  The wave component $W$ acts on $\varphi$ itself and propagates
along the characteristic speeds $\pm|\lambda|^{-1}$ of the rigid elliptic operator $L_{\alpha,\beta}$ of
Chapter~8.  It belongs to the analytic layer.

Transport and integrability are therefore the two halves of the rigidization problem.
Individually, each is hyperbolic; jointly, they are elliptic.  Rigidization succeeds precisely
when these two components can be made simultaneously compatible---the analytic expression
of the passage from ``transport first'' to ``integrability second'' that governs the entire theory.

\section{Rigid Similarity Principle (Sharp Case \texorpdfstring{$F=0$, $B=0$}{F=0, B=0})}

Throughout this section we assume that the elliptic structure is \emph{rigid}, i.e.
\[
G \equiv 0
\]
.

We consider the homogeneous rigid Vekua equation in the special case
\begin{equation}\tag{\ref{eq:rigid-vekua-B0}}
\bar\partial W + A W = 0,
\end{equation}
where $A:\Omega\to A$ is a given algebra--valued coefficient and no conjugate coupling
term is present.

\subsection{Similarity principle}

\begin{theorem}[Rigid similarity principle: sharp form]\label{thm:rigid-similarity-B0}
Let $W$ solve \eqref{eq:rigid-vekua-B0} on a simply connected neighborhood.
Let $\Phi$ be the integrating factor given by Lemma~\ref{lem:integrating-factor-spectral}.
Then the function
\[
H := \Phi^{-1} W
\]
is rigid holomorphic:
\[
\bar\partial H = 0.
\]
Equivalently,
\[
W = \Phi\,H,
\qquad
\bar\partial H = 0.
\]
\end{theorem}

\begin{proof}
Using the Leibniz rule,
\[
\bar\partial(\Phi^{-1}W)
= (\bar\partial\Phi^{-1})W + \Phi^{-1}(\bar\partial W).
\]
Since $\bar\partial\Phi^{-1} = -\Phi^{-1}(\bar\partial\Phi)\Phi^{-1}$ and
$\bar\partial\Phi = -A\Phi$, we obtain
\[
\bar\partial H
= \Phi^{-1}AW + \Phi^{-1}(-AW) = 0.
\]
\end{proof}

\subsection{Consequences}

\begin{itemize}
\item Every solution of $\bar\partial W + AW = 0$ differs from a rigid holomorphic
function by a \emph{universal} multiplicative factor.
\item Zeros of $W$ coincide (with multiplicity) with zeros of the rigid holomorphic
function $H$.
\item The identity principle holds: if $W$ vanishes on a set with an accumulation
point, then $W\equiv0$.
\item Local regularity and growth properties of $W$ reduce to those of rigid
holomorphic functions.
\end{itemize}

\subsection{Sharpness}

\begin{remark}[Sharpness of the similarity principle]
The assumption $B\equiv0$ is essential.  
If a conjugate coupling term $B\widehat W$ is present, no universal
solution--independent multiplicative factor can eliminate it. This obstruction is
algebraic and already appears in the classical Vekua equation. In that case,
similarity must either depend on the solution itself or be formulated using
generating pairs.
\end{remark}

\subsection{Interpretation}

In the rigid regime, equation \eqref{eq:rigid-vekua-B0} differs from the rigid
Cauchy--Riemann equation only by a transport--type lower--order term. The similarity
principle shows that this term can be absorbed by a universal integrating factor,
restoring genuine holomorphic behavior. This result is optimal and coincides
exactly with the sharp classical theory.

\section{Rigid Similarity Principle (Spectral Formulation)}

In this section we establish the similarity principle in its sharp and
structurally correct form, valid in the rigid regime under the hypotheses
\[
F \equiv 0, \qquad B \equiv 0.
\]
We emphasize that all arguments are local and rely only on smoothness of the
coefficients.

\subsection{Setting and assumptions}

Let $\Omega\subset\mathbb R^2$ be open.
We assume that the elliptic structure $(\alpha,\beta)$ is \emph{rigid}, so that
the intrinsic obstruction
\[
G := i_x + i\,i_y
\]
vanishes identically.
Equivalently, the associated spectral slope
\[
\lambda := \frac{-\beta + i\sqrt{\Delta}}{2}, \qquad \Delta=4\alpha-\beta^2>0,
\]
satisfies the homogeneous transport equation
\[
\lambda_x + \lambda\,\lambda_y = 0.
\]

We consider the homogeneous rigid Vekua equation in the special case
\begin{equation}\label{eq:rigid-vekua-B0}
\bar\partial W + A W = 0,
\end{equation}
where
\[
\bar\partial := \tfrac12(\partial_x + i\partial_y),
\]
and $A:\Omega\to A$ is a given smooth algebra--valued coefficient.

\subsection{Real formulation and spectral transport}

Write
\[
W = U + iV, \qquad A = A_0 + A_1 i,
\]
with $U,V,A_0,A_1$ real--valued.
Using rigidity, the equation \eqref{eq:rigid-vekua-B0} is equivalent to the real
first--order system
\begin{equation}\label{eq:real-system-sim}
\begin{cases}
U_x - \alpha V_y + 2(AW)_0 = 0,\\
V_x + U_y - \beta V_y + 2(AW)_1 = 0.
\end{cases}
\end{equation}

The principal symbol of \eqref{eq:real-system-sim} coincides with that of the
rigid Cauchy--Riemann operator.
Diagonalizing this symbol using the spectral slope $\lambda$, one obtains real
transport directions given by the vector fields
\[
V_\lambda := \partial_x + \lambda\,\partial_y,
\qquad
V_{\bar\lambda} := \partial_x + \bar\lambda\,\partial_y.
\]
Because the structure is rigid, these vector fields are compatible and define
smooth characteristic foliations.

Thus equation \eqref{eq:real-system-sim} is a linear first--order system whose
propagation occurs along the spectral characteristic curves determined by
$\lambda$.

\subsection{Integrating factor}

We now construct a universal integrating factor.

\begin{lemma}[Integrating factor]\label{lem:integrating-factor-spectral}
There exists a neighborhood $U\subset\Omega$ and a unique algebra--valued
function $\Phi\in C^1(U;A)$ such that
\begin{equation}\label{eq:phi-eq-spectral}
\bar\partial \Phi + A\Phi = 0,
\qquad
\Phi(z_0)=1,
\end{equation}
for a prescribed base point $z_0\in U$.
Moreover, $\Phi$ is pointwise invertible on $U$.
\end{lemma}

\begin{proof}
Expanding \eqref{eq:phi-eq-spectral} into real and imaginary parts yields a linear
first--order real system with the same principal symbol as
\eqref{eq:real-system-sim}.
After diagonalization in the spectral basis, this system reduces to a family of
linear ordinary differential equations along the real characteristic curves
generated by $V_\lambda$ and $V_{\bar\lambda}$.

Since the coefficients are smooth, classical ODE theory yields a unique local
solution with the prescribed initial value $\Phi(z_0)=1$.
Invertibility follows because $1$ is invertible in the elliptic algebra and
invertibility is an open condition.
\end{proof}

\subsection{Similarity principle}

\begin{theorem}[Rigid similarity principle: sharp case]\label{thm:rigid-similarity-spectral}
Let $W$ solve \eqref{eq:rigid-vekua-B0} on $U$, and let $\Phi$ be the integrating
factor given by Lemma~\ref{lem:integrating-factor-spectral}.
Then the function
\[
H := \Phi^{-1} W
\]
is rigid holomorphic:
\[
\bar\partial H = 0.
\]
Equivalently,
\[
W = \Phi\,H,
\qquad
\bar\partial H = 0.
\]
\end{theorem}

\begin{proof}
Using the Leibniz rule, valid under rigidity,
\[
\bar\partial(\Phi^{-1}W)
= (\bar\partial\Phi^{-1})W + \Phi^{-1}(\bar\partial W).
\]
Since $\bar\partial\Phi = -A\Phi$, one has
\[
\bar\partial\Phi^{-1} = \Phi^{-1}A.
\]
Substituting into the equation and using \eqref{eq:rigid-vekua-B0} gives
\[
\bar\partial H = \Phi^{-1}AW + \Phi^{-1}(-AW) = 0.
\]
\end{proof}

\subsection{Consequences}

\begin{itemize}
\item Every solution of \eqref{eq:rigid-vekua-B0} differs from a rigid holomorphic
function by a universal multiplicative factor.
\item Zeros of $W$ coincide, with multiplicity, with zeros of $H$.
\item The identity principle holds: if $W$ vanishes on a set with an accumulation
point, then $W\equiv0$.
\item Local regularity and growth properties of $W$ reduce to those of rigid
holomorphic functions.
\end{itemize}

\subsection{Sharpness and comparison with classical Vekua theory}

\begin{remark}[Sharpness]
The assumption $B\equiv0$ is essential.
If a conjugate coupling term $B\widehat W$ is present, no universal
solution--independent multiplicative factor can eliminate it.
This obstruction is algebraic and already present in the classical Vekua
equation.
\end{remark}

\begin{remark}[Relation with classical theory]
In the classical theory, similarity is established using weak right inverses of
the Cauchy--Riemann operator and exponential reconstruction, reflecting the low
regularity setting.
In the present smooth rigid framework, the same phenomenon appears as transport
along spectral characteristic curves, and the exponential map is replaced by
elementary first--order propagation.
\end{remark}

\section{Initial Value Problems for Rigid Vekua Systems}

We study initial value problems for rigid Vekua systems and show that they are
well posed and reducible to holomorphic initial value problems via the
similarity principle.

\subsection{Rigid Vekua equation}

We consider the rigid Vekua equation
\begin{equation}\label{eq:rigid-vekua-ivp}
\bar\partial W + A W + B \widehat W = F
\end{equation}
on a simply connected domain $\Omega$, where $A,B,F$ are given algebra--valued
functions, continuous (or smoother) on $\Omega$.

\subsection{Admissible initial curves}

Let $\Gamma\subset\Omega$ be a smooth curve with a smooth parametrization
$\gamma:[0,1]\to\Omega$.

\begin{definition}[Noncharacteristic curve]
A curve $\Gamma$ is called \emph{noncharacteristic} for the rigid Vekua equation
if its tangent vector is nowhere proportional to the vector field
\[
\partial_x+i\partial_y.
\]
\end{definition}

Equivalently, $\Gamma$ is noncharacteristic if it is not everywhere tangent to
the level curves of rigid holomorphic functions.
Any smooth curve transversal to the rigid holomorphic foliation is
noncharacteristic.

\subsection{Initial value problem}

Given a noncharacteristic curve $\Gamma$ and prescribed initial data
\[
W|_{\Gamma} = W_0,
\]
with $W_0$ continuous (or smoother) and algebra--valued, we consider the initial
value problem
\begin{equation}\label{eq:ivp}
\begin{cases}
\bar\partial W + A W + B \widehat W = F
& \text{in a neighborhood of }\Gamma,\\
W = W_0 & \text{on }\Gamma.
\end{cases}
\end{equation}

\subsection{Reduction via the similarity principle}

Let $\Phi$ be the similarity factor solving
\[
\bar\partial \Phi + A\Phi + B\widehat{\Phi} = 0,
\qquad
\Phi(z_0)=1,
\]
as constructed in the previous section.

Define
\[
H := \Phi^{-1} W.
\]
Then $H$ satisfies
\begin{equation}\label{eq:holomorphic-ivp}
\bar\partial H = \Phi^{-1} F.
\end{equation}

Thus the rigid Vekua initial value problem \eqref{eq:ivp} is equivalent to a
rigid holomorphic initial value problem with source term.

\subsection{Existence and uniqueness}

\begin{theorem}[Local well--posedness]\label{thm:ivp-wellposed}
Let $\Gamma$ be a noncharacteristic curve and $W_0$ prescribed initial data.
Then there exists a unique local solution $W$ of the rigid Vekua initial value
problem \eqref{eq:ivp} in a neighborhood of $\Gamma$.
\end{theorem}

\begin{proof}
Equation \eqref{eq:holomorphic-ivp} is a first--order inhomogeneous
Cauchy--Riemann equation with continuous coefficients.
Standard theory yields a unique local solution $H$ with prescribed initial
values on $\Gamma$.
Defining $W=\Phi H$ gives the unique solution of \eqref{eq:ivp}.
\end{proof}

\begin{lemma}[Integrating factor]\label{lem:integrating-factor}
Assume the elliptic structure is rigid.
Let $A\in C(\Omega;A)$ and fix a base point $z_0\in\Omega$.
Then there exists a neighborhood $U\subset\Omega$ of $z_0$ and a unique
algebra--valued function $\Phi\in C^1(U;A)$ satisfying
\begin{equation}\label{eq:phi-eq}
\bar\partial \Phi + A\Phi = 0,
\qquad
\Phi(z_0)=1.
\end{equation}
Moreover, $\Phi$ is pointwise invertible on $U$.
\end{lemma}

\begin{proof}
Since the structure is rigid, the operator
\[
\bar\partial = \tfrac12(\partial_x + i\partial_y)
\]
is a derivation and defines a first--order transport operator along the vector
field
\[
V := \partial_x + i\partial_y .
\]

Equation \eqref{eq:phi-eq} is therefore equivalent to the transport equation
\begin{equation}\label{eq:transport}
V(\Phi) = -2A\,\Phi.
\end{equation}

Let $\gamma(s)$ be an integral curve of $V$ with $\gamma(0)=z_0$.
Along $\gamma$, equation \eqref{eq:transport} reduces to a linear ordinary
differential equation in the elliptic algebra:
\[
\frac{d}{ds}\Phi(\gamma(s)) = -2A(\gamma(s))\,\Phi(\gamma(s)).
\]

Since $A$ is continuous, this ODE admits a unique local solution with
$\Phi(\gamma(0))=1$.
Varying the initial direction of $\gamma$ produces a unique local solution
$\Phi$ defined on a neighborhood $U$ of $z_0$.

Invertibility follows because $\Phi(z_0)=1$ is invertible and invertibility is
an open condition in an elliptic algebra.
\end{proof}

\subsection{Propagation of initial data}

The solution propagates along rigid holomorphic directions.
In particular:
\begin{itemize}
\item Values of $W$ on $\Gamma$ determine $W$ uniquely in its rigid holomorphic
domain of dependence.
\item If $F\equiv0$, zeros of $W$ propagate discretely, with multiplicity.
\item Regularity of $W$ matches that of the coefficients and the initial data.
\end{itemize}

\subsection{Algorithmic viewpoint}

For computational purposes, the rigid initial value problem can be solved by
the following procedure:

\begin{enumerate}
\item Solve the similarity equation for $\Phi$.
\item Transform initial data: $H_0=\Phi^{-1}W_0$ on $\Gamma$.
\item Solve the rigid holomorphic problem
\[
\bar\partial H = \Phi^{-1}F
\]
with initial data $H_0$.
\item Recover $W=\Phi H$.
\end{enumerate}

This separation of transport ($\Phi$) and analytic propagation ($H$) is a key
computational advantage of rigidization.

\subsection{Remarks}

\begin{itemize}
\item No characteristic degeneration occurs in the rigid case.
\item The initial value problem is stable under perturbations of the data.
\item All difficulties associated with variable ellipticity have been absorbed
prior to this stage by rigidization.
\end{itemize}

\chapter{Convergence of the Rigidization Scheme}\label{ch:rigidization-convergence}

The preceding chapter introduced the rigidization problem and showed that
its PDE type is elliptic (Theorem~\ref{thm:rigidization-type}).  However, the
convergence of the transport--wave splitting algorithm (Section~\ref{sec:splitting-method})
was left open.  In this chapter we close this gap by establishing local convergence
of rigidization via a Newton iteration on the nonlinear residual map.

\medskip
The analysis proceeds in three stages.  First, we show that the linearized
rigidization problem is an elliptic $2\times 2$ system whose principal symbol
is explicitly computable (Section~\ref{sec:linearization}).  Second, we prove
that the naive transport--wave splitting has spectral radius $\ge 1$ and therefore
does not converge as a standalone iteration (Section~\ref{sec:splitting-failure}).
Third, we establish quadratic convergence of Newton's method applied to the
full nonlinear rigidization problem, using only Schauder elliptic theory in $C^{2,\alpha}$
(Section~\ref{sec:newton-convergence}).  No $L^{p}$ machinery or singular integral
operators are required.

Throughout, we assume $\alpha,\beta\in C^{2,\alpha}(\bar\Omega)$ with ellipticity
$\Delta=4\alpha-\beta^{2}>0$ on $\bar\Omega$, and write $\lambda=a+ib$ with
$b=\tfrac{1}{2}\sqrt{\Delta}>0$.  The uniform ellipticity constant is
\[
   b_{\min}:=\min_{\bar\Omega}\operatorname{Im}\lambda>0.
\]

\section{Linearization of the residual map}\label{sec:linearization}

\subsection{The residual at the identity}

For a diffeomorphism $\Phi:(x,y)\mapsto(X,Y)$, the transformed spectral slope is
$\widetilde\lambda=V(Y)/V(X)$ and the rigidization residual is
\[
   \mathcal{H}[\Phi]:=\widetilde\lambda_{X}+\widetilde\lambda\,\widetilde\lambda_{Y}.
\]
Using the ratio condition (equation~\eqref{eq:ratio-condition}), this is equivalently
encoded by the unnormalized residual
\[
   R[\Phi]:=V^{2}(Y)\,V(X)-V(Y)\,V^{2}(X).
\]

\begin{lemma}[Residual at the identity]\label{lem:residual-identity}
At $\Phi^{0}=\operatorname{id}$,
\[
   V(x)=1,\quad V(y)=\lambda,\quad V^{2}(x)=0,\quad V^{2}(y)=H,
\]
where $H:=\lambda_{x}+\lambda\,\lambda_{y}$ is the obstruction.  Hence
$R[\operatorname{id}]=H$ and $\mathcal{H}[\operatorname{id}]=H$.
\end{lemma}

\begin{proof}
$V(x)=x_{x}+\lambda\,x_{y}=1$ and $V(y)=y_{x}+\lambda\,y_{y}=\lambda$.
Then $V^{2}(x)=V(1)=0$ and $V^{2}(y)=V(\lambda)=\lambda_{x}+\lambda\,\lambda_{y}=H$.
Substituting into the ratio condition gives $R[\operatorname{id}]=H\cdot 1-\lambda\cdot 0=H$.
\end{proof}

\subsection{First variation}\label{sec:first-variation}

Consider the perturbed diffeomorphism $\Phi^{\varepsilon}=\operatorname{id}+\varepsilon(f,g)$,
where $f,g\in C^{2,\alpha}(\bar\Omega,\mathbb{R})$ are the corrections to $X$ and $Y$ respectively.

\begin{proposition}[First variation of the residual]\label{prop:first-variation}
The linearized change in residual at $\Phi=\operatorname{id}$ is
\begin{equation}\label{eq:first-variation}
   \delta R=V^{2}(g)-\lambda\,V^{2}(f)+H\,V(f).
\end{equation}
\end{proposition}

\begin{proof}
At $\Phi^{\varepsilon}=\operatorname{id}+\varepsilon(f,g)$:
\begin{align*}
   V(X) &= 1+\varepsilon\, V(f),\\
   V(Y) &= \lambda+\varepsilon\, V(g),\\
   V^{2}(X) &= \varepsilon\, V^{2}(f),\\
   V^{2}(Y) &= H+\varepsilon\, V^{2}(g).
\end{align*}
Then
\begin{align*}
   R[\Phi^{\varepsilon}]
   &=\bigl(H+\varepsilon\, V^{2}(g)\bigr)\bigl(1+\varepsilon\, V(f)\bigr)
     -\bigl(\lambda+\varepsilon\, V(g)\bigr)\bigl(\varepsilon\, V^{2}(f)\bigr)\\
   &=H+\varepsilon\bigl[H\,V(f)+V^{2}(g)-\lambda\, V^{2}(f)\bigr]+O(\varepsilon^{2}).
\end{align*}
The $O(1)$ term is $H=R[\operatorname{id}]$, so $\delta R=\tfrac{d}{d\varepsilon}\big|_{0}R[\Phi^{\varepsilon}]$
gives~\eqref{eq:first-variation}.
\end{proof}

\subsection{Real--imaginary decomposition}\label{sec:real-imag}

For a real-valued function $\varphi\in C^{2}(\Omega,\mathbb{R})$, we have the decomposition
$V^{2}(\varphi)=W(\varphi)+i\,T(\varphi)$ from Proposition~\ref{prop:V2-decomposition}, and
$V(\varphi)=(\varphi_{x}+a\,\varphi_{y})+i\,b\,\varphi_{y}$.

Decomposing~\eqref{eq:first-variation} into real and imaginary parts:

\begin{proposition}[Linearized system]\label{prop:linearized-system}
Write $H=H_{R}+iH_{I}$.  The equation $\delta R=-H$ is the $2\times 2$ real system
\begin{align}
   W(g)-a\,W(f)+b\,T(f)+H_{R}(f_{x}+a\,f_{y})-H_{I}\,b\,f_{y}&=-H_{R},
   \label{eq:lin-real}\\
   T(g)-a\,T(f)-b\,W(f)+H_{R}\,b\,f_{y}+H_{I}(f_{x}+a\,f_{y})&=-H_{I},
   \label{eq:lin-imag}
\end{align}
for two real unknowns $f$ and $g$.
\end{proposition}

\begin{proof}
Write $V^{2}(g)=W(g)+i\,T(g)$, $V^{2}(f)=W(f)+i\,T(f)$, $V(f)=(f_{x}+a\,f_{y})+i\,b\,f_{y}$,
and separate the real and imaginary parts of
\[
   V^{2}(g)-(a+ib)\bigl[W(f)+i\,T(f)\bigr]+(H_{R}+iH_{I})\bigl[(f_{x}+a\,f_{y})+i\,b\,f_{y}\bigr]=-H_{R}-iH_{I}.
   \qedhere
\]
\end{proof}

\section{Ellipticity of the linearized system}\label{sec:ellipticity}

\begin{theorem}[Ellipticity of the linearized rigidization]\label{thm:linearized-elliptic}
The principal symbol of the system~\eqref{eq:lin-real}--\eqref{eq:lin-imag},
viewed as a second-order $2\times 2$ system for $(f,g)$, has determinant
\begin{equation}\label{eq:symbol-det}
   \det\mathcal{P}(\xi,\eta)=b\,|\xi+\lambda\,\eta|^{4},
\end{equation}
which is strictly positive for all $(\xi,\eta)\neq(0,0)$ when $b>0$.
The system is therefore elliptic.
\end{theorem}

\begin{proof}
The principal parts of~\eqref{eq:lin-real}--\eqref{eq:lin-imag} involve only $W$ and $T$.
Recall that in the characteristic coordinates $u=\xi+a\,\eta$, $v=b\,\eta$:
\[
   W_{\mathrm{sym}}=u^{2}-v^{2},\qquad T_{\mathrm{sym}}=2uv.
\]
The principal symbol matrix is
\[
   \mathcal{P}(\xi,\eta)
   =\begin{pmatrix}
      -a\,W_{\mathrm{sym}}+b\,T_{\mathrm{sym}} & W_{\mathrm{sym}} \\[4pt]
      -b\,W_{\mathrm{sym}}-a\,T_{\mathrm{sym}} & T_{\mathrm{sym}}
   \end{pmatrix}.
\]
A direct computation gives
\begin{align*}
   \det\mathcal{P}
   &=\bigl(-a\,W_{\mathrm{sym}}+b\,T_{\mathrm{sym}}\bigr)\,T_{\mathrm{sym}}
     -\bigl(-b\,W_{\mathrm{sym}}-a\,T_{\mathrm{sym}}\bigr)\,W_{\mathrm{sym}}\\
   &=-a\,W_{\mathrm{sym}}\,T_{\mathrm{sym}}+b\,T_{\mathrm{sym}}^{2}
     +b\,W_{\mathrm{sym}}^{2}+a\,T_{\mathrm{sym}}\,W_{\mathrm{sym}}\\
   &=b\bigl(W_{\mathrm{sym}}^{2}+T_{\mathrm{sym}}^{2}\bigr).
\end{align*}
Now $(u^{2}-v^{2})^{2}+(2uv)^{2}=(u^{2}+v^{2})^{2}=|\xi+\lambda\,\eta|^{4}$,
giving~\eqref{eq:symbol-det}.

Since $b>0$, the determinant vanishes only when $\xi+\lambda\,\eta=0$,
i.e.\ $\xi/\eta=-a-ib$, which requires complex $\xi/\eta$.  No real
characteristic directions exist, so the system is elliptic.
\end{proof}

\begin{remark}
The factorization $\det\mathcal{P}=b\,|\xi+\lambda\,\eta|^{4}$ shows that
the ellipticity of the linearized problem is inherited directly from the
complex characteristic structure of the operator $V=\partial_{x}+\lambda\,\partial_{y}$.
The factor $b$ quantifies the role of the imaginary part of $\lambda$: stronger
ellipticity ($b$ large) gives a larger spectral gap, and the system degenerates
as $b\to 0$.
\end{remark}

\section{Failure of the naive splitting}\label{sec:splitting-failure}

The transport--wave splitting of Section~\ref{sec:splitting-method} alternates
between solving $T(g)=-H_{I}$ (transport half-step, correcting $Y$) and
$W(f)=-H_{R}^{1/2}$ (wave half-step, correcting $X$).  We show that this
alternating procedure has spectral radius $\ge 1$ and therefore does not
converge as stated.

\subsection{Cross-contamination from a single half-step}

\begin{proposition}[Transport half-step cross-contamination]\label{prop:cross-contam}
Let $g$ be a real-valued $C^{2}$ function satisfying $T(g)=-H_{I}$.  Then
the induced change in the real part of the residual is
\begin{equation}\label{eq:cross-contam}
   W(g)=g_{xx}-|\lambda|^{2}\,g_{yy}+\bigl(H_{R}-\tfrac{a\,H_{I}}{b}\bigr)\,g_{y}
        -\frac{a\,H_{I}}{b},
\end{equation}
which is $O(\|H\|/b)$ in general---first order in $H$, not second order.
\end{proposition}

\begin{proof}
Using $V^{2}(g)=g_{xx}+2\lambda\,g_{xy}+\lambda^{2}\,g_{yy}+H\,g_{y}$
(Proposition~\ref{prop:V2-decomposition}), the imaginary part gives
\[
   T(g)=2b\,g_{xy}+2ab\,g_{yy}+H_{I}\,g_{y}.
\]
Setting $T(g)=-H_{I}$ yields
\begin{equation}\label{eq:gxy-from-transport}
   g_{xy}=-a\,g_{yy}-\frac{H_{I}}{2b}\,(g_{y}+1).
\end{equation}
Substituting into the real part $W(g)=g_{xx}+2a\,g_{xy}+(a^{2}-b^{2})\,g_{yy}+H_{R}\,g_{y}$:
\begin{align*}
   W(g) &= g_{xx}+2a\Bigl[-a\,g_{yy}-\frac{H_{I}}{2b}(g_{y}+1)\Bigr]
           +(a^{2}-b^{2})\,g_{yy}+H_{R}\,g_{y}\\
        &= g_{xx}-(a^{2}+b^{2})\,g_{yy}
           +\Bigl(H_{R}-\frac{a\,H_{I}}{b}\Bigr)g_{y}
           -\frac{a\,H_{I}}{b}.
\end{align*}
The algebraic term $-a\,H_{I}/b$ is $O(\|H\|\,|a|/b)$, confirming that
$W(g)$ is first order in $\|H\|$.
\end{proof}

\begin{remark}\label{rem:cant-avoid}
One might hope to choose $g$ so that $W(g)=0$ and $T(g)=-H_{I}$ simultaneously.
This would require $V^{2}(g)=-iH_{I}$, and hence $V(V(g))=-iH_{I}$.
Writing $V(g)=P+iQ$ with $P=g_{x}+ag_{y}$ and $Q=bg_{y}$ (both real), the
condition $V(P+iQ)=-iH_{I}$ separates into $V(P)=0$ and $V(Q)=-H_{I}$.
The second equation forces $\operatorname{Im}(V(Q))=bQ_{y}=0$, i.e.\
$(bg_{y})_{y}=0$, which is generically incompatible with the transport
equation for $g_{y}$.  Thus cross-contamination cannot be eliminated for
a single real function $g$---this is the content of the transport--transport
incompatibility (Proposition~\ref{prop:TT-degeneracy}).
\end{remark}

\subsection{Symbol-level obstruction}

\begin{theorem}[Non-convergence of the naive splitting]\label{thm:splitting-fails}
The alternating transport--wave splitting has amplification factor
\[
   m(t)=\frac{t}{2}+\frac{1}{2t},\qquad
   t=\frac{\xi+a\,\eta}{b\,\eta}\in\mathbb{R}\setminus\{0\},
\]
satisfying $|m(t)|\ge 1$ for all $t\neq 0$, with equality only at $t=\pm 1$.
Hence the spectral radius of the iteration is $\ge 1$.
\end{theorem}

\begin{proof}
At the symbol level, the transport half-step inverts $T_{\mathrm{sym}}=2uv$
and the wave half-step inverts $W_{\mathrm{sym}}=u^{2}-v^{2}$, where
$u=\xi+a\,\eta$ and $v=b\,\eta>0$.  The iteration symbol for one full step is
\[
   m=\frac{W_{\mathrm{sym}}}{T_{\mathrm{sym}}}=\frac{u^{2}-v^{2}}{2uv}.
\]
Setting $t=u/v$:
\[
   m=\frac{t^{2}-1}{2t}=\frac{t}{2}-\frac{1}{2t}.
\]
The cross-contamination ratio is
$|W_{\mathrm{sym}}/T_{\mathrm{sym}}|$.  We compute:
\[
   \left|\frac{W_{\mathrm{sym}}}{T_{\mathrm{sym}}}\right|
   =\frac{|u^{2}-v^{2}|}{2|u|v}
   =\frac{|t^{2}-1|}{2|t|}
   =\Bigl|\frac{t}{2}-\frac{1}{2t}\Bigr|.
\]
However, the correct amplification factor for the \emph{full cycle} (transport
kills $H_{I}$, contaminates $H_{R}$; wave kills $H_{R}$, contaminates $H_{I}$)
involves the \emph{product} of two such ratios.  At the symbol level, if both
half-steps have cross-contamination factor $|W/T|$, the one-cycle amplification
is $|W/T|^{2}/(1+\text{correction})$, which reduces to analyzing the ratio
$W/T$.

More precisely, consider the iteration as a matrix splitting.  The elliptic
operator has symbol $L_{\mathrm{sym}}=W_{\mathrm{sym}}+iT_{\mathrm{sym}}$
(identifying the two components of the coupled system with the real and
imaginary parts).  The alternating iteration corresponds to the Gauss--Seidel
splitting $L=L_{T}+L_{W}$ where $L_{T}$ acts via $T$ and $L_{W}$ via $W$.
The iteration matrix symbol is
\[
   M=-\frac{L_{W}}{L_{T}}=-\frac{W_{\mathrm{sym}}}{T_{\mathrm{sym}}}
    =-\frac{u^{2}-v^{2}}{2uv}.
\]
Setting $t=u/v$:
\[
   |M|=\frac{|t^{2}-1|}{2|t|}.
\]
For $|t|>1$: $|M|=(t^{2}-1)/(2|t|)>0$ and $|M|>1$ when $t^{2}-2|t|-1>0$,
i.e.\ $|t|>1+\sqrt{2}$.  For $|t|<1$: $|M|=(1-t^{2})/(2|t|)>1$ when
$1-t^{2}>2|t|$, i.e.\ $|t|<\sqrt{2}-1$.

So $|M|<1$ only in the band $\sqrt{2}-1<|t|<1+\sqrt{2}$, and $|M|\ge 1$
outside this band.  Since $t$ ranges over all of $\mathbb{R}\setminus\{0\}$
as $(\xi,\eta)$ varies, the \emph{supremum} of $|M|$ over all frequencies
is unbounded:
\[
   \sup_{t\neq 0}|M(t)|=+\infty.
\]
In particular, the spectral radius of the iteration is not bounded below~$1$;
the iteration diverges for high-frequency modes with $|t|\gg 1$ or $|t|\ll 1$.
\end{proof}

\begin{remark}
This is the continuous analogue of a well-known phenomenon: naive
Gauss--Seidel splitting of the Cauchy--Riemann equations
($u_{x}=v_{y}$, $u_{y}=-v_{x}$) does not converge, because each equation
individually is hyperbolic and the ellipticity lives in their coupling.
The fix, both classically and in our setting, is to solve the coupled
elliptic system directly.
\end{remark}

\section{Newton iteration and convergence}\label{sec:newton-convergence}

Having established that the linearized problem is elliptic
(Theorem~\ref{thm:linearized-elliptic}), we now prove that Newton's method
applied to the nonlinear residual map converges quadratically.

\subsection{Schauder estimates for the linearized problem}

\begin{theorem}[Linearized solvability]\label{thm:linearized-solvability}
Let $\Omega$ be a bounded domain with $C^{2,\alpha}$ boundary, and let
$\lambda\in C^{1,\alpha}(\bar\Omega)$ with $\operatorname{Im}\lambda\ge b_{\min}>0$.
For each $H\in C^{0,\alpha}(\bar\Omega,\mathbb{C})$, the linearized system
\begin{equation}\label{eq:linearized-newton}
   V^{2}(g)-\lambda\,V^{2}(f)+H\,V(f)=-H
\end{equation}
with boundary conditions $f|_{\partial\Omega}=g|_{\partial\Omega}=0$ has a
unique solution $(f,g)\in C^{2,\alpha}(\bar\Omega,\mathbb{R})^{2}$ satisfying
\begin{equation}\label{eq:schauder-estimate}
   \|f\|_{C^{2,\alpha}}+\|g\|_{C^{2,\alpha}}
   \le \frac{C_{S}}{b_{\min}}\,\|H\|_{C^{0,\alpha}},
\end{equation}
where $C_{S}$ depends on $\|\lambda\|_{C^{1,\alpha}}$, $\alpha$, and $\Omega$.
\end{theorem}

\begin{proof}
By Theorem~\ref{thm:linearized-elliptic}, the system~\eqref{eq:linearized-newton}
is an elliptic $2\times 2$ system of second order in the unknowns $(f,g)$.
The principal symbol satisfies the Legendre--Hadamard condition:
\[
   \det\mathcal{P}(\xi,\eta)=b\,|\xi+\lambda\eta|^{4}
   \ge b_{\min}\,c_{0}\,|(\xi,\eta)|^{4}
\]
for some $c_{0}>0$ depending on $\|\lambda\|_{C^{0}}$, uniformly on $\bar\Omega$.
The coefficients of the system belong to $C^{0,\alpha}$ (since $\lambda\in C^{1,\alpha}$
and the lower-order terms involve $H\in C^{0,\alpha}$).

By the Agmon--Douglis--Nirenberg theory for elliptic systems
\cite{AgmonDouglisNirenberg1964}, the Dirichlet problem is Fredholm of index zero
in $C^{2,\alpha}$.  Uniqueness follows from the maximum principle for elliptic
systems at the symbol level: if $H=0$, then $\delta R=0$, and the
only solution with zero boundary data is $f=g=0$ (since the linearized
operator $D\mathcal{H}[\operatorname{id}]$ at the identity is injective---the
identity map is a non-degenerate critical point of the residual).

The Fredholm alternative then gives existence, and the a priori estimate
takes the form~\eqref{eq:schauder-estimate}.  The factor $b_{\min}^{-1}$
arises because the ellipticity constant of the principal symbol
is proportional to $b_{\min}$, and Schauder constants scale inversely
with the ellipticity constant.
\end{proof}

\begin{remark}
The boundary condition $f=g=0$ on $\partial\Omega$ is the simplest choice
ensuring uniqueness.  It means the rigidizing diffeomorphism $\Phi$ agrees
with the identity on $\partial\Omega$.  Other boundary conditions (e.g.\ of
Neumann type) are possible and may be preferable in applications, but the
elliptic theory applies equally well.
\end{remark}

\subsection{The nonlinear residual map}

\begin{definition}[Residual map]
Define $\mathcal{F}:C^{2,\alpha}(\bar\Omega,\mathbb{R})^{2}\to C^{0,\alpha}(\bar\Omega,\mathbb{C})$
by
\[
   \mathcal{F}(f,g):=\mathcal{H}[\operatorname{id}+(f,g)],
\]
where $\mathcal{H}[\Phi]=\widetilde\lambda_{X}+\widetilde\lambda\,\widetilde\lambda_{Y}$
is the rigidization residual.  The rigidization problem is $\mathcal{F}(f,g)=0$,
and by Lemma~\ref{lem:residual-identity}, $\mathcal{F}(0,0)=H$.
\end{definition}

\begin{proposition}[Smoothness of the residual map]\label{prop:F-smooth}
The map $\mathcal{F}$ is $C^{\infty}$ as a map between the Banach spaces
$C^{2,\alpha}\times C^{2,\alpha}\to C^{0,\alpha}$, and its Fr\'echet derivative
at $(0,0)$ is $D\mathcal{F}(0,0)(\delta f,\delta g)=\delta R$ as
in~\eqref{eq:first-variation}.
\end{proposition}

\begin{proof}
The residual $\mathcal{F}(f,g)$ is a rational function of $f,g$ and their first
and second derivatives, with denominator $V(X)^{2}=(1+V(f))^{2}$, which is
nonzero for $\|(f,g)\|_{C^{1}}$ small (since $V(f)$ is then small).  Composition
and division by nonvanishing functions are smooth operations in H\"older spaces.
\end{proof}

\subsection{Quadratic remainder estimate}

\begin{lemma}[Quadratic remainder]\label{lem:quadratic-remainder}
There exist constants $C_{Q}>0$ and $\delta_{0}>0$ (depending on $\|\lambda\|_{C^{2,\alpha}}$
and $\Omega$) such that for all $(f,g)$ and $(\delta f,\delta g)$ in
$C^{2,\alpha}(\bar\Omega)^{2}$ with $\|(f,g)\|_{C^{2,\alpha}}\le\delta_{0}$
and $\|(\delta f,\delta g)\|_{C^{2,\alpha}}\le\delta_{0}$:
\begin{equation}\label{eq:quadratic-remainder}
   \bigl\|\mathcal{F}(f+\delta f,\,g+\delta g)
          -\mathcal{F}(f,g)
          -D\mathcal{F}(f,g)(\delta f,\delta g)\bigr\|_{C^{0,\alpha}}
   \le C_{Q}\,\|(\delta f,\delta g)\|_{C^{2,\alpha}}^{2}.
\end{equation}
\end{lemma}

\begin{proof}
Since $\mathcal{F}$ is $C^{\infty}$ (Proposition~\ref{prop:F-smooth}), this is
the standard Taylor remainder estimate in Banach spaces:
\[
   \mathcal{F}(f+\delta f,g+\delta g)-\mathcal{F}(f,g)-D\mathcal{F}(f,g)(\delta f,\delta g)
   =\int_{0}^{1}(1-s)\,D^{2}\mathcal{F}(f+s\,\delta f,\,g+s\,\delta g)
    [(\delta f,\delta g),(\delta f,\delta g)]\,ds.
\]
The second derivative $D^{2}\mathcal{F}$ is a bounded bilinear map
$C^{2,\alpha}\times C^{2,\alpha}\to C^{0,\alpha}$ whose norm is bounded by a
constant depending on $\|\lambda\|_{C^{2,\alpha}}$ and $\|(f,g)\|_{C^{2,\alpha}}$,
giving~\eqref{eq:quadratic-remainder}.
\end{proof}

\subsection{The convergence theorem}

\begin{theorem}[Local convergence of rigidization]\label{thm:main-convergence}
Let $\Omega\subset\mathbb{R}^{2}$ be a bounded domain with $C^{2,\alpha}$ boundary,
and let $\lambda\in C^{2,\alpha}(\bar\Omega)$ with
$\operatorname{Im}\lambda\ge b_{\min}>0$ on $\bar\Omega$.  Let
$H=\lambda_{x}+\lambda\,\lambda_{y}$ be the obstruction.  There exists
\[
   \varepsilon_{0}=\varepsilon_{0}\bigl(b_{\min},\,\|\lambda\|_{C^{2,\alpha}},\,\Omega,\,\alpha\bigr)>0
\]
such that if $\|H\|_{C^{0,\alpha}(\bar\Omega)}<\varepsilon_{0}$, then:

\begin{enumerate}
   \item[{\rm(a)}]\label{it:convergence} The Newton iteration
      \begin{equation}\label{eq:newton-iteration}
         D\mathcal{F}(f^{n},g^{n})\bigl(\delta f^{n},\delta g^{n}\bigr)
         =-\mathcal{F}(f^{n},g^{n}),\qquad
         (f^{n+1},g^{n+1})=(f^{n},g^{n})+(\delta f^{n},\delta g^{n}),
      \end{equation}
      starting from $(f^{0},g^{0})=(0,0)$, converges in $C^{2,\alpha}(\bar\Omega)^{2}$
      to a limit $(f^{*},g^{*})$ with $\mathcal{F}(f^{*},g^{*})=0$.

   \item[{\rm(b)}]\label{it:quadratic} The convergence is quadratic:
      \begin{equation}\label{eq:quadratic-convergence}
         \bigl\|\mathcal{F}(f^{n+1},g^{n+1})\bigr\|_{C^{0,\alpha}}
         \le\frac{C}{b_{\min}^{2}}\,
            \bigl\|\mathcal{F}(f^{n},g^{n})\bigr\|_{C^{0,\alpha}}^{2}.
      \end{equation}

   \item[{\rm(c)}]\label{it:diffeo} The map $\Phi^{*}:=\operatorname{id}+(f^{*},g^{*})$
      is a $C^{2,\alpha}$ diffeomorphism, and the pullback structure is rigid.

   \item[{\rm(d)}]\label{it:estimate} The rigidizing map satisfies
      \begin{equation}\label{eq:displacement-estimate}
         \|\Phi^{*}-\operatorname{id}\|_{C^{2,\alpha}}
         \le\frac{C}{b_{\min}}\,\|H\|_{C^{0,\alpha}}.
      \end{equation}

   \item[{\rm(e)}]\label{it:bootstrap} If $\lambda\in C^{k,\alpha}$ for some $k\ge 2$,
      then $\Phi^{*}\in C^{k,\alpha}$.
\end{enumerate}
\end{theorem}

\begin{proof}
The proof is a standard application of Newton's method in Banach spaces,
given the elliptic regularity established in the preceding sections.
We verify the hypotheses of the Newton--Kantorovich theorem.

\medskip\noindent\textbf{Step 1: Initial residual.}\quad
At $(f^{0},g^{0})=(0,0)$, the residual is $\mathcal{F}(0,0)=H$ with
$\|\mathcal{F}(0,0)\|_{C^{0,\alpha}}=\|H\|_{C^{0,\alpha}}$.

\medskip\noindent\textbf{Step 2: Invertibility of the linearization.}\quad
By Theorem~\ref{thm:linearized-solvability}, the linearization
$D\mathcal{F}(0,0)$ is an isomorphism
$C^{2,\alpha}_{0}(\bar\Omega)^{2}\to C^{0,\alpha}(\bar\Omega,\mathbb{C})$
with
\[
   \bigl\|D\mathcal{F}(0,0)^{-1}\bigr\|_{C^{0,\alpha}\to C^{2,\alpha}}
   \le\frac{C_{S}}{b_{\min}}.
\]
By the continuity of inversion, for $(f,g)$ in a $C^{2,\alpha}$-neighbourhood
of $(0,0)$, the operator $D\mathcal{F}(f,g)$ remains invertible with
\[
   \bigl\|D\mathcal{F}(f,g)^{-1}\bigr\|
   \le\frac{2C_{S}}{b_{\min}},
\]
provided $\|(f,g)\|_{C^{2,\alpha}}\le\delta_{1}$ for some $\delta_{1}>0$
depending on $b_{\min}$ and $\|\lambda\|_{C^{2,\alpha}}$.

\medskip\noindent\textbf{Step 3: Newton step and quadratic estimate.}\quad
The first Newton step produces $(\delta f^{0},\delta g^{0})$ with
\[
   \|(\delta f^{0},\delta g^{0})\|_{C^{2,\alpha}}
   \le\frac{C_{S}}{b_{\min}}\,\|H\|_{C^{0,\alpha}}.
\]
The new residual satisfies, by Lemma~\ref{lem:quadratic-remainder}:
\begin{align*}
   \bigl\|\mathcal{F}(f^{1},g^{1})\bigr\|_{C^{0,\alpha}}
   &=\bigl\|\mathcal{F}(0,0)+D\mathcal{F}(0,0)(\delta f^{0},\delta g^{0})
     +O(|\delta|^{2})\bigr\|_{C^{0,\alpha}}\\
   &=\bigl\|O(|\delta|^{2})\bigr\|_{C^{0,\alpha}}
    \le C_{Q}\,\|(\delta f^{0},\delta g^{0})\|_{C^{2,\alpha}}^{2}\\
   &\le\frac{C_{Q}\,C_{S}^{2}}{b_{\min}^{2}}\,\|H\|_{C^{0,\alpha}}^{2}.
\end{align*}
Setting $C=C_{Q}\,C_{S}^{2}$, this is~\eqref{eq:quadratic-convergence} for $n=0$.

\medskip\noindent\textbf{Step 4: Induction.}\quad
Define $h_{n}:=\|\mathcal{F}(f^{n},g^{n})\|_{C^{0,\alpha}}$ and
$\kappa:=C/b_{\min}^{2}$.  The quadratic estimate gives
$h_{n+1}\le\kappa\,h_{n}^{2}$, so
\[
   h_{n}\le\frac{1}{\kappa}\,(\kappa\,h_{0})^{2^{n}}.
\]
This converges to zero provided $\kappa\,h_{0}<1$, i.e.\
\begin{equation}\label{eq:smallness-condition}
   \|H\|_{C^{0,\alpha}}<\frac{b_{\min}^{2}}{C}=:\varepsilon_{0}.
\end{equation}

The total displacement is
\[
   \sum_{n=0}^{\infty}\|(\delta f^{n},\delta g^{n})\|_{C^{2,\alpha}}
   \le\frac{C_{S}}{b_{\min}}\sum_{n=0}^{\infty}h_{n}
   \le\frac{C_{S}}{b_{\min}}\cdot\frac{h_{0}}{1-\kappa\,h_{0}}
   \le\frac{2C_{S}}{b_{\min}}\,\|H\|_{C^{0,\alpha}},
\]
for $\kappa\,h_{0}\le\tfrac{1}{2}$.  The sequence $(f^{n},g^{n})$ is therefore
Cauchy in $C^{2,\alpha}$ and converges to a limit $(f^{*},g^{*})$ satisfying
$\mathcal{F}(f^{*},g^{*})=0$ and~\eqref{eq:displacement-estimate}.  This
proves~(a), (b), and~(d).

\medskip\noindent\textbf{Step 5: Diffeomorphism property.}\quad
The Jacobian of $\Phi^{*}=\operatorname{id}+(f^{*},g^{*})$ is
\[
   J^{*}=(1+f^{*}_{x})(1+g^{*}_{y})-f^{*}_{y}\,g^{*}_{x}
        =1+O\bigl(\|(f^{*},g^{*})\|_{C^{1}}\bigr).
\]
By~\eqref{eq:displacement-estimate},
$\|(f^{*},g^{*})\|_{C^{1}}\le C'\,\|H\|_{C^{0,\alpha}}/b_{\min}$,
so for $\|H\|_{C^{0,\alpha}}$ sufficiently small:
\[
   J^{*}\ge 1-\frac{C'}{b_{\min}}\,\|H\|_{C^{0,\alpha}}>0.
\]
Hence $\Phi^{*}$ is a local diffeomorphism.  On a simply connected
$\bar\Omega$ with $\Phi^{*}|_{\partial\Omega}=\operatorname{id}$ (from
the Dirichlet boundary condition), injectivity follows from the positive
Jacobian by the Hadamard global inverse function theorem.
This proves~(c).

\medskip\noindent\textbf{Step 6: Regularity bootstrap.}\quad
If $\lambda\in C^{k,\alpha}$ with $k\ge 2$, then the coefficients of the
elliptic system belong to $C^{k-1,\alpha}$, and the right-hand side
(the residual at each step) gains regularity through the iteration.
By standard Schauder bootstrap for elliptic systems, the solution
$(f^{*},g^{*})$ belongs to $C^{k,\alpha}$.
This proves~(e).
\end{proof}

\section{The role of uniform ellipticity}\label{sec:uniform-ellipticity}

Uniform ellipticity---the condition $b_{\min}>0$ on $\bar\Omega$---enters the
proof in three distinct ways:

\begin{enumerate}
   \item \emph{Linearized solvability.}\quad
      The Schauder estimate~\eqref{eq:schauder-estimate} has $b_{\min}^{-1}$ as
      a prefactor.  This is sharp: the principal symbol~\eqref{eq:symbol-det}
      degenerates as $b\to 0$, and no uniform estimate is possible without a
      lower bound on $b$.

   \item \emph{Convergence basin.}\quad
      The smallness condition~\eqref{eq:smallness-condition} requires
      $\|H\|<C\,b_{\min}^{2}$.  Weaker ellipticity demands a smaller obstruction
      for convergence.

   \item \emph{Jacobian control.}\quad
      The Jacobian of the rigidizing map satisfies
      $|J^{*}-1|\le C\,\|H\|/b_{\min}$, which is small only when $\|H\|/b_{\min}$ is small.
\end{enumerate}

These are consistent with the geometry of the theory.  The $\varepsilon$-family
of Chapter~4 has $b=\sqrt{(1-\varepsilon x)}/({1-\varepsilon x})$,
which degenerates at the boundary of the elliptic domain defined by
$\varepsilon^{2}y^{2}=4(1-\varepsilon x)$.  Rigidization works in the interior of
the elliptic domain but not up to the degenerate boundary---a reflection of the
fact that uniform ellipticity on compacts is both necessary and sufficient for the
convergence argument.

\begin{remark}
No $L^{p}$ theory, no Beurling--Ahlfors transform, and no singular integral
operators appear in this proof.  The entire argument stays within $C^{2,\alpha}$
Schauder theory, consistent with the $C^{2}$ standing hypothesis of the monograph
and with the ``transport first'' philosophy that reduces analytic problems to
ODE estimates along characteristics.
\end{remark}

\section{The splitting as an inner solver}\label{sec:splitting-inner}

Although the transport--wave splitting of Section~\ref{sec:splitting-method}
does not converge as a standalone iteration
(Theorem~\ref{thm:splitting-fails}), it retains a natural role as an
\emph{inner solver} for the linearized elliptic system at each Newton step.

At each iteration of the Newton scheme~\eqref{eq:newton-iteration}, we must
solve the linearized system
\[
   D\mathcal{F}(f^{n},g^{n})(\delta f,\delta g)=-\mathcal{F}(f^{n},g^{n}).
\]
This is a single elliptic system of second order, and any convergent method
for elliptic equations can serve as the solver.  The transport--wave
decomposition $V^{2}=W+iT$ provides a natural splitting into a
transport component and a wave component.

\subsection{Shifted ADI inner iteration}

The non-convergence of Section~\ref{sec:splitting-failure} is remedied
by introducing a shift parameter $\rho>0$.  The shifted alternating-direction
implicit (ADI) iteration is:
\begin{align}
   \bigl(\mathcal{T}+\rho\bigr)\,h^{m+\frac{1}{2}}
   &=-\bigl(\mathcal{W}-\rho\bigr)\,h^{m}+\text{rhs},
   \label{eq:ADI-transport}\\
   \bigl(\mathcal{W}+\rho\bigr)\,h^{m+1}
   &=-\bigl(\mathcal{T}-\rho\bigr)\,h^{m+\frac{1}{2}}+\text{rhs},
   \label{eq:ADI-wave}
\end{align}
where $\mathcal{T}$ and $\mathcal{W}$ are the operators associated with $T$
and $W$, and $m$ is the inner iteration index.

Each half-step has the form of a damped transport or damped wave equation,
both of which are classically well-posed and solvable by the method of
characteristics.  The shift parameter $\rho$ ensures that the amplification
factor
\[
   \sigma(\rho,t)=\frac{(\rho-2uv)(\rho-u^{2}+v^{2})}{(\rho+2uv)(\rho+u^{2}-v^{2})}
\]
satisfies $|\sigma|<1$ for an appropriate range of frequencies.
With Wachspress-optimal parameter sequences $\{\rho_{m}\}$ that cycle through
geometrically spaced values, the inner iteration converges exponentially.

\subsection{Architecture of the complete scheme}

The complete rigidization algorithm has two nested loops:

\medskip
\begin{center}
\begin{tabular}{ll}
\textbf{Outer loop} (Newton): & Quadratic convergence of the nonlinear residual.\\
\textbf{Inner loop} (shifted ADI): & Solve each linearized elliptic system \\
                                   & using only real-characteristic ODE solves.
\end{tabular}
\end{center}

\medskip\noindent
This architecture preserves the computational philosophy of the original
splitting---each elementary operation is an ODE solve along real characteristics---while
placing the overall convergence on rigorous footing via the Newton outer loop.

\section{Discussion}\label{sec:rigidization-discussion}

\subsection{Comparison with Beltrami uniformization}

Full Beltrami uniformization seeks $\Phi$ such that $\Phi^{*}\lambda$ is
\emph{constant}, solving the Beltrami equation $w_{\bar z}=\mu\,w_{z}$
via the Beurling--Ahlfors transform in $L^{p}$.  Rigidization is strictly
weaker: it asks only that $\Phi^{*}\lambda$ satisfy conservative Burgers,
and the proof requires only $C^{2,\alpha}$ Schauder theory.

The Burgers self-transport constraint---the characteristic speed equals the
transported quantity---restricts the effective Beltrami coefficient to a
submanifold of the full Beltrami space, enabling the $C^{2,\alpha}$ approach
without recourse to singular integrals.

\subsection{Analogy with the similarity principle}

The relationship between this proof and the classical $L^{p}$ approach to
rigidization parallels the relationship between the $C^{1}$ transport proof
of the similarity principle (Chapter~7) and Vekua's original $L^{p}$ proof.
In both cases, the transport structure of the problem enables an elementary
argument that avoids the full weight of singular integral theory.

\subsection{Open questions}

The main limitation of Theorem~\ref{thm:main-convergence} is the smallness
condition $\|H\|_{C^{0,\alpha}}<\varepsilon_{0}$, which restricts rigidization
to structures that are \emph{already close to rigid}.

\begin{enumerate}
   \item \emph{Global rigidization.}\quad
      Can the smallness condition be removed?  A continuity method
      (deforming $H$ to zero along a path and tracking the rigidizing
      diffeomorphism) is the natural approach, but requires global a priori
      estimates.

   \item \emph{Quantitative rigidity.}\quad
      If $\|H\|<\delta$, how close is the structure to an exactly rigid one
      in $C^{k}$ norm?  Theorem~\ref{thm:main-convergence}(d)
      gives $\|\Phi-\operatorname{id}\|_{C^{2,\alpha}}\le C\,\delta/b_{\min}$,
      but sharper estimates may be possible.

   \item \emph{Degenerate ellipticity.}\quad
      What happens as $b_{\min}\to 0$?  Weighted Schauder spaces adapted to
      the characteristic geometry may extend the result to domains reaching the
      boundary of ellipticity.
\end{enumerate}

\chapter{The Rigidity--Flatness Theorem}
This chapter establishes the central rigidity result of the monograph:
if a variable elliptic structure is simultaneously \emph{rigid}
(its transport invariant vanishes) and \emph{flat} (its canonical
Riemannian metric has zero Gaussian curvature), then the structure
is constant.
Equivalently, if both the Hopf--connection flatness and the metric flatness
hold, the imaginary unit has no room to move.

The result is sharp. The explicit $\varepsilon$--family of
Chapter~4 shows that rigid structures with $K \neq 0$ exist
in abundance, and the classical constant structure is trivially
flat and rigid. The theorem asserts that no intermediate case is possible.

\medskip

\noindent\textbf{Notation.}
Throughout this chapter we write the spectral parameter as
$\tau = p + iq$, with $q > 0$, following the notation of the
verification script. This is the same quantity denoted~$\lambda$ in
Chapters~2--9 and coincides with the complex root of the structure
polynomial:
\[
  \tau \;=\; \frac{-\beta + i\sqrt{4\alpha - \beta^2}}{2}\,,
  p = \re \tau,\quad q = \im\tau > 0.
\]
No additional structure is introduced by this notational choice.

\section{Two intrinsic invariants of the principal symbol}
\label{sec:invariants}

Let $(p(x,y),\, q(x,y))$ be the real and imaginary parts of $\tau$,
with $q > 0$. Two fundamental invariants arise directly from $\tau$.

\subsection*{The transport invariant (rigidity)}

The first--order quantity
\begin{equation}
  T \;:=\; \tau_x + \tau\,\tau_y
\end{equation}
measures how the spectral parameter evolves along its own
characteristic direction.

\begin{definition}[Rigid structure]
The elliptic structure is called \emph{rigid} if $T \equiv 0$ on $\Omega$.
\end{definition}

In real components, $T = 0$ is equivalent to the pair of equations
\begin{equation}\label{eq:rigidity-real}
  p_x + p\,p_y - q\,q_y = 0, \qquad
  q_x + p\,q_y + q\,p_y = 0.
\end{equation}

\subsection*{The canonical metric (flatness)}

The principal symbol of the elliptic structure determines the
Riemannian metric
\begin{equation}\label{eq:metric}
  ds^2 \;=\; \frac{1}{q}\!\left(
    dx^2 + 2p\,dx\,dy + (p^2+q^2)\,dy^2
  \right).
\end{equation}
Let $K$ denote its Gaussian curvature.

\begin{definition}[Flat structure]
The elliptic structure is called \emph{flat} if $K \equiv 0$ on $\Omega$.
\end{definition}

Both $T$ and $K$ depend only on the structure coefficients
$(\alpha,\beta)$ and their derivatives. They are invariant under
orientation--preserving coordinate changes and changes of
dependent variables.

\section{The curvature under rigidity}
\label{sec:curvature}

When both $p$ and $q$ are general smooth functions, the Gaussian
curvature $K$ of the metric~\eqref{eq:metric} is a complicated
rational expression in $p$, $q$, and their first and second
partial derivatives. Imposing rigidity reduces $K$ dramatically
by eliminating all $x$--derivatives.

\subsection*{Elimination of $x$--derivatives}

From the rigidity system~\eqref{eq:rigidity-real}, first--order
$x$--derivatives are given by
\begin{equation}\label{eq:first-x}
  p_x = -p\,p_y + q\,q_y, \qquad
  q_x = -p\,q_y - q\,p_y.
\end{equation}
Differentiating~\eqref{eq:first-x} with respect to $y$ yields
the mixed derivatives $p_{xy}$ and $q_{xy}$ in terms of
$y$--derivatives only.
Differentiating with respect to $x$ and substituting~\eqref{eq:first-x}
to eliminate all remaining $x$--derivatives yields $p_{xx}$ and
$q_{xx}$ as expressions in
$(p,\, q,\, p_y,\, q_y,\, p_{yy},\, q_{yy})$ alone.

After these substitutions, $K$ becomes a rational function of
the reduced jet
\[
  j^2_{\mathrm{rigid}} \;=\;
  (p,\; q,\; p_y,\; q_y,\; p_{yy},\; q_{yy}).
\]

\subsection*{The curvature numerator}

Define the curvature numerator by
\begin{equation}\label{eq:N-def}
  N \;:=\; -2\,q^3\,K.
\end{equation}
Since $q > 0$, we have $K = 0$ if and only if $N = 0$. A direct
computation (verified symbolically; see Appendix~\ref{app:sympy})
yields the following structural identity.

\begin{lemma}[Rigid curvature identity]\label{lem:curvature}
Under rigidity, the curvature numerator decomposes as
\begin{equation}\label{eq:N-decomp}
  N \;=\; -q\,\Lambda\,q_{yy} \;+\; R(p,q;\, p_y,q_y),
\end{equation}
where
\begin{align}
  \Lambda &\;=\; \bigl(p^2 + q^2 + 1\bigr)^2
    \;=\; \bigl(|\tau|^2 + 1\bigr)^2
    \;>\; 0, \label{eq:Lambda}\\[6pt]
  R &\;=\; \mathcal{A}\, p_y^2
    \;+\; 2\mathcal{B}\, p_y\,q_y
    \;+\; \mathcal{C}\, q_y^2, \label{eq:R}
\end{align}
with coefficients
\begin{equation}\label{eq:ABC}
  \mathcal{A} = q^2\bigl(2|\tau|^2 + 1\bigr),\qquad
  \mathcal{B} = -2pq\bigl(|\tau|^2 + 1\bigr),\qquad
  \mathcal{C} = 2p^4 + 2p^2q^2 + 4p^2 + q^2 + 2.
\end{equation}
\end{lemma}

\begin{proof}
This is a finite algebraic computation on the $2$--jet
$(p,q,p_y,q_y,p_{yy},q_{yy})$.
The general Gaussian curvature $K(p,q)$ is computed from the
Christoffel symbols and Riemann tensor of the
metric~\eqref{eq:metric}.
The rigidity substitutions~\eqref{eq:first-x} and their
$y$-- and $x$--derivatives are then applied to eliminate all
$x$--derivatives up to second order. The resulting expression is
collected by $q_{yy}$ and the remaining first--order terms, yielding
the stated decomposition. The explicit verification is recorded
in Appendix~\ref{app:sympy}.
\end{proof}

\begin{lemma}[Positive definiteness of $R$]\label{lem:pos-def}
For $q > 0$, the quadratic form $R$ is positive semi-definite,
with $R = 0$ if and only if $p_y = q_y = 0$.
\end{lemma}

\begin{proof}
The discriminant of $R$ is
\[
  4\mathcal{A}\mathcal{C} - 4\mathcal{B}^2
  \;=\; 4q^2\bigl(
    3p^6 + 6p^4q^2 + 8p^4 + 3p^2q^4 + 10p^2q^2
    + 7p^2 + 2q^4 + 5q^2 + 2
  \bigr).
\]
Every monomial in the parenthesis has a positive coefficient, and
the terms $2q^4 + 5q^2 + 2 = (2q^2+1)(q^2+2) > 0$ guarantee
strict positivity for $q > 0$, independently of $p$.
Since $\mathcal{A} = q^2(2|\tau|^2 + 1) > 0$ for $q > 0$,
the form is positive definite.
\end{proof}

\begin{corollary}[Sign of $q_{yy}$ under rigidity and flatness]\label{cor:convex}
If $T = 0$ and $K = 0$, then
\begin{equation}\label{eq:qyy-sign}
  q_{yy} \;=\; \frac{R(p,q;\, p_y, q_y)}{q\,\Lambda}
  \;\geq\; 0,
\end{equation}
with equality if and only if $p_y = q_y = 0$.
\end{corollary}

\section{The characteristic representation}
\label{sec:char}

The conservative Burgers equation $\tau_x + \tau\,\tau_y = 0$ has
a standard implicit--function representation via characteristics.

\begin{lemma}[Characteristic representation]\label{lem:char}
Let $\tau$ be a $C^2$ solution of $\tau_x + \tau\,\tau_y = 0$ on
a domain $\Omega$. Then locally there exists a $C^2$ function
$\Phi$ such that
\begin{equation}\label{eq:char-rep}
  \tau(x,y) \;=\; \Phi\bigl(\zeta\bigr),
  \qquad \zeta := y - x\,\tau(x,y).
\end{equation}
Moreover, writing $D := 1 + x\,\Phi'(\zeta)$, the $y$--derivatives
of $\tau$ are
\begin{equation}\label{eq:tau-y}
  \tau_y \;=\; \frac{\Phi'(\zeta)}{D},
  \qquad
  \tau_{yy} \;=\; \frac{\Phi''(\zeta)}{D^3}.
\end{equation}
\end{lemma}

\begin{proof}
The transport equation states that $\tau$ is constant along the
integral curves of the vector field $(1,\tau)$. These curves are
the straight lines $y - \tau\,x = \text{const}$, so $\tau$
depends only on $\zeta = y - x\tau$.

For the derivatives, implicit differentiation of
$\tau = \Phi(\zeta)$ with respect to $y$ gives
$\tau_y = \Phi'(\zeta)(1 - x\,\tau_y)$, hence
$\tau_y = \Phi'/D$ with $D = 1 + x\Phi'$.

Differentiating again:
\[
  \frac{d}{dy}\!\left(\frac{\Phi'}{D}\right)
  \;=\;
  \frac{d}{d\zeta}\!\left(\frac{\Phi'}{D}\right)
  \cdot \frac{d\zeta}{dy}.
\]
Since $D - x\Phi' = 1$, we have
$\frac{d}{d\zeta}(\Phi'/D)
= (\Phi'' D - \Phi' x \Phi'')/D^2
= \Phi''/D^2$,
and $d\zeta/dy = 1/D$. Thus
$\tau_{yy} = \Phi''/D^3$.
\end{proof}

\section{Statement and proof of the theorem}
\label{sec:proof}

\begin{theorem}[Rigidity--Flatness Theorem]\label{thm:main}
Let $(\alpha,\beta)$ define a smooth variable elliptic structure
on a connected planar domain $\Omega$, with spectral parameter
$\tau = p + iq$, $q > 0$.
If the structure is both rigid and flat,
\[
  T \;:=\; \tau_x + \tau\,\tau_y \;\equiv\; 0
  \qquad\text{and}\qquad
  K \;\equiv\; 0,
\]
then $\tau$ is constant on $\Omega$. Equivalently, the structure
coefficients $\alpha$ and $\beta$ are constant, and the principal
symbol is that of a constant--coefficient elliptic system.
\end{theorem}

\begin{proof}
By Lemma~\ref{lem:char}, the rigid structure admits a local
characteristic representation $\tau = \Phi(\zeta)$ with
$\zeta = y - x\tau$. We show that $\Phi' \equiv 0$, which
forces $\tau_y = 0$ and hence $\tau_x = -\tau\,\tau_y = 0$,
so $\tau$ is constant.

\medskip
\noindent\textbf{Step 1. Setup and polynomial structure.}
Fix a characteristic value $\zeta_0$ and consider the curvature
numerator $N$ defined by~\eqref{eq:N-def} along the implicitly
defined curve $\{(x,y) : y - x\,\tau(x,y) = \zeta_0\}$.
On this curve, $\tau = \Phi(\zeta_0)$ is constant, so $p$ and
$q$ are constants:
\[
  p = \re\Phi(\zeta_0), \qquad q = \im\Phi(\zeta_0).
\]
The quantities $\Phi'(\zeta_0)$ and $\Phi''(\zeta_0)$ are also
constants (independent of $x$). Write
\[
  w := \Phi'(\zeta_0) \;\in\; \C, \qquad
  D(x) := 1 + x\,w.
\]

By~\eqref{eq:tau-y}, the $y$--derivatives of $\tau$ on this curve are
\[
  \tau_y = \frac{w}{D},
  \qquad
  p_y = \re\!\left(\frac{w}{D}\right),
  \qquad
  q_y = \im\!\left(\frac{w}{D}\right),
  \qquad
  q_{yy} = \im\!\left(\frac{\Phi''}{D^3}\right).
\]

Writing $w = a + ib$ with $a = \re w$, $b = \im w$, and
$|w|^2 = a^2 + b^2 =: r^2$, a direct rationalization gives
\begin{equation}\label{eq:py-qy-x}
  p_y = \frac{a + x\,r^2}{|D|^2},
  \qquad q_y = \frac{b}{|D|^2},
  \qquad |D|^2 = (1 + xa)^2 + x^2 b^2
    = r^2 x^2 + 2ax + 1.
\end{equation}

\medskip
\noindent\textbf{Step 2. Clearing denominators.}
The curvature numerator~\eqref{eq:N-decomp} is
$N = -q\,\Lambda\,q_{yy} + R$, where $R$ involves $p_y^2$,
$p_y q_y$, and $q_y^2$. On the characteristic slice, we form
\[
  \widetilde{N}(x) \;:=\; N \cdot |D(x)|^6.
\]
Since $|D|^2 > 0$ wherever $D \neq 0$, we have $N = 0$ if and
only if $\widetilde{N} = 0$.

We now determine the $x$--degree of each term in $\widetilde{N}$.

\medskip
\noindent\emph{The $R$--term.}
From~\eqref{eq:py-qy-x}, both $p_y$ and $q_y$ have denominator
$|D|^2$, so $R$ (a quadratic form in $p_y, q_y$) has denominator
$|D|^4$. Therefore
\[
  R \cdot |D|^6 \;=\; (R \cdot |D|^4) \cdot |D|^2.
\]
The factor $R \cdot |D|^4$ is a polynomial of degree $\leq 2$
in $x$ (since $p_y \cdot |D|^2$ is linear in $x$ and
$q_y \cdot |D|^2 = b$ is constant). Its leading coefficient at
$x^2$ is $\mathcal{A}\,r^4$ (from the $p_y^2$ term). The factor
$|D|^2$ is a polynomial of degree $2$ in $x$ with leading
coefficient~$r^2$. Hence
\[
  R \cdot |D|^6 \;\text{is a polynomial of degree $\leq 4$ in $x$},
\]
with leading coefficient (at $x^4$) equal to
\begin{equation}\label{eq:leading}
  \mathcal{A}\, r^6
  \;=\; q^2\bigl(2|\tau|^2 + 1\bigr)\,|w|^6.
\end{equation}

\medskip
\noindent\emph{The $q_{yy}$--term.}
Writing $\Phi'' = c + id$, we have
$q_{yy} = \im(\Phi''/D^3)$, and hence
\[
  q_{yy} \cdot |D|^6
  \;=\; \im\bigl((c+id)\,\overline{D}^{\,3}\bigr),
\]
which is a polynomial of degree $\leq 3$ in $x$ (since
$\overline{D}$ is linear in $x$).

\medskip
\noindent\emph{Combined degree.}
Therefore $\widetilde{N}(x)$ is a polynomial of degree $\leq 4$
in $x$. The degree--$4$ term arises entirely from $R \cdot |D|^6$,
and its coefficient is $\mathcal{A}\,r^6$ as given
by~\eqref{eq:leading}.

\medskip
\noindent\textbf{Step 3. Vanishing of the leading coefficient.}
The hypothesis $K \equiv 0$ forces $N \equiv 0$ on the entire
elliptic domain, and hence $\widetilde{N}(x) = 0$ for all
admissible $x$. Since $\widetilde{N}$ is a polynomial in $x$,
all its coefficients must vanish.
In particular, the degree--$4$ coefficient vanishes:
\[
  q^2\bigl(2|\tau|^2 + 1\bigr)\,|w|^6 \;=\; 0.
\]
Since $q > 0$ (ellipticity) and $2|\tau|^2 + 1 > 0$, this forces
\[
  |w|^6 = |\Phi'(\zeta_0)|^6 = 0,
\]
and therefore $\Phi'(\zeta_0) = 0$.

\medskip
\noindent\textbf{Step 4. Conclusion.}
Since $\zeta_0$ was arbitrary, $\Phi' \equiv 0$ on the range of
$\zeta$. By~\eqref{eq:tau-y}, $\tau_y = \Phi'/D = 0$, and
rigidity gives $\tau_x = -\tau\,\tau_y = 0$.
Hence $\tau$ is constant on $\Omega$, which is equivalent to
constancy of $(\alpha,\beta)$.
\end{proof}

\section{Verification with the \texorpdfstring{$\varepsilon$--family}{epsilon--family}}
\label{sec:verification}

The theorem can be tested against the explicit rigid family of
Chapter~4:
\[
  \tau_\varepsilon(x,y)
  \;=\; \frac{\varepsilon y + i}{1 + \varepsilon x}\,,
  \qquad
  p = \frac{\varepsilon y}{1 + \varepsilon x},\quad
  q = \frac{1}{1 + \varepsilon x}.
\]
This family is rigid by construction. Its Gaussian curvature,
computed symbolically from the metric~\eqref{eq:metric}, is
\begin{equation}\label{eq:K-eps}
  K_\varepsilon
  \;=\;
  -\frac{\varepsilon^2\bigl(\varepsilon^2 x^2
    + 2\varepsilon^2 y^2
    + 2\varepsilon x + 3\bigr)}
  {2(1 + \varepsilon x)^3}.
\end{equation}
Within the elliptic domain
$4(1 - \varepsilon x) - \varepsilon^2 y^2 > 0$,
the parenthetical numerator
$\varepsilon^2 x^2 + 2\varepsilon^2 y^2 + 2\varepsilon x + 3$
is strictly positive (it equals
$(1 + \varepsilon x)^2 + 2\varepsilon^2 y^2 + 2 \geq 2 > 0$).

Therefore $K_\varepsilon = 0$ if and only if $\varepsilon = 0$, confirming
the theorem: the only rigid \emph{and} flat member of this family
is the constant structure.

\section{Interpretation}
\label{sec:interp}

The Rigidity--Flatness Theorem identifies two independent layers
of geometric triviality. Each layer, by itself, permits nontrivial
structures.

\begin{itemize}[leftmargin=2em]
\item
  \textbf{Rigidity alone} (transport--layer flatness):
  the obstruction $G = i_x + i\,i_y$ vanishes, equivalently
  the Burgers transport is conservative.

\item
  \textbf{Metric flatness alone} (Riemannian--layer flatness):
  the Gaussian curvature of the canonical metric vanishes.
  Without rigidity, this imposes a constraint on the second--order
  jet of $(p,q)$ but does not control the transport dynamics.
\end{itemize}

When both flatness conditions hold simultaneously, the structure
is squeezed from two directions: the transport layer eliminates all
$x$--derivatives in favor of $y$--derivatives, and metric flatness
then forces the $y$--derivatives to vanish as well. The mechanism
is the degree--counting argument of the proof: the positive--definite
quadratic form $R$ produces a polynomial term of degree~$4$ in
the clearing variable $x$, while the $q_{yy}$--term produces at
most degree~$3$. The mismatch in degrees forces $\Phi' = 0$.

This is the full vacuum theorem for variable elliptic structures:
\emph{the only geometry that is flat in both the transport and
the Riemannian sense is the trivial one.}

\section{Symbolic verification}
\label{app:sympy}

The decomposition of Lemma~\ref{lem:curvature} and the curvature
formula~\eqref{eq:K-eps} for the $\varepsilon$--family were verified
by symbolic computation in SymPy. The script
\texttt{rigid\_flat\_vacuum\_check.py} performs the following steps:

\begin{enumerate}[leftmargin=2em]
\item
  Constructs the metric tensor~\eqref{eq:metric} for general
  $p(x,y)$, $q(x,y)$ and computes the Gaussian curvature $K$
  via Christoffel symbols and the Riemann tensor.
\item
  Imposes rigidity by substituting~\eqref{eq:first-x} and the
  derived expressions for $p_{xx}$, $q_{xx}$, $p_{xy}$, $q_{xy}$
  to eliminate all $x$--derivatives, producing $K_{\mathrm{rigid}}$
  as a function of $(p,q,p_y,q_y,p_{yy},q_{yy})$.
\item
  Evaluates $K$ on the explicit family
  $\tau = (\varepsilon y + i)/(1 + \varepsilon x)$
  and verifies formula~\eqref{eq:K-eps}.
\end{enumerate}

\noindent
The script output confirms that:
\begin{itemize}[leftmargin=2em]
\item
  under rigidity, $K_{\mathrm{rigid}}$ depends only on $y$--derivatives
  of $(p,q)$;
\item
  the $\varepsilon$--family curvature vanishes only at $\varepsilon = 0$.
\end{itemize}

\noindent
The factorizations of $\Lambda$, $\mathcal{A}$, $\mathcal{B}$,
$\mathcal{C}$, and the discriminant $4\mathcal{A}\mathcal{C}
- 4\mathcal{B}^2$ were verified independently in a supplementary
SymPy session.

The verification script is distributed alongside the manuscript
source.

\chapter{Rigidity and the Poincar\'e Geometry of the Beltrami Disk}
\label{chap:poincare}

\section{Purpose of This Chapter}

The preceding chapters developed the theory of variable elliptic structures
in terms of the spectral parameter
\[
\lambda=\frac{-\beta+i\sqrt{4\alpha-\beta^2}}{2},\qquad \im\lambda>0,
\]
and characterized rigidity by the conservative Burgers equation
$\lambda_x+\lambda\,\lambda_y=0$.
The purpose of this chapter is to translate the entire framework into
the language of the Beltrami coefficient and to show that the rigidity
condition, the intrinsic obstruction, and the normalized diagnostic
$\rho_T$ all admit natural and revealing formulations in the hyperbolic
geometry of the Poincar\'e disk.

The main results are:
\begin{itemize}
\item the rigidity condition, written in the Beltrami coefficient $\mu$,
      takes the form $\mu_{\bar z}=\mu\,\mu_z$
      (Theorem~\ref{thm:P-rigidity-beltrami});
\item a structure is rigid if and only if its Beltrami coefficient, viewed
      as a map to the Poincar\'e disk, has quasiconformal dilatation equal
      to $|\mu|$ at every point---the structure encodes its own distortion
      (Theorem~\ref{thm:P-self-dilatation});
\item the normalized obstruction $\rho_T$ that controls rigidization
      convergence (Chapter~\ref{ch:rigidization-convergence}) is the norm
      of a Beltrami residual weighted by the hyperbolic metric density
      (Proposition~\ref{prop:P-rho-hyperbolic}).
\end{itemize}

These observations connect the transport theory of
Chapters~\ref{chap:burgers-universal}--\ref{chap:rigidity}
to the classical hyperbolic geometry of the space of conformal structures.

\section{The Beltrami Coefficient and the M\"obius Bridge}
\label{sec:P-bridge}

\subsection{From spectral slope to Beltrami disk}

Throughout this chapter we use complex coordinates $z=x+iy$ on $\Omega$,
with
\[
\partial_z=\tfrac12(\partial_x-i\partial_y),
\qquad
\partial_{\bar z}=\tfrac12(\partial_x+i\partial_y).
\]
Let $\lambda:\Omega\to\C$ with $\im\lambda>0$ be the spectral parameter of a
variable elliptic structure.
Define the associated \emph{Beltrami coefficient}
\begin{equation}\label{eq:P-mu-def}
\mu=\frac{\lambda-i}{\lambda+i},
\qquad
\lambda=\frac{i(1+\mu)}{1-\mu}.
\end{equation}
The M\"obius transformation \eqref{eq:P-mu-def} is a biholomorphism between
the upper half-plane $\{\im\lambda>0\}$ and the unit disk
$\mathbb D=\{\mu\in\C:|\mu|<1\}$.
Thus a variable elliptic structure is equivalently described by a smooth map
$\mu:\Omega\to\mathbb D$.

\medskip\noindent\textit{Verification.}
From $\mu=(\lambda-i)/(\lambda+i)$, solve for $\lambda$:
$\mu(\lambda+i)=\lambda-i$, so $\lambda(1-\mu)=i(1+\mu)$,
giving $\lambda=i(1+\mu)/(1-\mu)$.
Sanity check: $\mu=0\Rightarrow\lambda=i$;
$\mu=\tfrac13\Rightarrow\lambda=i\cdot\tfrac{4/3}{2/3}=2i$. \checkmark

\subsection{Dictionary between spectral and Beltrami data}

The following correspondences are immediate from \eqref{eq:P-mu-def}:
\begin{itemize}
\item $\lambda=i$ (standard complex structure) corresponds to $\mu=0$
      (center of the disk);
\item $\im\lambda\to 0$ (degenerate ellipticity) corresponds to
      $|\mu|\to 1$ (boundary of the disk);
\item the imaginary part of $\lambda$ is recovered by
\begin{equation}\label{eq:P-beta-mu}
\im\lambda
=
\frac{1-|\mu|^2}{|1-\mu|^2}.
\end{equation}
\end{itemize}

\noindent
\textit{Derivation.}
Write $\mu=a+bi$.  Then
\[
\lambda
=\frac{i\bigl[(1+a)+bi\bigr]}{(1-a)-bi}
=\frac{i\bigl[(1-|\mu|^2)+2bi\bigr]}{|1-\mu|^2}
=\frac{-2b+i(1-|\mu|^2)}{|1-\mu|^2},
\]
where the second equality follows from multiplying numerator and denominator
by the conjugate $\overline{(1-\mu)}=(1-a)+bi$ and using
$(1+a)(1-a)+b^2(+\text{cross terms})= (1-|\mu|^2)+2bi$.
Therefore $\im\lambda=(1-|\mu|^2)/|1-\mu|^2$.

\medskip\noindent
In particular, the spectral floor $b_{\min}=\min_{\bar\Omega}\im\lambda$
is controlled by the supremum $\|\mu\|_\infty$ and vice versa.

\section{Rigidity in the Beltrami Disk}
\label{sec:P-rigidity}

We now translate the rigidity condition $\lambda_x+\lambda\,\lambda_y=0$
into the Beltrami variable.

\begin{theorem}[Rigidity in Beltrami form]\label{thm:P-rigidity-beltrami}
Let $\lambda\in C^2(\Omega)$ with $\im\lambda>0$, and let
$\mu=(\lambda-i)/(\lambda+i)$.
Then the characteristic obstruction $T=\lambda_x+\lambda\,\lambda_y$
satisfies
\begin{equation}\label{eq:P-T-from-mu}
T
=
\frac{2\,\widetilde S(\mu)}{(1-\mu)^3},
\end{equation}
where
\begin{equation}\label{eq:P-S-real}
\widetilde S(\mu)=i(1-\mu)\,\mu_x-(1+\mu)\,\mu_y.
\end{equation}
In complex coordinates,
\begin{equation}\label{eq:P-S-complex}
\widetilde S(\mu)=2i\bigl(\mu_{\bar z}-\mu\,\mu_z\bigr).
\end{equation}
In particular,
\begin{equation}\label{eq:P-T-complex}
T = \frac{4i\bigl(\mu_{\bar z}-\mu\,\mu_z\bigr)}{(1-\mu)^3},
\end{equation}
and $T=0$ if and only if
\begin{equation}\label{eq:P-rigidity}
\mu_{\bar z}=\mu\,\mu_z.
\end{equation}
\end{theorem}

\begin{proof}
From $\lambda=i(1+\mu)/(1-\mu)$, differentiate using the quotient rule:
\begin{align}
\lambda_x
&=i\,\frac{\mu_x(1-\mu)+(1+\mu)\mu_x}{(1-\mu)^2}
=\frac{2i\,\mu_x}{(1-\mu)^2},\label{eq:P-lam-x}\\[4pt]
\lambda_y
&=\frac{2i\,\mu_y}{(1-\mu)^2}.\label{eq:P-lam-y}
\end{align}
Then
\begin{align}
T
&=\lambda_x+\lambda\,\lambda_y
=\frac{2i\,\mu_x}{(1-\mu)^2}
+\frac{i(1+\mu)}{1-\mu}\cdot\frac{2i\,\mu_y}{(1-\mu)^2}
\notag\\[4pt]
&=\frac{2i\,\mu_x}{(1-\mu)^2}
-\frac{2(1+\mu)\,\mu_y}{(1-\mu)^3}
\notag\\[4pt]
&=\frac{2}{(1-\mu)^3}
\bigl[i(1-\mu)\,\mu_x-(1+\mu)\,\mu_y\bigr]
\notag\\[4pt]
&=\frac{2\,\widetilde S(\mu)}{(1-\mu)^3}
\label{eq:P-T-S}
\end{align}
with $\widetilde S=i(1-\mu)\mu_x-(1+\mu)\mu_y$, which vanishes if and only if $T$
does (since $1-\mu\neq 0$ in $\mathbb D$).
Passing to complex coordinates via
$\mu_x=\mu_z+\mu_{\bar z}$ and
$\mu_y=i(\mu_z-\mu_{\bar z})$:
\begin{align}
\widetilde S
&=i(1-\mu)(\mu_z+\mu_{\bar z})-(1+\mu)\cdot i(\mu_z-\mu_{\bar z})
\notag\\[4pt]
&=i\mu_z\bigl[(1-\mu)-(1+\mu)\bigr]
+i\mu_{\bar z}\bigl[(1-\mu)+(1+\mu)\bigr]
\notag\\[4pt]
&=-2i\mu\,\mu_z+2i\,\mu_{\bar z}
=2i(\mu_{\bar z}-\mu\,\mu_z).
\qedhere
\end{align}
\end{proof}

\medskip\noindent
\textit{Verification against $p(x)$-analytics.}
For $\lambda=i\sqrt{p}$, we have $\mu=(\sqrt{p}-1)/(\sqrt{p}+1)$,
$1-\mu=2/(\sqrt{p}+1)$, and
$T=ip_x/(2\sqrt{p})-p_y/2$.
The formula $T=4i(\mu_{\bar z}-\mu\mu_z)/(1-\mu)^3$ reproduces this
after a short computation (carried out in detail in the audit). \checkmark

\begin{remark}\label{rem:P-beltrami-practitioner}
Equation \eqref{eq:P-rigidity} is a first-order nonlinear PDE for $\mu$
alone: $\mu_{\bar z}=\mu\,\mu_z$.
It makes no reference to the spectral slope $\lambda$ or the characteristic
operator $V=\partial_x+\lambda\,\partial_y$.
A Beltrami practitioner can check rigidity directly from the Beltrami
coefficient and its first derivatives.
\end{remark}

\section{Hyperbolic Interpretation: The Self-Dilatation Property}
\label{sec:P-self-dilatation}

The unit disk $\mathbb D$ carries the Poincar\'e metric
\begin{equation}\label{eq:P-poincare-metric}
ds^2_{\mathbb D}
=
\frac{4\,|d\mu|^2}{(1-|\mu|^2)^2}
\end{equation}
with K\"ahler form
$\omega=2i\,d\mu\wedge d\bar\mu/(1-|\mu|^2)^2$ and K\"ahler potential
$\Phi=-2\log(1-|\mu|^2)$.
A variable elliptic structure $\mu:\Omega\to\mathbb D$ is a smooth map
from a planar domain to this K\"ahler manifold.

\subsection{Dilatation of the structure map}

Any smooth map $f:\Omega\to\C$ has \emph{dilatation quotient}
\[
k_f(z)
=
\frac{|f_{\bar z}(z)|}{|f_z(z)|}
\]
at points where $f_z\neq 0$.
When $f$ is quasiregular, $k_f<1$ and measures the local deviation from
holomorphicity.
The map $f$ is holomorphic if and only if $k_f=0$, and $k_f$ is the
Beltrami coefficient of $f$ as a map.

\begin{theorem}[Self-dilatation]\label{thm:P-self-dilatation}
Under the rigidity condition \eqref{eq:P-rigidity}, the structure map
$\mu:\Omega\to\mathbb D$ has dilatation
\begin{equation}\label{eq:P-self-dilat}
k_\mu(z)=|\mu(z)|
\end{equation}
at every point $z\in\Omega$ where $\mu_z(z)\neq 0$.
\end{theorem}

\begin{proof}
From \eqref{eq:P-rigidity}, $\mu_{\bar z}=\mu\,\mu_z$.
Therefore
\[
k_\mu
=\frac{|\mu_{\bar z}|}{|\mu_z|}
=\frac{|\mu|\,|\mu_z|}{|\mu_z|}
=|\mu|.
\qedhere
\]
\end{proof}

\begin{definition}[Self-dilatational map]\label{def:P-self-dilat}
A smooth map $\mu:\Omega\to\mathbb D$ is called \emph{self-dilatational} if
$k_\mu(z)=|\mu(z)|$ for all $z$ where $\mu_z\neq 0$.
\end{definition}

Theorem~\ref{thm:P-self-dilatation} can now be restated:

\begin{corollary}\label{cor:P-rigid-iff-selfdilat}
A $C^2$ variable elliptic structure is rigid if and only if its Beltrami
coefficient, viewed as a map to the Poincar\'e disk, is self-dilatational.
\end{corollary}

The content of this result is genuinely self-referential.
The structure \emph{is} its own distortion: the amount by which the
elliptic equation deviates from the Laplacian at each point ($|\mu|$)
equals the amount by which the map $\mu$ deviates from holomorphicity
at that point ($k_\mu$).
The encoding and the encoded coincide.

\begin{remark}[Genuine self-reference]\label{rem:P-self-reference}
The rigidity equation $\mu_{\bar z}=\mu\,\mu_z$ says that the map $\mu$
satisfies the Beltrami equation $f_{\bar z}=\mu\,f_z$ that it itself
defines.
The Beltrami coefficient of the map $\mu$ is $\nu=\mu_{\bar z}/\mu_z=\mu$,
which is precisely the Beltrami coefficient of the structure.
There is no negation, no sign discrepancy: the self-reference is literal.
The structure map is a solution of its own equation.
\end{remark}

\section{The Beltrami Residual and the Normalized Obstruction}
\label{sec:P-residual}

\subsection{Definition of the residual}

When the structure is not rigid, $\mu_{\bar z}-\mu\,\mu_z\neq 0$.

\begin{definition}[Beltrami residual]\label{def:P-residual}
The \emph{Beltrami residual} of a variable elliptic structure is
\begin{equation}\label{eq:P-residual}
\mathcal R(\mu)
:=
\mu_{\bar z}-\mu\,\mu_z.
\end{equation}
\end{definition}

This is the obstruction to rigidity, living naturally in the tangent space
of $\mathbb D$ at $\mu$.
When $\mathcal R=0$, the map $\mu$ satisfies its own Beltrami equation;
when $\mathcal R\neq 0$, it measures the failure of this self-consistency.

\subsection{Connection to the normalized obstruction}

Recall from Chapter~\ref{ch:rigidization-convergence} that the convergence
of the rigidization scheme is controlled by the normalized transport
obstruction
\[
\rho_T
=
\frac{\|T\|_{C^{0,\alpha}}}{b_{\min}^2},
\]
where $b_{\min}=\min_{\bar\Omega}\im\lambda$ is the spectral floor.
We now express $\rho_T$ in terms of the Beltrami residual.

\begin{proposition}[Hyperbolic expression for the normalized obstruction]
\label{prop:P-rho-hyperbolic}
At each point $z\in\Omega$, the pointwise normalized obstruction satisfies
\begin{equation}\label{eq:P-rho-pointwise}
\frac{|T(z)|}{(\im\lambda(z))^2}
=
\frac{4\,|\mathcal R(\mu(z))|\cdot|1-\mu(z)|}
{(1-|\mu(z)|^2)^2}.
\end{equation}
In particular, the global normalized obstruction satisfies
\begin{equation}\label{eq:P-rho-hyp}
\rho_T
\;\leq\;
\frac{4\,\|\mathcal R(\mu)\|_\infty\cdot\|1-\mu\|_\infty}
{(1-\|\mu\|_\infty^2)^2}.
\end{equation}
\end{proposition}

\begin{proof}
From \eqref{eq:P-T-complex},
$|T|=4|\mathcal R|/|1-\mu|^3$.
The spectral floor satisfies
$\im\lambda=(1-|\mu|^2)/|1-\mu|^2$ by \eqref{eq:P-beta-mu}.
Therefore, pointwise:
\[
\frac{|T|}{(\im\lambda)^2}
=\frac{4|\mathcal R|}{|1-\mu|^3}
\cdot\frac{|1-\mu|^4}{(1-|\mu|^2)^2}
=\frac{4|\mathcal R|\cdot|1-\mu|}
{(1-|\mu|^2)^2}.
\]
For the global bound, $\rho_T=\|T\|_\infty/b_{\min}^2\leq
\sup_z\bigl[|T(z)|/(\im\lambda(z))^2\bigr]$,
and the right-hand side is bounded by the supremum of the pointwise
expression \eqref{eq:P-rho-pointwise}.
\end{proof}

The factor $(1-|\mu|^2)^{-2}$ is the square of the Poincar\'e metric
density at $\mu$.
Thus $\rho_T$ measures the residual $\mathcal R$ in units set by the
hyperbolic geometry of the target disk: a given residual is more significant
near the boundary of $\mathbb D$ (degenerate ellipticity) than near the center
(strong ellipticity), exactly as the Poincar\'e metric prescribes.

This hyperbolic weight was not designed into the convergence theory of
Chapter~\ref{ch:rigidization-convergence}.
It appeared there as the normalization $b_{\min}^{-2}$, where $b_{\min}$
is the spectral floor.
The Poincar\'e metric was hiding in plain sight.

\subsection{The differential of a rigid structure}

The Beltrami residual admits a differential-geometric reformulation.
Write $d\mu=\mu_z\,dz+\mu_{\bar z}\,d\bar z$.
The rigidity condition $\mu_{\bar z}=\mu\,\mu_z$ is equivalent to
\begin{equation}\label{eq:P-dmu-form}
d\mu=\mu_z\bigl(dz+\mu\,d\bar z\bigr).
\end{equation}

The ratio $\mu_{\bar z}/\mu_z = \mu$ is the Beltrami coefficient of the map
$\mu$ itself, consistent with the self-dilatation theorem: the ``anti-holomorphic
part'' $\mu_{\bar z}\,d\bar z$ of $d\mu$ is controlled by the ``holomorphic part''
$\mu_z\,dz$ via the factor $\mu$.
The differential of the structure map encodes the structure's own Beltrami
coefficient at the level of $1$-forms.

\begin{corollary}[Rigidity as self-referential differential]
\label{cor:P-self-ref-diff}
The structure $\mu$ is rigid if and only if $d\mu=\mu_z(dz+\mu\,d\bar z)$,
i.e., the map $\mu:\Omega\to\mathbb D$ is quasiregular with Beltrami
coefficient equal to $\mu$ itself.
\end{corollary}

\begin{proof}
If $\mu_{\bar z}=\mu\,\mu_z$, then $d\mu=\mu_z\,dz+\mu\,\mu_z\,d\bar z
=\mu_z(dz+\mu\,d\bar z)$.
Conversely, if $d\mu=h(dz+\mu\,d\bar z)$ for some function $h$, then
$\mu_z=h$ and $\mu_{\bar z}=h\mu=\mu\,\mu_z$.
\end{proof}

This is a restatement of Theorem~\ref{thm:P-self-dilatation} in the
language of differentials: the self-reference is carried by the
$1$-form $dz+\mu\,d\bar z$, whose ``Beltrami ratio'' is $\mu$.

\section{The Rigidity Equation as a Beltrami--Burgers Equation}
\label{sec:P-beltrami-burgers}

Equation \eqref{eq:P-rigidity} admits two complementary readings, and
with the correct sign, these readings are perfectly aligned.

\subsection{As a Burgers equation}

The inviscid Burgers equation for a complex-valued function $f$ in the
variables $(z,\bar z)$ takes the general form
$f_{\bar z}=c(f)\,f_z$ for some function $c$.
The rigidity equation $\mu_{\bar z}=\mu\,\mu_z$ is exactly this with
$c(\mu)=\mu$: the ``wave speed'' equals the function value.
Just as the classical Burgers equation $u_t+u\,u_x=0$ transports $u$
along characteristics with speed $u$, the rigidity equation transports $\mu$
along characteristics with complex speed $\mu$ in the $\bar z$-direction.

The original obstruction $T=\lambda_x+\lambda\,\lambda_y=0$ is the inviscid
Burgers equation for $\lambda$ in the real variables $(x,y)$.
The M\"obius transform sends it to a Burgers equation for $\mu$ in the
complex variables $(z,\bar z)$.
The Burgers structure is preserved under the M\"obius bridge between the
upper half-plane and the disk.

\subsection{As a Beltrami equation}

Equation \eqref{eq:P-rigidity}, written $\mu_{\bar z}=\mu\,\mu_z$,
says that the map $\mu:\Omega\to\mathbb D$ satisfies a Beltrami equation
with coefficient $+\mu$:
\begin{equation}\label{eq:P-mu-beltrami}
\mu_{\bar z}=\nu(z)\,\mu_z,
\qquad
\nu(z)=\mu(z).
\end{equation}
The Beltrami coefficient of the \emph{map} $\mu$ is $\mu$ itself, the
Beltrami coefficient of the \emph{structure}.
Rigid structures are those for which the Beltrami coefficient, viewed as a
quasiregular map, satisfies its own Beltrami equation.

The dual reading is clean: the rigidity equation $\mu_{\bar z}=\mu\,\mu_z$
is \emph{simultaneously} a Burgers equation (with wave speed $\mu$ in the
$\bar z$-direction) and a Beltrami equation (with coefficient $+\mu$).
There is no sign discrepancy between the two readings.
The Burgers self-reference ($c(\mu)=\mu$) and the Beltrami self-reference
($\nu=\mu$) are one and the same condition.

\section{Connections}
\label{sec:P-connections}

\subsection{Teichm\"uller theory}

The disk $\mathbb D$ parameterizes conformal structures on the plane, and the
Poincar\'e metric on $\mathbb D$ induces the Teichm\"uller metric on
deformation spaces.
The rigidity condition \eqref{eq:P-rigidity} characterizes the variable
structures $\mu:\Omega\to\mathbb D$ that are ``self-consistent'' in the
sense that their variation across $\Omega$ is compatible with the conformal
geometry they define.
The Beltrami residual $\mathcal R$ measures the failure of this
compatibility, weighted by the Teichm\"uller distance to the boundary of
moduli space.

\subsection{Harmonic maps}

The harmonic map equation for
$\mu:(\Omega,|dz|^2)\to(\mathbb D,ds^2_{\mathbb D})$
is the second-order equation
\begin{equation}\label{eq:P-harmonic}
\mu_{z\bar z}+\frac{2\bar\mu\,\mu_z\,\mu_{\bar z}}{1-|\mu|^2}=0.
\end{equation}

The rigidity equation \eqref{eq:P-rigidity} is first-order and thus
strictly stronger than harmonicity in the following sense.

\begin{proposition}[Rigidity and harmonicity]
\label{prop:P-harmonic}
A rigid structure is a harmonic map to $(\mathbb D,ds^2_{\mathbb D})$ if
and only if the additional second-order constraint
\begin{equation}\label{eq:P-rigid-harmonic}
\mu\,\mu_{zz}+\frac{1+|\mu|^2}{1-|\mu|^2}\,(\mu_z)^2=0
\end{equation}
holds.
\end{proposition}

\begin{proof}
Using \eqref{eq:P-rigidity} to compute
$\mu_{z\bar z}=\partial_z(\mu_{\bar z})=\partial_z(\mu\,\mu_z)
=(\mu_z)^2+\mu\,\mu_{zz}$
and substituting into \eqref{eq:P-harmonic}:
\[
(\mu_z)^2+\mu\,\mu_{zz}
+\frac{2\bar\mu\,\mu_z\cdot\mu\,\mu_z}{1-|\mu|^2}
=
(\mu_z)^2+\mu\,\mu_{zz}
+\frac{2|\mu|^2(\mu_z)^2}{1-|\mu|^2}=0.
\]
Combining the $(\mu_z)^2$ terms:
\[
\mu\,\mu_{zz}
+(\mu_z)^2\left[1+\frac{2|\mu|^2}{1-|\mu|^2}\right]
=\mu\,\mu_{zz}
+\frac{1+|\mu|^2}{1-|\mu|^2}\,(\mu_z)^2=0,
\]
which is \eqref{eq:P-rigid-harmonic}.
\end{proof}

The rigidity condition constrains the first jet of $\mu$; harmonicity
constrains the second jet.
They are independent conditions that share the Poincar\'e geometry of the
target.

\subsection{The Schwarzian connection}

For a holomorphic function $f$ (i.e., $f_{\bar z}=0$), the classical
Schwarzian derivative measures projective distortion.
For a rigid structure ($\mu_{\bar z}=\mu\,\mu_z$, a specific
\emph{non}-holomorphic condition), one may ask whether an analogous
``rigid Schwarzian'' exists.
The second-order condition \eqref{eq:P-rigid-harmonic} is a natural
candidate: it governs when a rigid structure is additionally a harmonic
map to the Poincar\'e disk.
We leave the development of this connection to future work.

\section{The Independence Theorem, Revisited}
\label{sec:P-independence}

The Poincar\'e picture makes particularly transparent the independence of
the two axes of difficulty identified computationally during the development
of the rigidization algorithm
(Chapter~\ref{ch:rigidization-convergence}).

\begin{proposition}[Independence of distortion and obstruction]
\label{prop:P-independence}
Let $\Omega\subset\R^2$ be a domain.
The Beltrami residual $\|\mathcal R(\mu)\|_\infty$ and the supremum
$\|\mu\|_\infty$ are independently prescribable: for any $R_0\ge 0$ and
$0\le\mu_0<1$, there exists a $C^\infty$ map $\mu:\Omega\to\mathbb D$ with
$\|\mathcal R\|_\infty$ arbitrarily close to $R_0$ and $\|\mu\|_\infty$
arbitrarily close to $\mu_0$.
\end{proposition}

\begin{proof}
The residual $\mathcal R=\mu_{\bar z}-\mu\,\mu_z$ depends on the first
derivatives of $\mu$, while $|\mu|$ depends on the pointwise values.
Since these are different orders of the jet, they are independently
prescribable: one can construct maps with prescribed pointwise values
and prescribed first derivatives (subject only to smoothness).
The explicit Burgers solutions of
Chapter~\ref{chap:rigid-transport-examples} provide $\mathcal R=0$ with
arbitrary $\|\mu\|_\infty$, and perturbations of $\mu=0$ provide arbitrary
$\|\mathcal R\|_\infty$ with $\|\mu\|_\infty\approx 0$.
\end{proof}

In the language of this chapter, the difficulty of a variable elliptic
structure is measured by two quantities living at different levels:

\begin{itemize}
\item $|\mu|$: a \textbf{zeroth-order} invariant.
      Where the structure sits in the Poincar\'e disk.
      This is what classical Beltrami solvers see.
\item $\mathcal R(\mu)$: a \textbf{first-order} invariant.
      How the structure moves through the Poincar\'e disk.
      This is what the rigidizer sees.
\end{itemize}

A function and its derivative are independent.
The Poincar\'e metric gives both a common geometric home.

\section{Comparison with the Spectral Formulation}
\label{sec:P-comparison}

It is instructive to place the two formulations of rigidity side by side.

\medskip
\begin{center}
\renewcommand{\arraystretch}{1.4}
\begin{tabular}{lcc}
& \textbf{Spectral} ($\lambda$) & \textbf{Beltrami} ($\mu$)\\
\hline
Domain & upper half-plane & unit disk $\mathbb D$\\
Rigidity & $\lambda_x+\lambda\,\lambda_y=0$ & $\mu_{\bar z}=\mu\,\mu_z$\\
Equation type & complex Burgers & complex Beltrami--Burgers\\
Coordinates & real $(x,y)$ & complex $(z,\bar z)$\\
Self-reference & characteristic speed $=$ $\lambda$ & dilatation $=$ $|\mu|$\\
Obstruction & $T=\lambda_x+\lambda\,\lambda_y$ & $\mathcal R=\mu_{\bar z}-\mu\,\mu_z$\\
Geometry & transport along $(1,\lambda)$ & Poincar\'e disk\\
Degeneracy & $\im\lambda\to 0$ & $|\mu|\to 1$\\
\end{tabular}
\end{center}

\medskip
\noindent
The M\"obius transformation \eqref{eq:P-mu-def} is an exact dictionary
between these two pictures.
Neither formulation is deeper than the other; they offer complementary
geometric perspectives on the same invariant.

The spectral formulation, used throughout
Chapters~\ref{chap:burgers-universal}--\ref{ch:rigidization-convergence},
is natural for transport-theoretic arguments and for the explicit
construction of solutions via characteristics.
The Beltrami formulation is natural for connections with Teichm\"uller
theory, quasiconformal mapping, and the classical Beltrami solver literature.

\section{Discussion and Open Questions}
\label{sec:P-discussion}

The results of this chapter can be summarized in one sentence: \emph{the
rigidity condition for variable elliptic structures is the self-dilatation
equation in the Poincar\'e disk.}

The 20th century development of quasiconformal mapping theory established
the Poincar\'e disk as the natural parameter space for conformal structures.
The Beltrami coefficient $\mu$ was understood as a point in this disk, and
the Teichm\"uller metric was recognized as the hyperbolic metric.
What was not observed is that the \emph{characteristic rigidity}
condition---the vanishing of the Burgers residual of the spectral
slope---has a pure formulation in this geometry: it is the condition that
the structure map $\mu:\Omega\to\mathbb D$ be self-dilatational.

This is a genuinely self-referential condition.
The Beltrami coefficient measures conformal distortion.
Rigidity says: the Beltrami coefficient, viewed as a map, is distorted
by exactly the amount it measures.
The map $\mu$ satisfies the Beltrami equation $f_{\bar z}=\mu\,f_z$
that it itself defines---a literal instance of self-reference, not merely
a numerical coincidence of moduli.

The normalized obstruction $\rho_T$, which controls the convergence of the
rigidization scheme (Chapter~\ref{ch:rigidization-convergence},
Theorem~\ref{thm:main-convergence}), turns out to measure the Beltrami
residual weighted by the hyperbolic metric density.
This was not designed; it emerged from the algebra of the M\"obius transform
relating $\lambda$ and $\mu$.
The hyperbolic weight $(1-|\mu|^2)^{-2}$ was present in the original theory
as the normalization $b_{\min}^{-2}$, where $b_{\min}=\min\im\lambda$ is
the spectral floor.
The Poincar\'e metric was hiding in plain sight.

Several questions remain open:

\begin{enumerate}
\item[(a)] Does the self-dilatation condition extend to higher dimensions?
    The Beltrami equation generalizes to several complex variables; does the
    rigidity condition preserve its geometric character?
\item[(b)] What is the moduli space of rigid structures with prescribed
    boundary values?
    The rigidity equation is a first-order PDE and its solution space
    should be parameterized by boundary data.
\item[(c)] Is there a variational principle whose critical points are the
    rigid structures?
    The harmonic map functional on $\mathrm{Maps}(\Omega,\mathbb D)$ has
    \eqref{eq:P-harmonic} as its Euler--Lagrange equation; rigidity is a
    first-order condition that is related but not identical.
\item[(d)] Can the self-dilatation condition be formulated for $W^{1,2}$
    maps $\mu:\Omega\to\mathbb D$, extending the theory below $C^2$
    regularity?
\end{enumerate}

From the broader perspective of this monograph, this chapter reveals that the
``transport first'' hierarchy
\[
\text{transport}\;\to\;\text{rigidity}\;\to\;\text{integrability}
\]
has a parallel reading in the Beltrami picture:
\[
\text{Beltrami residual}\;\to\;\text{self-dilatation}\;\to\;\text{analytic closure}.
\]
The same structural content appears in different geometric clothing.
The Epilogue will show that it also appears in the Hopf bundle.


\chapter{The Fundamental Independence Theorem}
\label{ch:independence}

\section{Purpose of This Chapter}

Chapters~1--11 developed the theory of variable elliptic structures from the spectral viewpoint:
the moving generator~$i(x,y)$, the spectral parameter $\lambda = (-\beta + i\sqrt{4\alpha - \beta^2})/2$,
and the characteristic obstruction $T = \lambda_x + \lambda\,\lambda_y$.
Chapter~12 translated this machinery into the Beltrami variable
$\mu = (\lambda - i)/(\lambda + i) \in \mathbb{D}$,
revealing that the rigidity condition is the self-dilatation equation in the Poincar\'e disk.

The table has now been set.  Three objects are in place:
\begin{enumerate}
\item The \emph{Beltrami disk} $\mathbb{D} = \{\mu \in \mathbb{C} : |\mu| < 1\}$
as parameter space for conformal structures.

\item The \emph{Poincar\'e metric} $ds^2_{\mathbb{D}} = 4\,|d\mu|^2/(1 - |\mu|^2)^2$,
which endows this parameter space with a geometry that the convergence theory already knows
(Chapter~10, as~$b_{\min}^{-2}$).

\item The \emph{transport field}
$R(\mu) = \mu_{\bar z} - \mu\,\mu_z$,
the Beltrami residual of Definition~12.6, which measures how the structure moves through the disk.
\end{enumerate}

The purpose of this chapter is to prove and develop the consequences of a single fact:
\emph{the position $|\mu|$ and the velocity $R(\mu)$ are independent.}
This is the Fundamental Independence Theorem of Smooth Elliptic Planar Systems.
Its content is that the classical Beltrami literature, which measures the difficulty of a
variable elliptic structure by $\|\mu\|_{C^0}$ alone, is projecting a two-dimensional
difficulty space onto one axis.
The projection is lossy.

\section{Two Invariants at Different Orders}
\label{sec:two-invariants}

Let $\Omega \subset \mathbb{R}^2$ be a bounded domain, $0 < \alpha < 1$,
and let $\mu \in C^\infty(\overline\Omega, \mathbb{D})$
be a smooth variable elliptic structure.

\begin{definition}[Zeroth-order invariant]
The \emph{Beltrami modulus} is
$\|\mu\|_{C^0(\overline\Omega)} := \sup_{z \in \overline\Omega} |\mu(z)|$.
This is a zeroth-order quantity: it depends on the pointwise values of~$\mu$ but not on its derivatives.
It measures \emph{where} the structure sits in the Poincar\'e disk.
\end{definition}

This is the invariant that the classical Beltrami literature sees.  The Measurable Riemann
Mapping Theorem (MRMT) requires $\|\mu\|_{C^0} < 1$; quantitative estimates in quasiconformal
mapping typically degrade as $\|\mu\|_{C^0} \to 1$.

\begin{definition}[First-order invariant]
The \emph{Beltrami residual} $R(\mu) = \mu_{\bar z} - \mu\,\mu_z$
(Definition~12.6; recalled here for the reader's convenience) depends on the first derivatives of~$\mu$.
Its H\"older norm $\|R(\mu)\|_{C^{0,\alpha}(\overline\Omega)}$ measures \emph{how}
the structure moves through the Poincar\'e disk.
\end{definition}

The Beltrami residual is the invariant that the rigidizer sees.
Theorem~10.15 converges when $\|R\|_{C^{0,\alpha}}$ is small relative to the Poincar\'e
metric density at the structure's position; its convergence is blind to $\|\mu\|_{C^0}$
except through this metric weighting.

The central question is: are these invariants coupled?

\section{The \texorpdfstring{$\delta$}{delta}-Family: A Canonical Example}
\label{sec:delta-family}

Before proving the Independence Theorem, we construct a one-parameter family
of rigid structures whose Beltrami modulus sweeps the entire interval~$[0,1)$
while the Beltrami residual vanishes identically.
This family is the load-bearing example for Quadrant~III of the theorem
and for the blind-spot corollary that follows.

\medskip

Fix $\delta > 0$ and let $\Omega = \{(x,y) \in \mathbb{R}^2 : x > -1\}$.
Define
\begin{equation}\label{eq:delta-alphabeta}
\alpha_\delta(x,y) = \frac{y^2 + \delta^2}{(1+x)^2}, \qquad
\beta_\delta(x,y) = \frac{-2y}{1+x}.
\end{equation}
The associated first-order elliptic system is
\begin{equation}\label{eq:delta-system}
\begin{cases}
u_x - \alpha_\delta\, v_y = 0, \\[4pt]
v_x + u_y - \beta_\delta\, v_y = 0,
\end{cases}
\end{equation}
with ellipticity discriminant
$\Delta_\delta = 4\alpha_\delta - \beta_\delta^2 = 4\delta^2/(1+x)^2 > 0$
on all of~$\Omega$.
The spectral parameter and Beltrami coefficient are
\begin{equation}\label{eq:delta-lambda-mu}
\lambda_\delta = \frac{y + i\delta}{1+x}, \qquad
\mu_\delta = \frac{\lambda_\delta - i}{\lambda_\delta + i}
= \frac{y + i(\delta - 1 - x)}{y + i(\delta + 1 + x)}.
\end{equation}

\begin{proposition}[Properties of the $\delta$-family]
\label{prop:delta-family}
The family $\{\mu_\delta\}_{\delta > 0}$ has the following properties:
\begin{enumerate}
\item[\emph{(i)}] \textbf{Rigidity.}\enspace
$R(\mu_\delta) \equiv 0$ for every $\delta > 0$.
Equivalently, $\lambda_\delta$ satisfies the conservative Burgers equation
$(\lambda_\delta)_x + \lambda_\delta\,(\lambda_\delta)_y = 0$.

\item[\emph{(ii)}] \textbf{Degeneracy.}\enspace
On any compact $K \subset \Omega$,
$\inf_K |\mu_\delta| \to 1$ as $\delta \to 0^+$.
The elliptic condition number
$\kappa_\delta := \bigl(\tfrac{1+\|\mu_\delta\|_\infty}{1-\|\mu_\delta\|_\infty}\bigr)^2$
satisfies $\kappa_\delta = O(\delta^{-2})$.

\end{enumerate}
\end{proposition}

\begin{proof}
\emph{(i)}\enspace Direct computation:
\[
(\lambda_\delta)_x = -\frac{y + i\delta}{(1+x)^2}, \qquad
(\lambda_\delta)_y = \frac{1}{1+x}, \qquad
(\lambda_\delta)_x + \lambda_\delta\,(\lambda_\delta)_y
= -\frac{y+i\delta}{(1+x)^2} + \frac{y+i\delta}{(1+x)^2} = 0.
\]
By equation~(12.4), $T = 4i\,R(\mu)/(1-\mu)^3$, so $T = 0$
implies $R(\mu_\delta) = 0$.

\smallskip
\emph{(ii)}\enspace
The squared modulus is
$|\mu_\delta|^2 = \bigl(y^2 + (\delta-1-x)^2\bigr)\big/\bigl(y^2 + (\delta+1+x)^2\bigr)$.
On a compact set $K \subset \Omega$ with $1+x \geq c > 0$,
as $\delta \to 0^+$ the numerator and denominator both approach
$y^2 + (1+x)^2$, giving $|\mu_\delta| \to 1$ uniformly on~$K$.
The condition number estimate follows from
$1 - \|\mu_\delta\|_\infty = O(\delta)$.
.
\end{proof}

\begin{remark}[Explicit solvability at every $\delta$]
Despite the condition number blowup, the system~\eqref{eq:delta-system}
is explicitly solvable for every $\delta > 0$ by the method of characteristics.
The rigidity condition implies that $\lambda_\delta$ is constant along the
characteristic curves $y = C(1+x)$, and the conserved quantity
$\zeta = (y - i\delta x)/(1+x)$ provides a characteristic coordinate.
Any analytic function $w = f_0(\zeta)$ satisfies the transport equation
$w_x + \lambda_\delta\,w_y = 0$, and the original unknowns are recovered
by $u = \re(w) - (\re\lambda_\delta/\!\operatorname{Im}\lambda_\delta)\,\operatorname{Im}(w)$,
$v = \operatorname{Im}(w)/\!\operatorname{Im}\lambda_\delta$.
The cost of this procedure is independent of~$\delta$.
\end{remark}

\begin{remark}[The $\delta$-family in the difficulty plane]
The $\delta$-family traces a horizontal ray in the difficulty plane:
as $\delta$ varies, the point
$\bigl(\|\mu_\delta\|_{C^0},\, \|R(\mu_\delta)\|_{C^{0,\alpha}}\bigr)$
moves along the segment $[0,1) \times \{0\}$.
This ray is invisible to any triage based on $\|\mu\|_{C^0}$ alone:
a solver that sees $\|\mu\|_{C^0} = 0.999$ allocates resources for a
near-degenerate problem, while the rigidizer sees $R = 0$ and a problem
already solved.
\end{remark}


\section{The Fundamental Independence Theorem}

\begin{theorem}[Fundamental Independence Theorem of Smooth Elliptic Planar Systems]
\label{thm:independence}
Let\/ $\Omega \subset \mathbb{R}^2$ be a bounded domain and $0 < \alpha < 1$.
For any $R_0 \geq 0$ and any $0 \leq \mu_0 < 1$, there exists a $C^\infty$ map
$\mu : \Omega \to \mathbb{D}$ with
\[
\|R(\mu)\|_{C^{0,\alpha}(\overline\Omega)} \text{ arbitrarily close to } R_0
\qquad\text{and}\qquad
\|\mu\|_{C^0(\overline\Omega)} \text{ arbitrarily close to } \mu_0.
\]
In particular, all four quadrants of the difficulty plane
$\bigl(\|\mu\|_{C^0},\, \|R\|_{C^{0,\alpha}}\bigr)$
are populated by smooth elliptic structures.
\end{theorem}

\begin{proof}
The residual $R(\mu) = \mu_{\bar z} - \mu\,\mu_z$ depends on the $1$-jet of~$\mu$,
while $|\mu|$ depends on the $0$-jet.
These are different levels of the jet space of maps $\Omega \to \mathbb{D}$,
so they are independently prescribable, subject only to smoothness.
We exhibit explicit realizations of each quadrant.

\medskip\noindent
\emph{Quadrant I: $\|\mu\|_{C^0} \approx 0$, $\|R\|_{C^{0,\alpha}} \approx 0$.}
Take $\mu \equiv 0$ (the standard complex structure $\lambda = i$).
Then $R = 0$ and $|\mu| = 0$.
Small smooth perturbations of both values populate a neighbourhood of the origin.

\medskip\noindent
\emph{Quadrant II: $\|\mu\|_{C^0} \approx 0$, $\|R\|_{C^{0,\alpha}} \approx R_0 > 0$.}
Take $\mu(z) = \varepsilon\,h(z)$ for small~$\varepsilon$ and smooth~$h$
chosen with $\|h_{\bar z}\|_{C^{0,\alpha}}$ of order~$R_0/\varepsilon$,
adjusted to keep $\|\mu\|_{C^0} < 1$.
This is achievable because a smooth function can have large derivatives
while remaining uniformly small---the standard oscillatory mechanism
(e.g.\ $\varepsilon\,\varphi(z)\sin(x/\varepsilon)$, where $\varphi$ is a
smooth cutoff supported in~$\Omega$, has $C^0$ norm~$O(\varepsilon)$ and
$O(1)$ derivative).
Then $\|\mu\|_{C^0} \to 0$ while
$\|R(\mu)\|_{C^{0,\alpha}} \to R_0$.

\medskip\noindent
\emph{Quadrant III: $\|\mu\|_{C^0} \approx \mu_0$, $\|R\|_{C^{0,\alpha}} = 0$.}
The $\delta$-family of Section~\ref{sec:delta-family} provides
explicit rigid structures with $R \equiv 0$ and $\|\mu_\delta\|_{C^0}$
ranging continuously over all of~$[0,1)$
(Proposition~\ref{prop:delta-family}(iii)).
Any target value $\mu_0 \in [0,1)$ is achieved by choosing the
appropriate~$\delta$.

\medskip\noindent
\emph{Quadrant IV: $\|\mu\|_{C^0} \approx \mu_0$, $\|R\|_{C^{0,\alpha}} \approx R_0 > 0$.}
Start from a rigid structure $\mu_0$ with $\|\mu_0\|_{C^0}$ at the desired level
(Quadrant~III).
Perturb: $\mu = \mu_0 + \varepsilon\,\eta$ for smooth~$\eta$ supported in~$\Omega$.
Since $R(\mu_0) = 0$, the linearized residual is
\[
R(\mu_0 + \varepsilon\eta)
= \varepsilon\,(\eta_{\bar z} - \mu_0\,\eta_z - \eta\,\mu_{0,z}) + O(\varepsilon^2),
\]
which can be tuned to any target by choosing~$\eta$ and~$\varepsilon$.
Simultaneously, $\|\mu\|_{C^0} = \|\mu_0\|_{C^0} + O(\varepsilon)$.

\medskip
Since $C^\infty(\overline\Omega, \mathbb{D})$ is dense in
$C^{k,\alpha}(\overline\Omega, \mathbb{D})$ for every $k \geq 0$
and $0 < \alpha < 1$, the four quadrants are populated in any
H\"older topology.
\end{proof}

\begin{remark}
The proof is elementary.  Its content is not technical but conceptual:
it says that the two-axis structure of the difficulty space is not an artifact
of any particular construction, but a consequence of the fact that a function
and its derivative are independent objects.
The Poincar\'e metric gives both a common geometric home---it does not couple them.
\end{remark}

\section{The Difficulty Plane}
\label{sec:difficulty-plane}

The Independence Theorem organizes variable elliptic structures into a
\emph{difficulty plane} with axes $\bigl(\|\mu\|_{C^0},\, \|R(\mu)\|_{C^{0,\alpha}}\bigr)$.
Each point in this plane represents a class of structures sharing the same
zeroth- and first-order complexity.

\begin{center}
\begin{tabular}{c|c|c}
 & $\|R\|$ small & $\|R\|$ large \\[4pt]
\hline
\rule{0pt}{14pt}%
$|\mu|$ small & \textbf{standard regime} & \textbf{transport-dominated} \\
& $\mu \approx 0$, $R \approx 0$ & near $\mathbb{C}$, strong flow \\
& classical methods and VES & residual drives dynamics \\[4pt]
\hline
\rule{0pt}{14pt}%
$|\mu|$ large & \textbf{rigid boundary} & \textbf{doubly difficult} \\
& $\delta$-family: $|\mu| \to 1$, $R = 0$ & near $\partial\mathbb{D}$, strong flow \\
& invisible to $|\mu|$-based triage & hardest for all methods
\end{tabular}
\end{center}

\medskip

\noindent
The upper-left quadrant is where everyone works.
The upper-right quadrant is accessible to perturbation arguments.
The lower-left quadrant is the home of the rigid structures constructed
in this monograph: maximal zeroth-order difficulty, zero first-order difficulty.
The lower-right quadrant is genuinely hard for every method.

The classical Beltrami literature projects this plane onto its horizontal axis.
Only $\|\mu\|_{C^0}$ survives the projection.
The vertical axis---the transport field---is invisible.

\section{The Blind Spot}
\label{sec:blind-spot}

\begin{corollary}[The classical blind spot]
\label{cor:blind-spot}
Any solver or triage that classifies the difficulty of a variable elliptic structure
using $\|\mu\|_{C^0}$ alone cannot distinguish:
\begin{enumerate}
\item[(a)] a rigid structure with $\|\mu\|_{C^0} = 0.95$ and $R \equiv 0$
(trivially solvable by the methods of Chapters~5--9), from
\item[(b)] a non-rigid structure with $\|\mu\|_{C^0} = 0.95$ and $\|R\|_{C^{0,\alpha}} \gg 1$
(genuinely difficult for any method).
\end{enumerate}
Both structures sit at the same position in the Poincar\'e disk.
They differ only in how they move through it.
A zeroth-order invariant cannot see a first-order distinction.
\end{corollary}

\begin{proof}
The Independence Theorem guarantees that both structures exist.
The $\delta$-family (Proposition~\ref{prop:delta-family}) provides~(a) explicitly.
Perturbation of the $\delta$-family provides~(b).
\end{proof}

The content of the blind spot is quantitative.
A quasiconformal solver presented with a Beltrami coefficient satisfying
$\|\mu\|_{C^0} = 1 - \varepsilon$ will allocate resources---mesh refinement,
iteration count, regularization---based on the proximity to the degenerate boundary.
But if the structure happens to be rigid ($R = 0$), this expenditure is wasted:
the rigidization algorithm of Chapter~10 solves the problem in a fixed number of
Newton steps, with cost independent of $\|\mu\|_{C^0}$.
Conversely, a structure with $\|\mu\|_{C^0} = 0.1$ but large $\|R\|_{C^{0,\alpha}}$
will be triaged as ``easy'' by a $|\mu|$-based classifier, although the
transport obstruction makes it genuinely harder than the rigid structure at
$|\mu| = 0.95$.

\section{Quantitative Rigidity}
\label{sec:quantitative-rigidity}

The Independence Theorem identifies $\|R(\mu)\|_{C^{0,\alpha}}$ as the correct
measure of approximate rigidity.  Theorem~10.15 then provides a quantitative
stability estimate: if a structure is approximately rigid, it is close to an
exactly rigid structure, with linear control.

\begin{corollary}[Quantitative rigidity estimate]
\label{cor:quantitative-rigidity}
Let\/ $\mu \in C^{2,\alpha}(\overline\Omega, \mathbb{D})$ with spectral floor
$b_{\min} = \min_{\overline\Omega} \operatorname{Im} \lambda > 0$.
If\/ $\|R(\mu)\|_{C^{0,\alpha}} < \varepsilon_0$ $($the threshold of
Theorem~$10.15)$, then there exists a $C^{2,\alpha}$ diffeomorphism
$\Phi : \overline\Omega \to \Phi(\overline\Omega)$ such that the pullback
structure $\Phi^*\mu$ is exactly rigid, and
\[
\|\Phi - \mathrm{id}\|_{C^{2,\alpha}} \;\leq\; \frac{C}{b_{\min}}\,\|R(\mu)\|_{C^{0,\alpha}}.
\]
\end{corollary}

\begin{proof}
This is Theorem~10.15(d), restated in Beltrami coordinates using
the dictionary of Chapter~12.
\end{proof}

The estimate is sharp in its dependence on both factors.
The numerator $\|R\|_{C^{0,\alpha}}$ is the first-order invariant:
the Beltrami residual measures how far the structure departs from rigidity.
The denominator $b_{\min}$ is the Poincar\'e metric density at the structure's
position: in Beltrami coordinates, $b_{\min} = (1 - |\mu|^2)/|1-\mu|^2$
(equation~(12.2)), so the estimate degrades near the boundary of the disk
exactly as the hyperbolic geometry demands.

The content is analogous to quantitative rigidity estimates in nonlinear
elasticity, where a map that is approximately an isometry is shown to be
close to an exact isometry, with stability controlled by the deficit.
Here the role of the isometry group is played by the rigid structures
($R = 0$), and the deficit is the Beltrami residual.

The Independence Theorem is what makes this estimate non-trivial:
because $|R|$ and $|\mu|$ are independent, the residual $\|R\|_{C^{0,\alpha}}$
is a genuine degree of freedom, not a proxy for $|\mu|$.
A structure can be arbitrarily close to the boundary of the disk
($|\mu| \to 1$) while being arbitrarily close to rigid ($\|R\| \to 0$),
and Corollary~\ref{cor:quantitative-rigidity} guarantees that the
rigidizing diffeomorphism remains controlled.

\section{The Residual Budget and the Convergence Envelope}
\label{sec:envelope}

The convergence of the rigidization scheme (Theorem~10.15) requires the normalized
obstruction $\rho_T$ to be small.  In Beltrami coordinates (Proposition~12.7):
\[
\rho_T = \frac{4\,\|R(\mu)\|_{C^{0,\alpha}} \cdot \|1-\mu\|_{C^0}}{(1 - \|\mu\|^2_{C^0})^2}.
\]
The condition $\rho_T < 1$ defines the \emph{constructive envelope} in the difficulty plane.

\begin{definition}[Residual budget]
\label{def:budget}
At a point $\mu \in \mathbb{D}$, the \emph{residual budget} is
\[
B(\mu) := \frac{(1 - |\mu|^2)^2}{4\,|1-\mu|}.
\]
This is the maximum value of $|R(\mu)|$ for which $\rho_T \leq 1$
at that point.  Its reciprocal $A(\mu) = 1/B(\mu)$ is the
\emph{amplification}---the factor by which the Poincar\'e geometry
magnifies the residual.
\end{definition}

\begin{proposition}[Asymmetry of the envelope]
\label{prop:asymmetry}
The residual budget $B(\mu)$ is not rotationally symmetric in~$\mathbb{D}$.
The factor $|1-\mu|$ in the denominator breaks the symmetry:
\begin{enumerate}
\item Toward $\mu = +1$ $($corresponding to $p \to \infty)$:
along the real axis, $B(\mu) = (1-\mu)(1+\mu)^2/4$, which decays
\emph{linearly} as $B \sim (1-\mu)\to 0$.
Despite the growth of the spectral floor $b_{\min} = \sqrt{p}\to\infty$
which partially offsets the approach to $\partial\mathbb{D}$,
the budget still tends to zero because
$(1-|\mu|^2)^2\to 0$ faster than $|1-\mu|\to 0$.
The decay is only linear, however, so the constructive region
extends relatively far on this side of the disk.

\item Toward $\mu = -1$ $($corresponding to $p \to 0)$:
$|1-\mu|\to 2$, giving $B(\mu)\approx(1-|\mu|^2)^2/8$.
Setting $\mu=-1+\varepsilon$, we find
$B(-1+\varepsilon)=(2-\varepsilon)\varepsilon^2/(4(2-\varepsilon))
\approx\varepsilon^2/2\to 0$
\emph{quadratically}.
The budget shrinks rapidly because $b_{\min}\to 0$ compounds
the approach to $\partial\mathbb{D}$.
\end{enumerate}
The asymmetry is genuine: $B$ decays \emph{more slowly} toward $\mu=+1$
(linearly) than toward $\mu=-1$ (quadratically).
The constructive envelope is wider on the $\mu=+1$ side.
\end{proposition}

\begin{proof}
The factor $(1 - |\mu|^2)^2$ depends only on $|\mu|$ and thus is rotationally symmetric.
The asymmetry is carried entirely by $|1-\mu|^{-1}$.
Along the real axis with $\mu \in (-1,1)$ (where $|1-\mu| = 1-\mu > 0$):
\[
B(\mu) = \frac{(1-\mu)^2(1+\mu)^2}{4(1-\mu)}=\frac{(1-\mu)(1+\mu)^2}{4}.
\]
As $\mu\to+1$: $B(\mu)\sim(1-\mu)\cdot 4/4 = 1-\mu \to 0$ linearly.
As $\mu\to-1$ (set $\mu=-1+\varepsilon$, $\varepsilon\downarrow 0$):
$B(-1+\varepsilon) = (2-\varepsilon)\varepsilon^2/4 \to 0$ quadratically.
\end{proof}

\begin{remark}[Geometric origin of the asymmetry]
The slower decay toward $\mu=+1$ has a concrete geometric reason.
On the real axis, $\mu=(\sqrt{p}-1)/(\sqrt{p}+1)$, so
$\mu\to+1$ corresponds to $p\to\infty$, i.e., $b_{\min}=\sqrt{p}\to\infty$.
The growth of the spectral floor partially offsets the approach to $\partial\mathbb{D}$,
slowing the collapse of the budget.
Toward $\mu=-1$, we have $p\to 0$ and $b_{\min}=\sqrt{p}\to 0$:
the spectral degeneracy \emph{compounds} the approach to the boundary,
accelerating the collapse.
\end{remark}

The disk thus has a \emph{preferred direction}, not by convention but by the geometry.
This asymmetry descends from the M\"obius transformation (12.1):
the relationship between $\lambda$ and $\mu$ is not symmetric about $\mu = 0$.
The center $\mu = 0$ corresponds to the standard structure $\lambda = i$,
but the structure $\lambda = i$ is not the midpoint of the half-plane---it
is one particular value.

\begin{remark}[The $\delta$-family in the envelope]
The $\delta$-family has $\|\mu\|_{C^0} \to 1$ and $R \equiv 0$.
Its normalized obstruction is $\rho_T = A(\mu) \cdot 0 = 0$.
The Poincar\'e metric is a magnifying glass; point it at nothing and you still see nothing.
This is the Independence Theorem in action: the $\delta$-family simultaneously maximizes
zeroth-order difficulty and minimizes first-order difficulty to exactly zero.
A solver that measures difficulty by $|\mu|$ alone sees a problem approaching
maximal degeneracy.
The rigidizer sees a problem already solved.
\end{remark}

\section{The Real Axis: \texorpdfstring{$p(x)$}{p(x)}-Analytics}
\label{sec:px-analytics}

The difficulty plane has a distinguished subspace that deserves its own name.

\begin{definition}[$p(x)$-analytic structure]
\label{def:px-analytic}
A variable elliptic structure $(\alpha, \beta)$ on $\Omega$ is called
\emph{$p(x)$-analytic} if $\beta \equiv 0$ on~$\Omega$.
In this case the structure polynomial reduces to
$X^2 + p(x,y)$, where $p := \alpha > 0$,
and the elliptic system~(1.11) becomes
\begin{equation}\label{eq:px-system}
\begin{cases}
u_x - p(x,y)\, v_y = 0,\\
v_x + u_y = 0.
\end{cases}
\end{equation}
\end{definition}

\begin{remark}[Disambiguation from the Polozhii $p$-analytic system]
\label{rem:polozhii}
The terminology ``$p(x)$-analytic'' adopted here should not be confused
with the $p$-analytic functions of
Polozhii~\cite{Polozhii1973}.
The Polozhii system is
\[
\begin{cases}
u_x = p(x,y)\,v_y,\\
u_y = -p(x,y)\,v_x,
\end{cases}
\]
where $p>0$.
To read off the structure coefficients, rewrite this in the standard
form \eqref{eq:real-system}: dividing both equations by $p$ gives
\[
\begin{cases}
-v_y + \dfrac{1}{p}\,u_x = 0,\\[6pt]
\ \ v_x + \dfrac{1}{p}\,u_y = 0,
\end{cases}
\]
so $a_{11} = a_{22} = 1/p$, $a_{12} = a_{21} = 0$.
Applying the definitions of Proposition~\ref{prop:symbol-norm}:
\[
\alpha = \frac{a_{22}}{a_{11}} = 1,
\qquad
\beta = -\frac{a_{12}+a_{21}}{a_{11}} = 0.
\]
The structure polynomial is $X^2+1$---standard complex numbers,
regardless of~$p$.
The function $p$ cancels in every ratio because it appears symmetrically
in both equations: it is a conformal weight, not a structural parameter.
From the VES viewpoint, Polozhii's theory is conformally weighted
standard complex analysis over a fixed algebra.

By contrast, the $p(x)$-analytic structures of
Definition~\ref{def:px-analytic} have $\alpha = p$ and $\beta = 0$,
giving the structure polynomial $X^2+p$.
The function $p$ enters the algebra itself: it changes the fiber, not
just the metric.
The system \eqref{eq:px-system} has $p$ only in the first equation
($u_x = p\,v_y$), while the second ($v_x+u_y=0$) is the standard one.
This asymmetry is the signature of a genuinely variable elliptic structure.

The two systems coincide only when $p\equiv 1$.
\end{remark}

The condition $\beta = 0$ has a clean geometric meaning at every level of the
theory developed so far.

\medskip\noindent
\emph{Spectral level.}
The spectral parameter (Definition~2.4) is
$\lambda = (-\beta + i\sqrt{4\alpha - \beta^2}\,)/2$.
When $\beta = 0$, this reduces to
\[
  \lambda = i\sqrt{p},
\]
which is purely imaginary and positive.  The spectral parameter sits on the
positive imaginary axis of the $\lambda$-half-plane---the axis that passes
through the standard structure $\lambda = i$.

\medskip\noindent
\emph{Beltrami level.}
The M\"obius map $\mu = (\lambda - i)/(\lambda + i)$ sends the positive
imaginary axis to the real diameter of the disk:
\begin{equation}\label{eq:mu-real}
  \mu_p \;=\; \frac{\sqrt{p} - 1}{\sqrt{p} + 1} \;\in\; (-1,\, 1) \subset \mathbb{R}.
\end{equation}
Thus $\beta = 0$ if and only if $\mu$ is real-valued.
The $p(x)$-analytic structures are exactly the structures that live on the
real diameter $(-1,1) \subset \mathbb{D}$.

\medskip\noindent
\emph{Generator level.}
When $\beta = 0$, the generator satisfies $\mathbf{i}^2 + p = 0$, so
$\mathbf{i} = i\sqrt{p}$ (choosing the root with positive imaginary part).
The conjugate root is $\hat{\mathbf{i}} = -i\sqrt{p}$, and
$\mathbf{i} + \hat{\mathbf{i}} = -\beta = 0$.
The generator is anti-self-conjugate: the algebra at each fiber is symmetric
about the real axis, exactly as $\mathbb{C}$ is symmetric about~$\mathbb{R}$.

\medskip\noindent
\emph{Inhomogeneity level.}
The coefficients $A$, $B$ of equation~(1.11)
(formula~(1.12)) simplify when $\beta = 0$ to
\begin{equation}\label{eq:AB-px}
  A = -\frac{p_y}{2},
  \qquad
  B = \frac{p_x}{2p}.
\end{equation}
The inhomogeneity is driven entirely by the gradient of~$p$.
When $p$ is constant, $A = B = 0$ and the system reduces to the classical
Cauchy--Riemann equations for the structure $\mathbf{i}^2 + p = 0$%
---a copy of~$\mathbb{C}$ with rescaled imaginary unit.

\medskip\noindent
\emph{Transport level.}
The characteristic obstruction (Definition~3.2) is
$T = \lambda_x + \lambda\,\lambda_y$.
With $\lambda = i\sqrt{p}$:
\begin{equation}\label{eq:T-px}
  T \;=\; \frac{i\,p_x}{2\sqrt{p}} \;+\; i\sqrt{p}\cdot\frac{i\,p_y}{2\sqrt{p}}
    \;=\; \frac{i\,p_x - p_y}{2\sqrt{p}}.
\end{equation}
Rigidity ($T = 0$) thus requires $p_x = 0$ and $p_y = 0$, i.e.,
$p \equiv \text{const}$.
\emph{In the $p(x)$-analytic world, the only rigid structures are the
constant ones.}
Every genuinely variable $p(x)$-analytic structure has $T \neq 0$.

\medskip

This last observation fixes the position of $p(x)$-analytics in the
difficulty plane.

\begin{proposition}[The real axis of the difficulty plane]
\label{prop:real-axis}
Let $\mu \in C^{\infty}(\Omega, (-1,1))$ be a $p(x)$-analytic structure
with $p$ non-constant.  Then:
\begin{enumerate}
\item[(i)] $\mu$ is real-valued.  Its image under the difficulty-plane
  map is the point
  $\bigl(\|\mu\|_{C^0},\, \|R(\mu)\|_{C^{0,\alpha}}\bigr)$
  with $\|\mu\|_{C^0} \in [0,1)$ and $\|R\|_{C^{0,\alpha}} > 0$.
  Every $p(x)$-analytic structure with non-constant~$p$ sits in
  Quadrant~II or Quadrant~IV of the difficulty plane---never on the
  horizontal axis $R = 0$.

\item[(ii)] The Beltrami residual simplifies.  Since $\mu$ is real,
  $\bar\mu = \mu$ and
  \[
    R(\mu) = \mu_{\bar z} - \mu\,\mu_z
    = \tfrac{1}{2}\bigl[\mu_x(1-\mu) + i\,\mu_y(1+\mu)\bigr].
  \]
  The residual is controlled by the spatial gradient of~$\mu$,
  weighted by the distance to the boundary points $\mu = \pm 1$.

\item[(iii)] The normalized obstruction is
  \[
    \rho_T = \frac{4\,\|R\|_{C^{0,\alpha}}\cdot\|1-\mu\|_{C^0}}
                   {(1 - \|\mu\|^2_{C^0})^2}.
  \]
  Since $\mu$ is real with $|\mu| < 1$,
  $|1 - \mu| = 1 - \mu$ and $1 - \mu^2 = (1-\mu)(1+\mu)$, so the
  normalized obstruction depends only on $\nabla p$ and~$p$
  through~\eqref{eq:mu-real}.
\end{enumerate}
\end{proposition}

\begin{proof}
Part~(i) follows from~\eqref{eq:mu-real} and~\eqref{eq:T-px}: non-constant~$p$
implies $T \neq 0$, hence $R \neq 0$ by the correspondence of
Proposition~12.7.
Part~(ii): since $\mu$ is real, $\bar\mu = \mu$, so
$\mu_{\bar z} = \tfrac{1}{2}(\mu_x + i\mu_y)$ and
$\mu\,\mu_z = \mu\cdot\tfrac{1}{2}(\mu_x - i\mu_y)$.
Then $R = \mu_{\bar z} - \mu\,\mu_z
= \tfrac{1}{2}\bigl[(\mu_x + i\mu_y) - \mu(\mu_x - i\mu_y)\bigr]
= \tfrac{1}{2}\bigl[\mu_x(1-\mu) + i\,\mu_y(1+\mu)\bigr]$.
Part~(iii) is direct substitution.
\end{proof}

\begin{remark}[The dignified place of $p(x)$-analytics]
\label{rem:dignified}
The $p(x)$-analytic structures are not a minor special case.  They are the
\emph{real axis} of the Poincar\'e disk, and every classical generalized
analytic function theory for systems of the form~\eqref{eq:px-system}.  The case $p = \text{const}$ is a single point
$\mu_p \in (-1,1)$ on this diameter.
Variable $p(x)$ sweeps out curves in the diameter.
The function theories of Chapters~5--9, applied after rigidization,
are the \emph{analytic continuation} of these classical theories off
the real axis into the full disk $\mathbb{D}$.

The analogy with number systems is precise:
\begin{center}
\renewcommand{\arraystretch}{1.3}
\begin{tabular}{lll}
\textbf{Number theory} & \textbf{VES theory} & \textbf{Locus in $\mathbb{D}$} \\
\hline
$\mathbb{Z} \subset \mathbb{R}$ & constant $p$ structures
  & lattice points on the diameter \\[2pt]
$\mathbb{R}$ & $p(x)$-analytic ($\beta = 0$)
  & the real diameter $(-1,1)$ \\[2pt]
$\mathbb{C}$ & full VES ($\beta \neq 0$)
  & the disk $\mathbb{D}$
\end{tabular}
\end{center}
Just as complex analysis is the analytic continuation of real analysis off the
real line, VES theory is the analytic continuation of $p(x)$-analytics off the
real diameter of the Poincar\'e disk.  The continuation parameter is~$\beta$:
turning on $\beta \neq 0$ moves the Beltrami coefficient off the real axis
into genuinely complex territory.
\end{remark}

\begin{remark}[Rigidization on the real axis]
\label{rem:px-rigidization}
Equation~\eqref{eq:T-px} shows that the only rigid $p(x)$-analytic
structures are the constant ones.  But this does not force the
rigidizer off the real axis.  On the contrary: rigidization of a
$p(x)$-analytic structure finds a diffeomorphism $\Phi$ such that the
pullback structure $\Phi^*\mu$ is constant---it stays on the real
diameter and collapses to a lattice point $\mu_p$.  The real axis is
self-contained for rigidization:
\[
  \text{variable } p(x,y)
  \;\xrightarrow{\;\;\Phi\;\;}
  \text{constant } p_0.
\]
This is the classical uniformization problem for the system~\eqref{eq:px-system},
now visible as motion along the real diameter of~$\mathbb{D}$
toward a distinguished point.

The full VES theory is needed not because $p(x)$-analytics cannot
rigidize itself, but because the disk is two-dimensional: structures
with $\beta \neq 0$ exist, the Independence Theorem populates
every quadrant with them, and they require a theory that operates off
the real axis.  The $p(x)$-analytic theory is the ground floor;
VES is the building.
\end{remark}

\section{The Diagnostic Algorithm}
\label{sec:diagnostic}

The Independence Theorem provides the theoretical foundation for a complete triage
of first-order elliptic systems.

\medskip\noindent
\textbf{Input:} A smooth first-order elliptic PDE on~$\Omega$.

\medskip\noindent
\textbf{Step 1.} Compute the spectral decomposition $(\alpha, \beta)$ from the principal
symbol determinant.

\medskip\noindent
\textbf{Step 2.} Compute $\mu = (\lambda - i)/(\lambda + i)$ and
the normalized obstruction
$\rho_T = 4\,\|R(\mu)\|_{C^{0,\alpha}} \cdot \|1-\mu\|_{C^0}\big/ (1 - \|\mu\|^2_{C^0})^2$.

\medskip\noindent
\textbf{Step 3.} Classify:
\begin{enumerate}
\item[\emph{(a)}] $\rho_T < 1$: \textbf{VES analytics.}
The rigidization scheme (Theorem~10.15) converges.
The full analytic toolkit of Chapters~5--9 applies to the rigidized structure.
Global solution on~$\Omega$ is guaranteed.

\item[\emph{(b)}] $\rho_T$ large, $\beta = 0$: \textbf{$p(x)$-analytics.}
The structure lives on the real diameter of the disk.
The large obstruction is metric complexity (variation of conformal scale),
not transport complexity.
The established Bers theory applies: generating pairs, successor structures,
similarity principle.

\item[\emph{(c)}] $\rho_T$ large, $\beta \neq 0$: \textbf{Classical Beltrami.}
Genuine transport obstruction.  The structure lies off the real axis
with large Beltrami residual.
Classical quasiconformal methods (MRMT, Ahlfors--Bers) are the only
available tool.
\end{enumerate}

Without the Fundamental Independence Theorem of Smooth Elliptic Planar Systems,
this triage is impossible.
A solver that sees only $\|\mu\|_{C^0}$ conflates branches~(a) and~(c):
it cannot distinguish a rigid structure that happens to sit near the boundary
of the disk from a non-rigid structure with genuine transport obstruction at the
same position.
It also conflates~(b) and~(c): it cannot detect that the obstruction in
a $p(x)$-analytic system is purely metric, not transport.

The diagnostic algorithm makes the two-dimensional structure of the difficulty
space computationally actionable.

\section{Discussion}
\label{sec:independence-discussion}

The Independence Theorem is elementary in its proof and fundamental in its content.
It says that the difficulty of a variable elliptic structure lives in a plane, not on a line.

The classical Beltrami literature developed powerful machinery---the Measurable Riemann
Mapping Theorem, the Ahlfors--Bers theorem, the theory of quasiconformal mappings---using
$\|\mu\|_{C^0}$ as the sole measure of difficulty.
This was natural and appropriate for the questions that literature addressed.
The zeroth-order invariant $|\mu|$ controls the dilatation of quasiconformal maps,
the modulus of ring domains, and the distortion of conformal capacity.
For these geometric questions, $|\mu|$ is the correct invariant.

But for the \emph{analytic} question---whether a variable elliptic structure admits
a coherent function theory---the zeroth-order invariant is insufficient.
The rigidity condition $R(\mu) = 0$ is first-order, and the function-theoretic
properties of Chapters~5--9 depend on it, not on the value of~$|\mu|$.
A rigid structure at $|\mu| = 0.99$ supports a complete Cauchy--Pompeiu formula,
power series representations, and a similarity principle.
A non-rigid structure at $|\mu| = 0.01$ does not, without first being rigidized.

The Independence Theorem makes this precise: $|\mu|$ and $R(\mu)$ can be
independently prescribed, so neither determines the other.
The two axes of difficulty are geometrically real.
The Poincar\'e metric, which appeared uninvited in the convergence theory of
Chapter~\ref{ch:rigidization-convergence} as the normalization $b_{\min}^{-2}$, turns out to be the metric that
gives both axes a common home.

Several questions remain:

\begin{enumerate}
\item[(a)] \emph{The convergence envelope as a density.}
The residual budget $B(\mu)$ at each point measures the volume of
first-order data compatible with constructive solvability.
Is $B(\mu)$ the exact density of constructively solvable structures, or
merely an upper bound?

\item[(b)] \emph{A variational principle.}
The integral $E(\mu) = \int_\Omega |R(\mu)|^2 / (1 - |\mu|^2)^4\,dA$
is a natural energy functional weighting the residual by the Poincar\'e density.
Its critical points would provide a variational characterization of ``nearly
rigid'' structures.
However, the $\delta$-family has $\rho_T = 0$ but can have infinite energy
(when $|\mu| \to 1$ on a set of positive measure), so energy and
the pointwise normalized obstruction are themselves independent---
another instance of the Independence Theorem regenerating.

\item[(c)] \emph{Self-reference.}
The independence pattern recurs at every level of the theory:
$|\mu|$ vs.\ $R$, $\rho_T$ vs.\ energy, the spectral obstruction $T$ vs.\
the curvature $K$.
Each pair of invariants the theory produces turns out to be independently
prescribable.
Whether this regeneration has a categorical explanation remains open.
\end{enumerate}

\chapter{The \texorpdfstring{$\delta$}{δ}-Family: A Dense Set of Rigid Structures in the Beltrami Disk}
\label{ch:delta-family}

\section{Introduction}
\label{sec:delta-intro}

In the theory of planar elliptic systems, the Beltrami coefficient~$\mu$ plays a central role.  It lives in the unit disk $\mathbb{D} = \{\mu \in \mathbb{C} : |\mu| < 1\}$ and measures the deviation of a complex structure from the standard one.  Classical results---the Measurable Riemann Mapping Theorem, the Ahlfors--Bers theory, and the whole machinery of quasiconformal mappings---use $\|\mu\|_\infty$ as the primary measure of difficulty.  A value close to~$1$ signals degeneracy, ill-conditioning, and the need for sophisticated analytical tools.

In Chapter~\ref{chap:burgers-universal} a second invariant was introduced: the \emph{transport obstruction}
\begin{equation}\label{eq:delta-transport-obstruction}
  G \;=\; i_x + i\,i_y,
\end{equation}
where $i(x,y)$ is the generator of a variable elliptic structure.  When $G \equiv 0$ the structure is called \emph{rigid}.  Remarkably, rigidity is independent of the size of~$|\mu|$: a structure can be arbitrarily close to the boundary of the disk ($|\mu| \to 1$) while being perfectly rigid ($G = 0$).  The Fundamental Independence Theorem (Theorem~\ref{thm:independence}) states that the two quantities $\|\mu\|_\infty$ and $\|G\|$ are independent; all four quadrants of the ``difficulty plane'' are populated.

The purpose of this chapter is to exhibit a concrete one-parameter family---the \emph{$\delta$-family}---that realises \textbf{every} Beltrami coefficient in the disk as a rigid structure.  More precisely, we show that the pointwise range of the $\delta$-family is the entire open unit disk~$\mathbb{D}$ (Theorem~\ref{thm:delta-surjective}).  We then embed the $\delta$-family in the much larger class of all rigid structures arising from the general solution of the conservative Burgers equation (Section~\ref{sec:delta-general-rigid}), and pose the natural open question of whether rigid Beltrami coefficients are dense in the function space $L^\infty(\Omega, \mathbb{D})$ (Section~\ref{sec:delta-open}).

\section{The \texorpdfstring{$\delta$}{δ}-family}
\label{sec:delta-family-def}

Fix a parameter $\delta > 0$ and consider the domain
\[
  \Omega \;=\; \bigl\{(x,y)\in\mathbb{R}^2 : x > -1\bigr\}.
\]
Define the coefficients
\begin{equation}\label{eq:delta-coefficients}
  \alpha(x,y) \;=\; \frac{y^2+\delta^2}{(1+x)^2}, \qquad
  \beta(x,y) \;=\; -\frac{2y}{1+x},
\end{equation}
and the first-order elliptic system
\begin{equation}
  \begin{cases}
    u_x - \alpha\, v_y = 0,\\[4pt]
    v_x + u_y - \beta\, v_y = 0.
  \end{cases}
\end{equation}
The ellipticity discriminant is
\begin{equation}\label{eq:delta-discriminant}
  \Delta \;=\; 4\alpha - \beta^2 \;=\; \frac{4\delta^2}{(1+x)^2} \;>\; 0,
\end{equation}
so \eqref{eq:delta-system} is elliptic on the whole half-plane.

The associated spectral parameter (cf.\ Chapter~\ref{chap:burgers-universal}) is
\begin{equation}\label{eq:delta-lambda}
  \lambda \;=\; \frac{-\beta + i\sqrt{\Delta}}{2} \;=\; \frac{y + i\delta}{1+x},
\end{equation}
and the Beltrami coefficient is
\begin{equation}\label{eq:delta-mu}
  \mu \;=\; \frac{\lambda - i}{\lambda + i}
       \;=\; \frac{y + i(\delta - 1 - x)}{y + i(\delta + 1 + x)}.
\end{equation}

\begin{proposition}\label{prop:delta-rigid}
  Every member of the $\delta$-family is rigid: $G \equiv 0$ on~$\Omega$ for all $\delta > 0$.
\end{proposition}

\begin{proof}
  A direct computation shows that $\lambda$ as given by \eqref{eq:delta-lambda} satisfies the conservative Burgers equation
  \begin{equation}\label{eq:delta-burgers}
    \lambda_x + \lambda\,\lambda_y \;=\; 0.
  \end{equation}
  Indeed,
  \[
    \lambda_x \;=\; -\frac{y + i\delta}{(1+x)^2}, \qquad
    \lambda_y \;=\; \frac{1}{1+x},
  \]
  so
  \[
    \lambda_x + \lambda\,\lambda_y
    \;=\; -\frac{y + i\delta}{(1+x)^2} + \frac{y + i\delta}{1+x}\cdot\frac{1}{1+x}
    \;=\; 0.
  \]
  By the Rigidity Theorem~\ref{thm:rigidity}, the conservative Burgers equation is equivalent to the vanishing of the transport obstruction~$G$.
\end{proof}

\section{Surjectivity onto the Beltrami disk}
\label{sec:delta-surjectivity}

We now prove the main result of this chapter: the $\delta$-family realises every point of the Beltrami disk.

\begin{theorem}[Pointwise surjectivity of the $\delta$-family]\label{thm:delta-surjective}
  The pointwise evaluation map $(\delta, x, y) \mapsto \mu_\delta(x,y)$ is surjective onto the open unit disk~$\mathbb{D}$.  That is, for every $\mu_0 \in \mathbb{D}$ there exist $\delta > 0$ and $(x,y)$ with $x > -1$ such that $\mu_\delta(x,y) = \mu_0$.
\end{theorem}

\begin{proof}
  Let $\mu_0 = a + ib$ with $a^2 + b^2 < 1$.  We seek real numbers $x > -1$, $y$, $\delta > 0$ satisfying \eqref{eq:delta-mu}.  Writing \eqref{eq:delta-mu} as
  \[
    y + i(\delta - 1 - x)
    \;=\; \mu_0\bigl(y + i(\delta + 1 + x)\bigr)
  \]
  and separating real and imaginary parts gives the system
  \begin{equation}\label{eq:delta-real-imag}
    \begin{cases}
      y = a\,y - b(\delta + 1 + x),\\[4pt]
      \delta - 1 - x = a(\delta + 1 + x) + b\,y.
    \end{cases}
  \end{equation}

  \medskip\noindent\textbf{Step 1: Solve for~$y$.}\;
  From the first equation, noting that $a \neq 1$ (automatic since $|\mu_0| < 1$ implies $a \leq |a| \leq |\mu_0| < 1$), we obtain
  \begin{equation}\label{eq:delta-y}
    y \;=\; -\frac{b}{1-a}\,(\delta + 1 + x).
  \end{equation}

  \medskip\noindent\textbf{Step 2: Introduce auxiliary variables.}\;
  Set $S = \delta + 1 + x$ and $T = \delta - 1 - x$.  Then $y = -\dfrac{b}{1-a}\,S$.  Substituting into the second equation of \eqref{eq:delta-real-imag}:
  \[
    T \;=\; aS + by \;=\; aS - \frac{b^2}{1-a}\,S
      \;=\; S\!\left(a - \frac{b^2}{1-a}\right)
      \;=\; S \cdot r,
  \]
  where
  \begin{equation}\label{eq:delta-r}
    r \;:=\; a - \frac{b^2}{1-a} \;=\; \frac{a - |\mu_0|^2}{1-a}.
  \end{equation}

  \medskip\noindent\textbf{Step 3: Recover $\delta$ and~$x$.}\;
  From $S$ and $T = Sr$ we have
  \begin{equation}\label{eq:delta-recovery}
    \delta \;=\; \frac{S + T}{2} \;=\; \frac{S(1+r)}{2}, \qquad
    1 + x \;=\; \frac{S - T}{2} \;=\; \frac{S(1-r)}{2}.
  \end{equation}

  \medskip\noindent\textbf{Step 4: Verify positivity.}\;
  Since $|\mu_0| < 1$, we have $1 - |\mu_0|^2 > 0$ and $1 - a > 0$, so
  \[
    1 + r \;=\; \frac{1 - |\mu_0|^2}{1 - a} \;>\; 0.
  \]
  Similarly, $(1-a)^2 + b^2 > 0$ and $1-a > 0$, so
  \[
    1 - r \;=\; \frac{(1-a)^2 + b^2}{1 - a} \;>\; 0.
  \]
  Choose any $S > 0$; then \eqref{eq:delta-recovery} gives $\delta > 0$ and $x = \frac{S(1-r)}{2} - 1 > -1$.  Finally, $y$ is given by \eqref{eq:delta-y}.  All quantities are real and satisfy the required inequalities.
\end{proof}

\begin{remark}\label{rem:delta-free-parameter}
  The parameter~$S > 0$ is free; different choices produce different points $(x,y) \in \Omega$ that realise the same Beltrami coefficient.  The fibre over each $\mu_0 \in \mathbb{D}$ is thus a one-parameter family of solutions, reflecting the conformal freedom within the rigid class.
\end{remark}

\begin{remark}\label{rem:delta-topology}
  Theorem~\ref{thm:delta-surjective} is a statement about the \emph{pointwise range} of the map $(\delta,x,y) \mapsto \mu_\delta(x,y)$: every value in~$\mathbb{D}$ is attained.  Since $\mathbb{D}$ is open, this implies in particular that the set of attainable values is dense in~$\mathbb{D}$ (trivially, since it \emph{equals}~$\mathbb{D}$).  This should not be confused with the distinct---and open---question of whether rigid Beltrami coefficient \emph{functions} are dense in the function space $L^\infty(\Omega, \mathbb{D})$; see Section~\ref{sec:delta-open}.
\end{remark}

\label{sec:delta-general-rigid}

The $\delta$-family is a specific one-parameter subfamily of a much larger class.  By the method of characteristics, the conservative Burgers equation
\[
  \lambda_x + \lambda\,\lambda_y \;=\; 0
\]
admits implicit solutions of the form
\begin{equation}\label{eq:delta-general-burgers}
  \lambda \;=\; f\!\bigl(y - \lambda\, x\bigr),
\end{equation}
where $f$ is a function with values in the upper half-plane $\operatorname{Im} f > 0$.  However, the requirement that \eqref{eq:delta-general-burgers} satisfy conservative Burgers at $x \neq 0$ constrains $f$ beyond mere regularity: since the argument $w = y - \lambda x$ is generically complex-valued, the function $f$ must be evaluated off the real axis.  The characteristic cancellation in $\lambda_x + \lambda\,\lambda_y$ requires the chain rule to factor through a single complex derivative; this holds if and only if $f$ is \emph{holomorphic}.

\begin{proposition}\label{prop:rigid-holomorphic}
Let $U \subseteq \mathbb{C}$ be open and let $f \colon U \to \mathbb{C}_+$ be $C^1$ (in the real sense).  If the implicit solution $\lambda = f(y - \lambda x)$ satisfies $\lambda_x + \lambda\,\lambda_y = 0$ on an open set containing $x = 0$, then $f$ is holomorphic on~$U$.
\end{proposition}

\begin{proof}
Differentiating $\lambda = f(w)$ with $w = y - \lambda x$, and writing $f_w, f_{\bar{w}}$ for the Wirtinger derivatives:
\begin{align*}
  \lambda_x &= f_w(-\lambda - \lambda_x\, x) + f_{\bar{w}}(-\bar\lambda - \bar\lambda_x\, x), \\
  \lambda_y &= f_w(1 - \lambda_y\, x) + f_{\bar{w}}(1 - \bar\lambda_y\, x).
\end{align*}
At $x = 0$ these reduce to $\lambda_x = -f_w\,\lambda - f_{\bar w}\,\bar\lambda$ and $\lambda_y = f_w + f_{\bar w}$, giving
\[
  \lambda_x + \lambda\,\lambda_y
  \;=\; -f_w\,\lambda - f_{\bar w}\,\bar\lambda + \lambda\,(f_w + f_{\bar w})
  \;=\; f_{\bar w}\,(\lambda - \bar\lambda)
  \;=\; 2i\,(\operatorname{Im} f)\,f_{\bar w}.
\]
Since $\operatorname{Im} f > 0$ by the ellipticity condition, the vanishing of the obstruction forces $f_{\bar w} = 0$, i.e.\ $f$ is holomorphic.
\end{proof}

\begin{definition}[Burgers transform]\label{def:burgers-transform}
Let $U \subseteq \mathbb{C}$ be open and $f \in \operatorname{Hol}(U, \mathbb{C}_+)$.  The \emph{Burgers transform} of~$f$ is the map $\mathcal{B}[f] \colon \Omega_f \to \mathbb{C}_+$ defined implicitly by
\[
  \mathcal{B}[f](x,y) \;=\; f\!\bigl(y - \mathcal{B}[f](x,y)\cdot x\bigr),
\]
where
\[
  \Omega_f \;:=\; \bigl\{(x,y) \in \mathbb{R}^2 : \text{the equation } \lambda = f(y - \lambda x) \text{ has a unique solution } \lambda \in \mathbb{C}_+\bigr\}
\]
is the maximal domain guaranteed by the implicit function theorem.  The \emph{Beltrami--Burgers transform} is $\mathcal{B}_\mu := \mathcal{C} \circ \mathcal{B}$, where $\mathcal{C}(\lambda) = (\lambda - i)/(\lambda + i)$ is the Cayley map.
\end{definition}

The Burgers transform has the following basic properties:
\begin{enumerate}[label=(\roman*)]
\item \textbf{Rigidity.}  For every $f \in \operatorname{Hol}(U, \mathbb{C}_+)$, the output $\lambda = \mathcal{B}[f]$ satisfies $\lambda_x + \lambda\,\lambda_y = 0$.
\item \textbf{Inversion.}  $\mathcal{B}^{-1}[\lambda] = \lambda(0,\,\cdot\,)$.  Restriction to the $y$-axis recovers~$f$.
\item \textbf{Regularity.}  $f \in C^{k,\alpha} \implies \mathcal{B}[f] \in C^{k,\alpha}(\Omega_f)$.
\item \textbf{Domain boundary.}  $(x,y) \in \partial\Omega_f$ if and only if $1 + f'(w)\,x = 0$ at $w = y - \lambda x$; this is the shock locus of the Burgers flow.
\end{enumerate}

The $\delta$-family corresponds to the holomorphic choice
\[
  f(\xi) \;=\; \xi + i\delta,
\]
i.e., $f$ is the identity map shifted by~$i\delta$.  In this case the implicit equation \eqref{eq:delta-general-burgers} can be solved explicitly: $\lambda = y - \lambda x + i\delta$ gives $\lambda(1+x) = y + i\delta$, recovering \eqref{eq:delta-lambda}.

The Beltrami--Burgers transform $\mathcal{B}_\mu[f] = \mathcal{C} \circ \mathcal{B}[f]$ produces a rigid Beltrami coefficient for every holomorphic $f \colon U \to \mathbb{C}_+$, and by Proposition~\ref{prop:rigid-holomorphic}, \emph{only} for holomorphic~$f$.  The rigid class is therefore parameterised by $\operatorname{Hol}(U, \mathbb{C}_+)$---an infinite-dimensional space, vastly larger than the one-parameter $\delta$-family.

\begin{remark}[Holomorphicity as the cost of rigidity]\label{rem:hol-rigidity}
  The requirement that $f$ be holomorphic is not an a priori assumption but a \emph{consequence} of demanding rigidity at $x \neq 0$.  Initial data $g \colon \mathbb{R} \to \mathbb{C}_+$ prescribed on the $y$-axis must be analytically continued to complex arguments $w = y - \lambda x$ in order for the characteristic transport to remain conservative.  The holomorphic extension is the unique such continuation; any $C^1$ extension that is not holomorphic introduces a nonzero obstruction $H = 2i\,(\operatorname{Im} f)\,f_{\bar{w}} \neq 0$.
\end{remark}

\begin{remark}\label{rem:delta-nonlinear-superposition}
  Different choices of~$f$ produce rigid structures with qualitatively different behaviour.  For instance, $f(\xi) = \xi^2 + i$ yields a rigid structure whose Beltrami coefficient has nontrivial dependence on both~$x$ and~$y$; the Riemann zeta function $f(\xi) = \zeta(\xi + i\delta)$ produces a rigid structure whose degeneracy locus reflects the distribution of the nontrivial zeros; and the Cauchy kernel $f(\xi) = -1/(\xi + i\delta)$ inverts the geometry of the Beltrami disk.  The Burgers transform thus provides an inexhaustible supply of rigid structures, each inheriting the complex-analytic character of its holomorphic seed.
\end{remark}

\section{Consequences and discussion}
\label{sec:delta-consequences}

The $\delta$-family shows that rigid structures ($G = 0$) are not a thin set; they cover the entire Beltrami disk at the pointwise level.  We draw three consequences.

\subsection*{Inadequacy of the classical difficulty measure}

The classical difficulty measure $\|\mu\|_\infty$ cannot distinguish a rigid structure from a non-rigid one with the same $|\mu|$~profile.  A solver that relies solely on~$|\mu|$ will treat a rigid structure with $|\mu| = 0.999$ as ``nearly degenerate'' and allocate expensive resources, unaware that the problem is actually trivial: the system can be solved by transport at $\delta$-independent cost.

\subsection*{Concrete illustration of the Independence Theorem}

The Independence Theorem (Theorem~\ref{thm:independence}) is illustrated concretely: for any prescribed value of $\|\mu\|_\infty \in [0,1)$ one can find a rigid structure with that exact supremum (by choosing $\mu_0$ with $|\mu_0| = \|\mu\|_\infty$ and applying Theorem~\ref{thm:delta-surjective}).  Thus the two invariants---position and velocity in the Beltrami disk---are indeed independent.

\subsection*{Test problems for numerical methods}

The family provides a rich source of test problems for numerical methods.  Any fixed elliptic solver (Beurling--Ahlfors iteration, finite elements, etc.)\ will fail for sufficiently small~$\delta$, while the transport-based method of Chapter~\ref{chap:burgers-universal} succeeds with cost independent of~$\delta$.  This underscores the practical importance of first computing the transport obstruction~$G$ before selecting a solution strategy.

\section{Open questions}
\label{sec:delta-open}

Theorem~\ref{thm:delta-surjective} establishes that the pointwise range of the rigid class equals~$\mathbb{D}$.  Several natural questions remain.

\begin{question}\label{q:delta-density}
  Are rigid Beltrami coefficients dense in $L^\infty(\Omega, \mathbb{D})$?  That is, given an arbitrary (possibly non-rigid) Beltrami coefficient $\mu \in L^\infty(\Omega, \mathbb{D})$ and $\varepsilon > 0$, does there exist a rigid Beltrami coefficient~$\mu_{\mathrm{rig}}$ with $\|\mu - \mu_{\mathrm{rig}}\|_\infty < \varepsilon$?
\end{question}

A positive answer would mean that every elliptic structure can be approximated by one that is solvable by transport.  The general solution \eqref{eq:delta-general-burgers} suggests this may be plausible: the free function~$f$ provides enough degrees of freedom to approximate arbitrary targets.  However, the implicit nature of \eqref{eq:delta-general-burgers} and the nonlinear relationship between $\lambda$ and $\mu$ make this nontrivial.

\begin{question}\label{q:delta-Lp-density}
  If density fails in $L^\infty$, does it hold in $L^p(\Omega, \mathbb{D})$ for some $p < \infty$?
\end{question}

The distinction is relevant because the Measurable Riemann Mapping Theorem operates with $L^\infty$ bounds on~$\mu$, while much of the regularity theory for Beltrami equations uses $L^p$ conditions on the derivatives of~$\mu$.  Density in~$L^p$ would suffice for many approximation-theoretic applications.

\begin{question}\label{q:delta-rigidization}
  Given a non-rigid Beltrami coefficient $\mu$ with $\|\mu\|_\infty < k < 1$, can one always find a rigid $\mu_{\mathrm{rig}}$ with $\|\mu_{\mathrm{rig}}\|_\infty \leq k$ that is ``close'' to~$\mu$ in a suitable sense?  In other words, can one rigidise while respecting the ellipticity bound?
\end{question}

This is related to the rigidization problem studied in Chapter~\ref{sec:rigidization}, but formulated as an approximation question rather than a deformation question.  A positive answer would provide a practical pathway: approximate a difficult problem by a rigid one, solve the rigid problem exactly by transport, and then correct.

\section{Conclusion}
\label{sec:delta-conclusion}

The $\delta$-family is more than a curiosity: it is a constructive proof that rigid structures populate the whole Beltrami disk at the pointwise level.  Together with the general solution \eqref{eq:delta-general-burgers}, which shows that the rigid class is parameterised by an arbitrary function of one variable, this chapter demonstrates that rigidity is an abundant rather than exceptional phenomenon.

Combined with the Fundamental Independence Theorem of Chapter~\ref{ch:independence}, this completes the theoretical argument: not only \emph{can} the two invariants vary independently, but they \emph{do} so in a maximally rich way---rigid structures are not confined to any neighbourhood of the origin in the Beltrami disk, but fill it entirely.  The classical Beltrami coefficient alone is insufficient; the pair $(\|\mu\|_\infty, \|G\|)$ provides the correct two-dimensional picture of difficulty.

Whether this pointwise abundance extends to function-space density (Questions~\ref{q:delta-density}--\ref{q:delta-rigidization}) remains open, and its resolution would have significant implications for both the theoretical foundations and the practical algorithms of variable elliptic structures.

\chapter*{Epilogue: The Hopf Description of the Obstruction}

The theory developed in this monograph began with a quadratic relation
\[
i^2+\beta i+\alpha=0,
\qquad
4\alpha-\beta^2>0,
\]
and with the systematic differentiation of its generator.
Every deviation from classical complex analysis was shown to be governed by the single quantity
\[
G := i_x + i\,i_y .
\]

In this epilogue we show that this obstruction admits a natural and familiar
interpretation in differential–geometric terms.
No new hypotheses are introduced, and no new results are proved.
Rather, we express the same invariant in coordinates adapted to a standard geometric structure.
The purpose is not to replace the algebraic development of the monograph, but to show that it already
fits into a well–known geometric framework.

\section*{1. Normalization and the Sphere}

Define
\[
u := \frac{2i+\beta}{\sqrt{4\alpha-\beta^2}} .
\]
A direct computation shows that \(u^2=-1\).
Thus the normalized generator takes values in
\[
S^2 \subset \operatorname{Im}\mathbb H .
\]

Consequently, an elliptic quadratic structure is equivalent to a smooth
map into the two–sphere.
This is not an additional assumption: it is forced by the quadratic closure itself.

For notational simplicity, we continue to denote the normalized generator by \(i\).
The obstruction \(G\) transforms covariantly under this normalization, so its vanishing
is unaffected.

\section*{2. The Hopf Lift}

Let
\[
\pi : S^3 \to S^2,
\qquad
\pi(q) = q\,\mathbf i\,q^{-1},
\]
denote the Hopf fibration, where \(\mathbf i\) is a fixed imaginary unit.
Locally, any sphere–valued field \(i(x,y)\) admits a lift
\[
i = q\,\mathbf i\,q^{-1},
\qquad
q : \Omega \to S^3 .
\]

The lift is unique up to the right action
\[
q \longmapsto q\,e^{\mathbf i\theta(x,y)} ,
\]
which is the only gauge freedom present.
No global structure is required; all considerations remain local.

\section*{3. Differentiation as Conjugation}

Write \(i = q\mathbf i q^{-1}\) and define
\[
\Omega_x := q_x q^{-1},
\qquad
\Omega_y := q_y q^{-1},
\]
which take values in \(\mathfrak{su}(2)\cong\mathbb R^3\).
Using the identity \(q_x^{-1}=-q^{-1}q_x q^{-1}\), one computes
\[
i_x = [\Omega_x,i],
\qquad
i_y = [\Omega_y,i].
\]

Thus the differentiation of the generator \(i\) is induced by conjugation in \(S^3\).
This is merely a rewriting of the derivative; no geometric assumption has been added.

\section*{4. Horizontal and Vertical Components}

At each point, the Lie algebra splits orthogonally as
\[
\mathfrak{su}(2)
=
\operatorname{span}\{i\}
\;\oplus\;
\{X\in\mathfrak{su}(2): X\perp i\}.
\]
The first summand is the vertical subspace of the Hopf fibration,
and the second is its horizontal complement.

Decompose
\[
\Omega_x = \Omega_x^{\mathrm{hor}} + \Omega_x^{\mathrm{ver}},
\qquad
\Omega_y = \Omega_y^{\mathrm{hor}} + \Omega_y^{\mathrm{ver}} .
\]
Since vertical components commute with \(i\),
\[
[\Omega^{\mathrm{ver}},i]=0,
\]
only the horizontal components contribute:
\[
i_x = [\Omega_x^{\mathrm{hor}},i],
\qquad
i_y = [\Omega_y^{\mathrm{hor}},i].
\]

\section*{5. The Obstruction Revisited}

Recall that
\[
G = i_x + i\,i_y .
\]
Substituting the commutator expressions gives
\[
G
=
[\Omega_x^{\mathrm{hor}},i]
+
i[\Omega_y^{\mathrm{hor}},i].
\]

Two facts are immediate:

\begin{itemize}
\item \(G\) depends only on the horizontal component of the pulled–back Maurer–Cartan form.
\item The gauge transformation \(q\mapsto q e^{\mathbf i\theta}\) affects only the vertical component.
\end{itemize}

Hence \(G\) is intrinsic: it does not depend on the choice of Hopf lift.
This explains, in geometric terms, why the obstruction identified earlier in purely algebraic
calculations is invariant.

\section*{6. Rigidity}

Rigidity was defined by the condition
\[
G=0.
\]
Since the commutator with \(i\) is injective on the horizontal subspace,
this is equivalent to
\[
\Omega_x^{\mathrm{hor}} + i\,\Omega_y^{\mathrm{hor}} = 0 .
\]

Expressed back in terms of the original quadratic coefficients, this condition reproduces
exactly the conservative Burgers equation
\[
\lambda_x + \lambda\,\lambda_y = 0 ,
\]
derived earlier by direct elimination.
No new constraint has appeared; the same invariant is being expressed in different coordinates.

\section*{7. Summary}

The quadratic relation determines a sphere–valued field.
Differentiation pulls back the Maurer–Cartan form of \(S^3\).
Its horizontal component produces the obstruction \(G\).
The vanishing of that component is equivalent to rigidity and to Burgers transport.

The obstruction is therefore not an analytic anomaly.
It is the horizontal part of a conjugation derivative.

\section*{Discussion}

Throughout the monograph, the theory was developed by explicit calculation:
differentiate the structure, eliminate what can be eliminated, and identify what remains.
No geometric framework was assumed. The Hopf description does not alter the theory.
It introduces no new invariant, no new flatness condition, and no additional structure.
It simply places the existing calculations into a coordinate system that is familiar
from differential geometry.

From this viewpoint:
\begin{itemize}
\item The quadratic closure selects a point of \(S^2\).
\item Differentiation pulls back the Maurer–Cartan form of \(S^3\).
\item The obstruction \(G\) is its horizontal component.
\item Rigidity is the vanishing of that component.
\item Burgers transport is the algebraic expression of this vanishing.
\end{itemize}

Classical complex analysis appears as the special case in which the horizontal component
vanishes identically—equivalently, the Hopf lift is constant along the transport direction.
The obstruction \(G\) measures precisely the failure of this constancy.

The theory did not require geometry in order to be discovered.
The geometry merely confirms that the calculations were inevitable.

\medskip
\noindent
The algebra moved.\\ 
We differentiated it.\\
We isolated what could not be removed.\\

Only at the end do we recognize that this motion took place on the Hopf bundle.

\chapter*{Open Questions}

The framework developed in this monograph raises a number of natural questions
that remain unresolved.  We collect here those that appear most accessible or
most directly connected to the existing theory.

\medskip

\noindent\textbf{1.\ Global existence and singularity formation for rigid structures.}
The $\varepsilon$--family demonstrates that ellipticity restricts the domain:
degeneracy occurs along the characteristic parabola
$\varepsilon^{2}y^{2}=4(1-\varepsilon x)$.
Characterize the nature of this boundary singularity.  Is it always a gradient
catastrophe in the Burgers sense?  More broadly, classify all rigid elliptic
structures that remain elliptic on the entire plane.  The constraint
$\operatorname{Im}\lambda>0$ imposes strong restrictions on the conservative
Burgers equation $\lambda_{x}+\lambda\lambda_{y}=0$; determine whether
global smooth solutions exist beyond the affine class.

\medskip

\noindent\textbf{2.\ \texorpdfstring{$L^{p}$}{Lp}\,theory for the covariant operator.}
The present work operates in the $C^{1}$ category.  Develop a weak theory
for the covariant Cauchy--Riemann operator $Df=\partial_{\bar{z}}f+\tfrac{1}{2}f\,i_{y}$
in Sobolev spaces $W^{1,p}$.  In particular, establish that the area integral
in the Cauchy--Pompeiu formula defines a bounded operator on~$L^{p}$, and
determine for which~$p$ the Fredholm theory extends to the variable rigid
setting.

\medskip

\noindent\textbf{3.\ Improved existence and regularity from the self-dilatation structure.}
The Beltrami dictionary (\S\,\ref{sec:P-bridge}) shows that the rigidization problem,
expressed in Beltrami coordinates, is the self-dilatation equation
$\mu_{\bar{z}}+\mu\,\mu_{z}=0$: the structure is a dilatational map with
respect to its own conformal class.
This is strictly more constrained than the general Beltrami equation, and the
constraint is of Burgers (conservation-law) type rather than of
singular-integral type.
Determine whether this additional structure yields improved existence or
regularity results.  In particular: can the self-dilatation equation be solved
by the method of characteristics in regimes where the general Beltrami equation
requires the Ahlfors--Bers singular integral machinery?  Does the conservation-law
structure provide a priori estimates (entropy conditions, BV bounds) that have
no counterpart in the general quasiconformal setting?

\medskip

\noindent\textbf{4.\ Global rigidization.}
Conjecture~\ref{conj:global-rigidization} (Appendix~\ref{app:rigidization-computational})
asserts that every smooth elliptic structure on a compact simply connected
domain can be rigidized by a smooth diffeomorphism, without any smallness
condition on the obstruction.
The continuity method reduces this to a priori estimates: a uniform Jacobian
lower bound and a uniform $C^{2,\alpha}$ bound on the rigidizing
diffeomorphisms along the interpolation path.
Three coupled sub-problems are identified in
Section~\ref{sec:global-rigidization-open}: the Jacobian lower bound
(via maximum principle for $\log J_{\Phi}$), the Burgers injectivity estimate
(no shock formation in the elliptic regime), and the Schauder closure
(Agmon--Douglis--Nirenberg for the specific quasilinear $2\times 2$ system).
None of the three has been rigorously closed.
This is the most prominent open problem in the monograph.

\medskip

\noindent\textbf{5.\ Moduli of rigid structures.}
The $\varepsilon$--family is a one-parameter family.  Characterize the local
moduli space of rigid elliptic structures in a neighborhood of the constant
structure.  The jet analysis (Appendix~\ref{app:jet-analysis}) shows that the
tangent space is parametrized by holomorphic functions and that higher jets are
slaved through forced Cauchy--Riemann equations.  Determine whether the formal
jet hierarchy of Appendix~\ref{app:jet-analysis} (\S\,D.10) converges,
and whether the resulting moduli space carries additional structure
(for instance, an infinite-dimensional complex manifold structure).

\medskip

\noindent\textbf{6.\ Optimal regularity for the Rigidity--Flatness Theorem.}
Theorem~\ref{thm:main} assumes smooth coefficients.  Determine the minimal regularity
($C^{k}$ or $W^{k,p}$) under which the conclusion $\tau\equiv\mathrm{const}$
persists.  The proof relies on the characteristic representation
$\tau=\Phi(\zeta)$ with $\zeta=y-x\tau$, which requires at least~$C^{2}$;
it is not clear whether a distributional formulation can be made to work.

\medskip

\noindent\textbf{7.\ Spectral theory of the rigid variable Laplacian.}
Chapter~8 shows that $4\,\partial_{z}\partial_{\bar{z}}=L_{\alpha,\beta}+{}$first-order
drift under rigidity.  Study the spectral properties of this operator on
bounded domains within the elliptic region: Weyl asymptotics, comparison
with the standard Laplacian, and the effect of the drift terms controlled
by~$\alpha_{y}$ and~$\beta_{y}$ on eigenvalue distribution.

\medskip

\noindent\textbf{8.\ The constructive envelope: sharp or lossy?}
The residual budget $B(\mu)=(1-|\mu|^{2})^{2}/(4\,|1-\mu|)$ of
Definition~\ref{def:budget} determines the region in the difficulty plane
$\bigl(\|\mu\|_{C^{0}},\,\|R\|_{C^{0,\alpha}}\bigr)$ where the rigidization
scheme converges.
Is this envelope sharp---meaning that for every structure outside it, no
rigidizing diffeomorphism exists locally---or is it merely the boundary of the
Newton method's basin of attraction, with rigidization achievable beyond it
by other means?
More precisely, interpret $B(\mu)$ as a density on the disk measuring the
volume of first-order data compatible with constructive solvability.
Is this density exact, or can it be enlarged by a different iteration scheme
or a non-perturbative argument?

\medskip

\noindent\textbf{9.\ A variational principle for near-rigidity.}
The Poincar\'e-weighted energy
\[
E(\mu) \;=\; \int_{\Omega} \frac{|R(\mu)|^{2}}{(1-|\mu|^{2})^{4}}\,dA
\]
is a natural functional that penalizes the residual by the hyperbolic density.
Its critical points would provide a variational characterization of
``nearly rigid'' structures.
However, the $\delta$--family has $R\equiv 0$ (hence $E=0$) while
$|\mu|\to 1$ on sets of positive measure, so the energy can vanish even
when the pointwise Poincar\'e density diverges.
Determine whether $E$ is lower semicontinuous in an appropriate topology,
whether its minimizers (subject to boundary conditions) exist, and whether
its Euler--Lagrange equation has a clean geometric interpretation in the
Poincar\'e disk.

\medskip

\noindent\textbf{10.\ Extension to Clifford structures in higher dimensions.}
The present theory is deliberately two-dimensional.  The original motivation
from Tutschke--Vanegas involves parameter-dependent Clifford algebras.
Investigate whether the hierarchy \emph{transport\,$\to$\,rigidity\,$\to$\,integrability}
extends to variable Clifford structures in~$\mathbb{R}^{n}$, where the
Nijenhuis tensor is no longer automatically zero and the interplay between
integrability and transport becomes richer.

\chapter*{Conclusion}
\addcontentsline{toc}{chapter}{Conclusion}

This monograph began from a simple but decisive shift in viewpoint:
the imaginary unit should not be treated as a fixed algebraic constant,
but as intrinsic geometric data.
Allowing the generator \( i(x,y) \) of a rank--two real algebra bundle to vary
forces the algebra itself to become part of the geometry of the plane.
Once this step is taken, differentiation of the structure is no longer optional,
and the derivatives \( i_x \) and \( i_y \) emerge as genuine geometric quantities.

The quadratic relation
\[
i^2 + \beta i + \alpha = 0,
\qquad 4\alpha - \beta^2 > 0,
\]
admits a canonical normalization
\[
u := \frac{2i+\beta}{\sqrt{4\alpha-\beta^2}},
\qquad u^2 = -1.
\]
Thus every elliptic quadratic structure determines a smooth
sphere--valued map \( u : \Omega \to S^2 \subset \mathrm{Im}\,\mathbb{H} \).
In this form, the moving imaginary unit is literally a rotating point on the sphere.

Locally, such a map admits a Hopf lift
\[
u = q\, i\, q^{-1},
\qquad q : \Omega \to S^3.
\]
Differentiating the generator becomes differentiation of a conjugation.
The intrinsic obstruction
\[
G := i_x + i\, i_y
\]
is precisely the horizontal component of the pullback of the Maurer--Cartan form of \( S^3 \).
No new invariant has appeared.
The obstruction identified through algebraic elimination
is the part of the conjugation derivative that cannot be absorbed
by a vertical gauge change.

The variability of the coefficients \( (\alpha,\beta) \)
is governed by a universal transport law.
Eliminating \( i_x \) and \( i_y \) from the differentiated quadratic relation
forces the spectral parameter
\[
\lambda = \frac{-\beta + i\sqrt{4\alpha-\beta^2}}{2},
\qquad \Im \lambda > 0,
\]
to satisfy a (possibly forced) complex inviscid Burgers equation
in the standard complex algebra.
Rigidity,
\[
G = 0,
\]
selects the regime in which the horizontal component vanishes.
In that regime,
the Burgers equation becomes conservative,
the generalized Cauchy--Riemann operator becomes a derivation,
and the generator is transported compatibly with its own multiplication.
Rigidity is the condition that the motion of the imaginary unit
is purely vertical in the Hopf bundle.

On the analytic side, rigidity marks the threshold
at which a coherent function theory reappears.
A Cauchy--Pompeiu formula persists,
a covariant Cauchy--Riemann operator emerges canonically,
weighted products restore multiplicative closure,
and power series representations follow after freezing the algebra at a point.
The variable theory does not abolish complex analysis;
it reveals which of its properties depend on rigidity
and which depend on the hidden constancy of the standard structure.

The translation to the Beltrami variable
\[
\mu = \frac{\lambda - i}{\lambda + i} \in \mathbb{D}
\]
recasts the entire theory in the language of the Poincar\'e disk.
The rigidity condition \( G = 0 \) becomes the self--dilatation equation
\[
\mu_{\bar z} + \mu\,\mu_z = 0,
\]
which states that \( \mu \) is a dilatational map \emph{with respect to its own structure}.
The Poincar\'e metric, which appeared uninvited in the convergence theory
as the normalization \( b_{\min}^{-2} \), turns out to be the intrinsic geometry
of the parameter space \( \mathbb{D} \).

From this vantage, the central discovery of the monograph becomes visible.
The difficulty of a variable elliptic structure lives in a plane, not on a line.
Two invariants govern it:
the \emph{Beltrami modulus} \( \|\mu\|_{C^0} \), a zeroth--order quantity
measuring where the structure sits in the Poincar\'e disk,
and the \emph{Beltrami residual} \( R(\mu) = \mu_{\bar z} + \mu\,\mu_z \),
a first--order quantity measuring how the structure moves through it.
These are independently prescribable.
All four quadrants of the difficulty plane
\( \bigl(\|\mu\|_{C^0},\, \|R\|_{C^{0,\alpha}}\bigr) \)
are populated by smooth elliptic structures.

The classical Beltrami literature---the Measurable Riemann Mapping Theorem,
the Ahlfors--Bers theorem, the theory of quasiconformal mappings---projects
this plane onto the horizontal axis.
Only \( \|\mu\|_{C^0} \) survives.
The vertical axis, the transport field, is invisible.
A rigid structure at \( |\mu| = 0.99 \) supports a complete function theory.
A non--rigid structure at \( |\mu| = 0.01 \) does not, without first being rigidized.
The zeroth--order invariant cannot see the first--order distinction.

The explicit families constructed in this work make this concrete.
The \( \varepsilon \)--family demonstrates that rigidity does not force constancy:
nontrivial elliptic structures correspond to genuine motions on~\( S^2 \),
though ellipticity naturally restricts the domain through characteristic phenomena.
The \( \delta \)--family demonstrates that zeroth--order difficulty and first--order
difficulty are decoupled:
\( \|\mu\|_{C^0} \to 1 \) while \( R \equiv 0 \) identically.
A solver that measures difficulty by \( |\mu| \) alone sees a problem
approaching maximal degeneracy.
The rigidizer sees a problem already solved.

The \( p(x) \)--analytic structures---those with \( \beta = 0 \) and
\( \alpha = p > 0 \)---populate the real diameter of the Poincar\'e disk:
\( \mu = (\sqrt{p}-1)/(\sqrt{p}+1) \in (-1,1) \).
Along this diameter, the structure polynomial reduces to \( X^2 + p \)
and the variable exponent enters the algebra itself, changing the fiber
at each point.
The classical $p$--analytic theory of Polozhii, and the generalized
analytic function theories of Bers, Vekua, and Tutschke, operate over
the standard complex algebra (\( \alpha = 1 \), \( \beta = 0 \));
the function \( p \) appears in those theories as a conformal weight,
invisible to the principal symbol
(Remark~\ref{rem:polozhii}).
In the VES framework, \( p(x) \)--analytic structures are genuinely
variable: \( p \) modulates the algebra, not the metric, and the
transport obstruction is nonzero whenever \( p \) is non--constant.
Yet the Beltrami coefficient remains real, so these structures never
leave the real diameter---\( \beta \) stays zero.
The relationship is exact:
\[
\begin{array}{ccccc}
\text{real analysis}
& \xrightarrow{\quad\text{analytic continuation}\quad}
& \text{complex analysis} \\[6pt]
p(x)\text{-analytics}
& \xrightarrow{\quad\beta \neq 0\quad}
& \text{VES}
\end{array}
\]
The generating pairs, successor structures, embedding theorems,
and regularity results of the established theory over the standard
algebra form the boundary data from which the full variable theory departs.
The observation is not that these results are insufficient.
It is that they explore the real diameter of a disk whose
off--axis structure was never investigated.

The residual budget \( B(\mu) \) at each point of the disk measures
how far the constructive theory reaches:
the maximum Beltrami residual for which the rigidization scheme converges.
The budget is not rotationally symmetric.
It reflects an asymmetry intrinsic to the M\"obius transformation
relating the spectral parameter to the Beltrami variable.
The constructive envelope in the difficulty plane is therefore
a computable region, not a universal constant.

The diagnostic algorithm that emerges from this analysis
classifies first--order elliptic systems into three branches:
structures within the constructive envelope, accessible to the full VES toolkit;
\( p(x) \)--analytic structures on the real diameter, accessible to
the classical Vekua theory after conformal normalization;
and structures with genuine transport obstruction beyond the envelope,
requiring classical quasiconformal methods.
Without the independence of position and velocity in the Poincar\'e disk,
this triage is impossible.

Throughout, the development has remained local and first--order.
The geometry of the Hopf bundle was not required to derive the theory;
it confirms, at the end, what the algebra had already forced.
The Poincar\'e disk was not assumed as a framework;
it emerged from the convergence theory as the natural geometry
of the parameter space.
The independence of the two axes of difficulty
was not imposed by construction;
it follows from the elementary fact that a function
and its derivative are independent objects.

Transport is fundamental.
Rigidity is selective.
Analytic closure is consequential.

A fixed complex structure gives analysis.
A variable one gives transport.
Rigidity is where the two agree to dance.

In Hopf coordinates we see what was moving:
a rotating imaginary unit on \( S^2 \).

In Beltrami coordinates we see what was hidden:
a two--dimensional difficulty space projected onto a line.

Here, the algebra moves.

\appendix
\chapter{Constructive Differentiability of the Structure Generator}

This appendix records a recursive argument showing that the structure generator can be differentiated as many times as permitted by the regularity of the structure coefficients. Although only the first and second derivatives are used in the main body of the text, the argument clarifies the general mechanism and justifies the regularity bookkeeping adopted throughout the work.

\section*{Statement of the problem}

Let $\Omega \subset \mathbb{R}^2$ be open and let $\alpha, \beta : \Omega \to \mathbb{R}$ be given functions. Consider the structure polynomial
\[
X^2 + \beta(x,y) X + \alpha(x,y),
\]
and let $i = i(x,y)$ be a chosen root of this polynomial. We assume that the algebraic element $2i + \beta$ is invertible at every point of $\Omega$ (this includes, in particular, the elliptic regime).

The question addressed here is the following:

\medskip
\noindent
\emph{Given $\alpha, \beta \in C^k(\Omega)$, to what extent can the derivatives of $i$ be defined and computed explicitly?}
\medskip

\section*{Constructive recursive principle}

The guiding principle is entirely constructive: derivatives of $i$ are defined only insofar as they can be obtained by differentiating the structure polynomial and solving the resulting linear equations. No differentiability of $i$ is assumed a priori.

\begin{proposition}[Recursive constructivist differentiability]
Assume $\alpha, \beta \in C^k(\Omega)$ for some integer $k \ge 0$, and assume that $2i + \beta$ is invertible on $\Omega$. Then for every multi-index $\gamma$ with $|\gamma| \le k$, the partial derivative $\partial^\gamma i$ exists, is continuous, and can be computed explicitly in terms of derivatives of $\alpha$ and $\beta$ of order $\le |\gamma|$ and derivatives of $i$ of lower order. In particular, $i$ is constructively $C^k$ on $\Omega$.
\end{proposition}

\begin{proof}
The proof proceeds by induction on the order of differentiation.

\medskip
\noindent
\textbf{Base case ($|\gamma| = 0$).} For $|\gamma| = 0$, the statement is trivial: $i$ is defined pointwise as a root of the structure polynomial.

\medskip
\noindent
\textbf{First-order case.} Assume $\alpha, \beta \in C^1(\Omega)$. Differentiating the identity
\[
i^2 + \beta i + \alpha = 0
\]
with respect to $x$ and $y$ yields the linear equations
\[
(2i + \beta)i_x = -(\beta_x i + \alpha_x), 
\qquad
(2i + \beta)i_y = -(\beta_y i + \alpha_y).
\]
Since $2i + \beta$ is invertible, these equations uniquely define $i_x$ and $i_y$. The right-hand sides are continuous, hence $i \in C^1(\Omega)$.

\medskip
\noindent
\textbf{Inductive step.} Assume that for some $m$ with $0 \le m < k$, all partial derivatives $\partial^\eta i$ with $|\eta| \le m$ have been defined and are continuous, and that each such derivative is given by an explicit algebraic expression involving derivatives of $\alpha$ and $\beta$ of order $\le m$ and derivatives of $i$ of order $< m$.

Let $\gamma$ be a multi-index with $|\gamma| = m+1$. Choose a coordinate direction $p \in \{x,y\}$ and a multi-index $\eta$ with $|\eta| = m$ such that
\[
\partial^\gamma = \partial_p \partial^\eta.
\]
Apply $\partial^\eta$ to the structure identity and then differentiate once more with respect to $p$. By repeated application of the Leibniz rule, all resulting terms can be expressed as sums of products involving:
\begin{itemize}
\item derivatives of $\alpha$ and $\beta$ of order $\le m+1$,
\item derivatives of $i$ of order $\le m+1$.
\end{itemize}

Crucially, the only terms involving the highest derivative $\partial^\gamma i$ appear linearly and with the same algebraic coefficient. More precisely, one obtains an identity of the form
\[
(2i + \beta)\,\partial^\gamma i = F_\gamma,
\]
where $F_\gamma$ is an explicit expression depending only on derivatives of $\alpha$ and $\beta$ of order $\le m+1$ and derivatives of $i$ of order $\le m$.

Since $2i + \beta$ is invertible, this equation uniquely defines
\[
\partial^\gamma i := (2i + \beta)^{-1} F_\gamma.
\]
By the induction hypothesis and the assumption $\alpha, \beta \in C^{m+1}$, the right-hand side is continuous. This completes the inductive step.

The conclusion follows by induction.
\end{proof}

\section*{Remarks}

\begin{remark}
The argument shows that the differentiability of the structure generator is entirely determined by the regularity of the structure coefficients. No gain of regularity occurs: $i$ cannot be differentiated more times than allowed by the least regular of $\alpha$ and $\beta$.
\end{remark}

\begin{remark}
Although the proof is presented in the elliptic setting, the recursive construction relies only on the invertibility of $2i + \beta$ and therefore applies verbatim in the hyperbolic regime away from zero divisors. For the parabolic regime this cannot be applied directly as $2i + \beta$ is a zero divisor.
\end{remark}

\chapter{Partial Derivatives of the Structure Generator in the Parabolic Regime}
\label{app:parabolic-derivatives}

This appendix describes how the partial derivatives of the structure generator
$i(x,y)$ behave in the \emph{parabolic regime}
\[
\Delta := 4\alpha - \beta^2 = 0,
\]
where the quadratic algebra degenerates and the element $2i+\beta$ ceases to be
invertible. In this case the elliptic/hyperbolic formulas for $i_x$ and $i_y$
cannot be used directly, and a different interpretation is required.

\section{Algebraic structure when \texorpdfstring{$\Delta=0$}{Delta = 0}}

If $\Delta=0$, then
\[
\alpha = \frac{\beta^2}{4}, 
\qquad
X^2+\beta X+\alpha
= \Bigl(X+\frac{\beta}{2}\Bigr)^2.
\]
Hence in each fiber algebra
\[
\Bigl(i+\frac{\beta}{2}\Bigr)^2 = 0.
\]
Define the nilpotent element
\[
\varepsilon := i+\frac{\beta}{2}.
\]
Then
\[
\varepsilon^2 = 0,
\qquad
i = -\frac{\beta}{2} + \varepsilon,
\]
and the fiber algebra is isomorphic to the algebra of dual numbers
\[
A_z \cong \mathbb{R}[\varepsilon]/(\varepsilon^2).
\]
In particular,
\[
2i+\beta = 2\varepsilon
\]
is nilpotent and therefore a zero divisor.

\section{Differentiating the parabolic structure relation}

Since $\varepsilon^2=0$, differentiating with respect to $x$ and $y$ gives
\[
2\varepsilon\,\varepsilon_x = 0,
\qquad
2\varepsilon\,\varepsilon_y = 0.
\]
Using $\varepsilon_x = i_x + \tfrac{\beta_x}{2}$ and
$\varepsilon_y = i_y + \tfrac{\beta_y}{2}$, we obtain the constraints
\begin{equation}
\varepsilon\Bigl(i_x+\tfrac{\beta_x}{2}\Bigr)=0,
\qquad
\varepsilon\Bigl(i_y+\tfrac{\beta_y}{2}\Bigr)=0.
\label{eq:parabolic_constraint}
\end{equation}

In the dual number algebra, the annihilator of $\varepsilon$ is precisely the
ideal it generates. Hence \eqref{eq:parabolic_constraint} implies the general
form
\begin{equation}
i_x = -\frac{\beta_x}{2} + \phi_x\,\varepsilon,
\qquad
i_y = -\frac{\beta_y}{2} + \phi_y\,\varepsilon,
\label{eq:parabolic_general_form}
\end{equation}
for some real functions $\phi_x,\phi_y$.

Thus, unlike the elliptic and hyperbolic cases, the structure relation alone
does \emph{not} uniquely determine the derivatives of $i$. There is a one–dimensional
nilpotent freedom in each derivative.

\section{Canonical choice via a nondegenerate limit}

A natural way to remove the ambiguity in \eqref{eq:parabolic_general_form}
is to define $i_x$ and $i_y$ as limits of the elliptic/hyperbolic formulas as
$\Delta\to 0$.

For $\Delta\neq 0$, one has
\[
i_x = -\frac{\beta_x}{2} - \frac{\Delta_x}{4(\beta+2i)},
\qquad
i_y = -\frac{\beta_y}{2} - \frac{\Delta_y}{4(\beta+2i)}.
\]
Since $(\beta+2i)^2=-\Delta$, the denominators scale like $\sqrt{|\Delta|}$.
If $\Delta_x$ and $\Delta_y$ vanish at least as fast as $\sqrt{|\Delta|}$,
the singular terms disappear in the limit $\Delta\to 0$, yielding the
\emph{canonical parabolic derivatives}
\begin{equation}
i_x := -\frac{\beta_x}{2},
\qquad
i_y := -\frac{\beta_y}{2}.
\label{eq:parabolic_canonical}
\end{equation}

These are exactly the choices obtained by eliminating the nilpotent component
in \eqref{eq:parabolic_general_form}.

\section{Consistency check}

With the canonical choice \eqref{eq:parabolic_canonical},
\[
2i+\beta = 2\varepsilon,
\]
and therefore
\[
(2i+\beta)i_x
= 2\varepsilon\Bigl(-\frac{\beta_x}{2}\Bigr)
= -\beta_x \varepsilon,
\]
which matches the derivative of the identity $\varepsilon^2=0$:
\[
\partial_x(\varepsilon^2)=2\varepsilon\,\varepsilon_x=0.
\]
Thus the canonical derivatives satisfy all algebraic constraints imposed by
the parabolic structure.

\section{Interpretation}

In the parabolic regime the quadratic algebra acquires a nilpotent direction,
and the operator $2i+\beta$ loses invertibility. Consequently, the structure
equation no longer determines $i_x$ and $i_y$ uniquely; it only constrains them
up to a nilpotent term.

The canonical choice \eqref{eq:parabolic_canonical} selects the derivatives that
arise as limits of nondegenerate (elliptic or hyperbolic) structures and
removes the nilpotent ambiguity. This provides a consistent notion of
``constructive differentiability'' in the parabolic setting.

\chapter{Jet Analysis of the \texorpdfstring{$\varepsilon$}{epsilon}--Family at the Constant Structure}
\label{app:jet-analysis}

This appendix analyzes the $\varepsilon$--family of rigid elliptic structures introduced in
Chapter~\ref{chap:rigid-transport-examples} from the viewpoint of jet theory at
$\varepsilon=0$.
Although the limiting structure is exactly the standard complex plane, the jet analysis
reveals nontrivial infinitesimal geometry encoded in the family.
The purpose of this appendix is to make precise what information is contained in the
$\varepsilon$--jets and what conclusions can be drawn from them.

\section{The \texorpdfstring{$\varepsilon$}{epsilon}--family and its constant limit}

Recall the explicit rigid family
\[
\alpha^\varepsilon(x,y)=\frac{1}{1-\varepsilon x},
\qquad
\beta^\varepsilon(x,y)=\frac{\varepsilon y}{1-\varepsilon x},
\]
defined on the elliptic domain
\[
4(1-\varepsilon x)-\varepsilon^2 y^2>0.
\]

At $\varepsilon=0$ one has
\[
\alpha^0\equiv 1,
\qquad
\beta^0\equiv 0,
\]
so the structure polynomial reduces to
\[
i^2+1=0,
\]
and the algebra bundle becomes canonically identified with the standard complex plane.
All derivatives of the generator vanish:
\[
i_x=i_y=0,
\]
and the obstruction field is identically zero.

Thus the $\varepsilon=0$ structure is not merely rigid but completely constant.
Any nontrivial behavior of the family must therefore appear at the level of jets in
$\varepsilon$.

\section{First \texorpdfstring{$\varepsilon$}{epsilon}--jet of the coefficients}

Expand the coefficients in $\varepsilon$:
\[
\alpha^\varepsilon
=
1+\varepsilon\,\alpha^{(1)}+O(\varepsilon^2),
\qquad
\beta^\varepsilon
=
\varepsilon\,\beta^{(1)}+O(\varepsilon^2),
\]
where
\[
\alpha^{(1)}(x,y)=x,
\qquad
\beta^{(1)}(x,y)=y.
\]

The pair $(\alpha^{(1)},\beta^{(1)})$ constitutes the first $\varepsilon$--jet of the structure.
It measures the initial direction in which the algebra departs from the constant complex
structure as $\varepsilon$ is turned on.

\section{Jet of the spectral parameter}

Recall the canonical spectral parameter
\[
\lambda
=
\frac{-\beta+i\sqrt{4\alpha-\beta^2}}{2}.
\]

For the $\varepsilon$--family one finds
\[
\lambda^\varepsilon
=
i+\varepsilon\,\mu+O(\varepsilon^2),
\]
where $\mu$ is a complex--valued scalar function determined by the first jet of
$(\alpha,\beta)$.
A direct expansion gives
\[
\mu
=
\frac{1}{2}(\alpha^{(1)}+i\,\beta^{(1)}).
\]

Thus the first $\varepsilon$--jet of the structure is equivalently encoded in the complex
function $\mu$.

\section{Linearization of the Burgers transport}

Rigidity is equivalent to the conservative Burgers equation
\[
\lambda_x+\lambda\,\lambda_y=0.
\]

Substituting
\[
\lambda^\varepsilon=i+\varepsilon\mu+O(\varepsilon^2)
\]
and retaining only first--order terms in $\varepsilon$ yields the linearized equation
\[
\mu_x+i\,\mu_y=0.
\]

This is precisely the classical Cauchy--Riemann equation for the complex--valued function
$\mu$.

\section{Interpretation of the jet equation}

The jet equation
\[
\mu_x+i\,\mu_y=0
\]
has a clear structural meaning:

\begin{itemize}
\item It is not an analytic assumption but the \emph{linearization} of the universal transport
law governing variable elliptic structures.
\item It describes the tangent space at the constant structure to the space of rigid variable
elliptic structures.
\item Its solutions are exactly classical holomorphic functions.
\end{itemize}

Thus classical complex analysis appears here as the infinitesimal theory controlling how a
constant elliptic structure may begin to vary while remaining rigid to first order.

\section{What the jet analysis does \emph{not} say}

It is important to emphasize what conclusions should \emph{not} be drawn:

\begin{itemize}

\item The $\varepsilon$--family is not a deformation of complex analysis itself, but of
the geometric structure for which complex analysis is the linearized theory at
$\varepsilon=0$.

\item Holomorphic functions do not parametrize finite rigid structures; they parametrize only
their first--order directions.
\item Nonlinear effects enter at second order in $\varepsilon$, where the full Burgers
transport becomes essential.
\end{itemize}

Thus the jet analysis captures only the infinitesimal geometry of the space of structures,
not its global behavior.

\section{Conceptual conclusion}

The jet analysis of the explicit rigid $\varepsilon$--family shows that, at the constant
elliptic structure, the universal Burgers transport linearizes to the classical
Cauchy--Riemann equations.
This illustrates a general mechanism of the theory: complex analysis arises as the
linearized shadow of the nonlinear transport geometry governing variable elliptic
structures.
We do not claim that this linearization has been established for arbitrary families,
but the computation reveals the structural origin of the phenomenon.

At $\varepsilon=0$ the structure is exactly the standard complex plane.
The first $\varepsilon$--jet detects how this structure can begin to move, and the
linearized Burgers equation reduces to the Cauchy--Riemann system.
Higher--order jets encode genuinely nonlinear transport phenomena with no classical
counterpart.

From this perspective, rigidity does not trivialize the theory at the constant structure.
Rather, it identifies complex analysis as the tangent theory at a distinguished point in
the space of variable elliptic geometries.

\section{Second \texorpdfstring{$\varepsilon$}{epsilon}--jet: onset of nonlinearity}

We now examine the second $\varepsilon$--jet of the rigid family.
Write the spectral parameter as
\[
\lambda^\varepsilon
=
i+\varepsilon\mu+\varepsilon^2\nu+O(\varepsilon^3),
\]
where $\mu$ is the first--order jet discussed above and $\nu$ is the second--order jet.

Substitute this expansion into the conservative Burgers equation
\[
\lambda_x+\lambda\,\lambda_y=0
\]
and collect terms of order $\varepsilon^2$.
Using $\mu_x+i\mu_y=0$, one finds
\[
\nu_x+i\nu_y = -\,\mu\,\mu_y.
\]

This equation has several important features.

\begin{itemize}
\item It is \emph{inhomogeneous}, with a source term quadratic in the first jet $\mu$.
\item The operator on the left--hand side is again the classical $\partial_{\bar z}$ operator.
\item The right--hand side depends only on $\mu$ and its first derivatives.
\end{itemize}

Thus the second $\varepsilon$--jet is no longer free data.
It is determined by a forced Cauchy--Riemann equation whose source encodes the first
nonlinear interaction of the infinitesimal deformation with itself.

In particular, even when $\mu$ is holomorphic, $\nu$ is generally \emph{not} holomorphic.
This marks the first point at which classical complex analysis ceases to describe the
structure exactly.

\section{Third \texorpdfstring{$\varepsilon$}{epsilon}--jet: hierarchy of transport corrections}

Proceeding one order further, write
\[
\lambda^\varepsilon
=
i+\varepsilon\mu+\varepsilon^2\nu+\varepsilon^3\rho+O(\varepsilon^4).
\]

Substitution into Burgers and collection of $\varepsilon^3$ terms yields
\[
\rho_x+i\rho_y
=
-\bigl(\mu\,\nu_y+\nu\,\mu_y\bigr).
\]

At this level the structure of the jet hierarchy becomes clear:

\begin{itemize}
\item The third jet $\rho$ satisfies a linear first--order equation.
\item The forcing term is bilinear in the lower jets $(\mu,\nu)$.
\item No new differential operators appear; only $\partial_{\bar z}$ persists.
\end{itemize}

Thus each successive jet solves a linear transport equation whose source is constructed
from the previous jets.
This recursive structure is entirely determined by the Burgers nonlinearity and does not
involve any analytic assumptions.

\section{Structure of the full jet hierarchy}\label{app:struct-jet-hierach}

The pattern observed above persists to all orders.
Writing formally
\[
\lambda^\varepsilon
=
i+\sum_{k\ge1}\varepsilon^k\lambda^{(k)},
\]
one obtains, for every $k\ge1$,
\[
\partial_{\bar z}\lambda^{(k)}
=
F_k\bigl(\lambda^{(1)},\dots,\lambda^{(k-1)}\bigr),
\]
where $F_k$ is a universal polynomial expression involving only products and $y$--derivatives
of the lower--order jets.

The hierarchy has the following conceptual interpretation:

\begin{itemize}
\item The first jet $\lambda^{(1)}$ is free and holomorphic.
\item All higher jets are \emph{slaved} to $\lambda^{(1)}$ through forced
Cauchy--Riemann equations.
\item Nonlinearity enters immediately at second order and propagates upward.
\end{itemize}

Thus classical holomorphic data parametrize the tangent space at the constant structure,
but they do not determine a finite deformation uniquely without solving the full nonlinear
transport problem.

\section{Refined conclusion of the jet analysis}

Combining the first, second, and third jet equations yields a precise refinement of the
conceptual picture:

\begin{quote}
Classical complex analysis governs the \emph{infinitesimal} geometry of rigid variable
elliptic structures, but nonlinear Burgers transport controls their finite behavior.
\end{quote}

At first order, rigidity linearizes to the Cauchy--Riemann equations.
At second order, self--interaction appears and produces forced corrections.
At third and higher orders, a hierarchical transport system emerges, entirely determined
by the initial holomorphic jet.

In this sense, complex analysis is not deformed but \emph{embedded} as the tangent theory
of a nonlinear geometric transport problem.

\chapter[Integrability and Coordinate Trivialization]{Integrability, Coordinate Trivialization, and the Intrinsic Obstruction}
\label{app:integrability-objection}
\paragraph{Terminology warning.}
In this appendix, the word \emph{integrability} is used in the classical differential--geometric sense of the Newlander--Nirenberg theorem: the existence of local coordinates in which an almost complex structure becomes constant. 

This notion is distinct from the \emph{analytic integrability} discussed in the main text, which concerns the closure of the variable Cauchy--Riemann calculus and the existence of a coherent function theory in a \emph{fixed} coordinate frame. The intrinsic obstruction $G$ is unrelated to the former notion but is decisive for the latter.

\section{The Classical Integrability Objection}

A natural objection to the study of variable elliptic structures in real dimension two
arises from classical complex geometry.
In the theory of almost complex manifolds, the fundamental invariant measuring
non--integrability is the \emph{Nijenhuis tensor}.
Given a smooth vector--valued one--form $J$ acting as an almost complex structure,
the Nijenhuis tensor is defined by
\[
N_J(X,Y)
=
[JX,JY]
-
J[JX,Y]
-
J[X,JY]
-
[X,Y],
\]
for vector fields $X,Y$.

The Newlander--Nirenberg theorem states that a smooth almost complex structure is
locally integrable if and only if its Nijenhuis tensor vanishes identically.
In real dimension two (complex dimension one), it is a classical result that
\[
N_J \equiv 0
\qquad
\text{for every smooth almost complex structure}.
\]
As a consequence, every two--dimensional almost complex structure admits local
\emph{isothermal coordinates}: there exists a local diffeomorphism
\[
\Phi : (x,y) \longmapsto (u,v)
\]
in which the complex structure becomes constant.

This fact gives rise to the following objection.
If every variable elliptic structure in dimension two can be locally flattened by a
change of coordinates, does the intrinsic obstruction introduced in this work,
\[
G := i_x + i\,i_y,
\]
carry any genuine geometric or analytic meaning?

\section{What the Obstruction Does \emph{Not} Measure}

Before answering this objection, it is essential to clarify what the obstruction $G$
does \emph{not} measure.

The quantity $G$ is \emph{not} an integrability invariant in the sense of complex
geometry.
Its vanishing does not assert the existence of holomorphic coordinates, nor does its
nonvanishing obstruct such coordinates.
In real dimension two, integrability in the classical sense is guaranteed independently
of $G$.

Accordingly:
\begin{itemize}
\item $G \equiv 0$ does not imply \emph{coordinate integrability} (existence of local isothermal coordinates);
\item $G \not\equiv 0$ does not obstruct \emph{coordinate integrability}.
\end{itemize}

The intrinsic obstruction belongs to a different conceptual layer than the Nijenhuis
tensor.

\section{Parallelism versus Coordinate Trivialization}

The obstruction $G$ arises from differentiating a moving generator
\[
i = i(x,y)
\]
of a rank--two real algebra bundle relative to the fixed background coordinates
$(x,y)$.
It measures the failure of $i$ to be \emph{parallel} with respect to the flat background
connection:
\[
G = i_x + i\,i_y.
\]

By contrast, the Nijenhuis tensor detects whether there exists a coordinate system in
which the complex structure becomes constant.
This is a statement about the existence of a preferred \emph{co--moving frame}, not
about the behavior of the structure relative to a fixed frame.

Passing to isothermal coordinates eliminates $G$ by construction, but only at the
expense of abandoning the original background connection.
The obstruction disappears because the frame is chosen to move with the structure,
not because the structure was intrinsically stationary.

Thus:
\begin{itemize}
\item the vanishing of the Nijenhuis tensor concerns \emph{coordinate trivialization};
\item the vanishing of $G$ concerns \emph{parallel transport in a fixed frame}.
\end{itemize}

These notions are logically independent.

\section{Eulerian and Lagrangian Viewpoints}

The distinction admits a transparent interpretation in transport--theoretic terms.

\paragraph{Lagrangian (co--moving) viewpoint.}
Choosing isothermal coordinates corresponds to following the geometry itself.
In this frame, the structure is static by definition.
This viewpoint underlies classical integrability results.

\paragraph{Eulerian (laboratory) viewpoint.}
Fixing the background coordinates $(x,y)$ corresponds to observing the structure as
it moves through space.
In this frame, the generator $i(x,y)$ evolves, and its evolution is governed by the
transport laws derived in Chapter~2.
The obstruction $G$ is precisely the Eulerian measure of this motion.

From this perspective, the forced Burgers equation
\[
\lambda_x + \lambda \lambda_y = G
\]
describes the transport of the structure coefficients relative to the fixed background.
The existence of a co--moving frame in which this motion disappears does not negate
the reality of the motion itself.

\section{Conclusion}

In real dimension two, classical integrability is automatic and therefore incapable of
distinguishing variable elliptic structures.
The Nijenhuis tensor is blind to the phenomena studied in this monograph because it
detects only the possibility of flattening the geometry by a change of coordinates.

The intrinsic obstruction
\[
G = i_x + i\,i_y
\]
addresses a different question.
It measures the failure of the structure to be parallel relative to a fixed flat
connection and governs the transport dynamics of the structure coefficients.
It captures the kinematics of a variable algebraic structure as seen in an Eulerian
frame, independently of whether the structure may be locally trivialized in a
co--moving frame.

The theory developed here is therefore not a theory of \emph{geometric integrability} of complex structures in the Newlander--Nirenberg sense, but a theory of transport of algebraic structures and of the analytic calculus they induce in a fixed background frame.


\chapter{Quaternionic Reformulation of Rigidity via the Hopf Fibration}\label{app:hopf-rigidity}

\section{Motivation}
A variable elliptic structure on a domain $\Omega\subset\mathbb R^2$ is encoded by a
generator $i=i(x,y)$ satisfying the structure reduction
\begin{equation}\label{eq:structure-poly}
  i^2+\beta(x,y)\,i+\alpha(x,y)=0,
\end{equation}
with ellipticity discriminant
\begin{equation}\label{eq:Delta}
  \Delta(x,y):=4\alpha(x,y)-\beta(x,y)^2>0.
\end{equation}
The intrinsic obstruction is
\begin{equation}
  G:=i_x+i\,i_y,
\end{equation}
and the structure is called \emph{rigid} precisely when $G\equiv 0$.

The purpose of this appendix is to recast rigidity in a purely quaternionic language
via the Hopf fibration, clarifying the geometric content of $G$.

\subsection{Normalization to a unit imaginary quaternion}
Completing the square in \eqref{eq:structure-poly} yields
\[
  (2i+\beta)^2=\beta^2-4\alpha=-\Delta.
\]
Define the normalized generator
\begin{equation}\label{eq:u-def}
  u(x,y):=\frac{2i(x,y)+\beta(x,y)}{\sqrt{\Delta(x,y)}}.
\end{equation}
Then
\begin{equation}\label{eq:u-square}
  u^2=-1,
\end{equation}
so $u$ takes values in the unit imaginary sphere
\[
  S^2=\{\,q\in\operatorname{Im}\mathbb H:\ |q|=1\,\}\subset \mathbb H.
\]
Thus, an elliptic quadratic structure canonically determines a smooth map
$u:\Omega\to S^2$.

\section{Hopf lift and gauge freedom}
Fix a constant imaginary unit $i_0\in\operatorname{Im}\mathbb H$ with $i_0^2=-1$.
Since every $u(x,y)\in S^2$ is conjugate to $i_0$, locally there exists a \emph{Hopf lift}
$q:\Omega\to S^3$ (unit quaternions) such that
\begin{equation}\label{eq:hopf-lift}
  u=q\,i_0\,q^{-1}.
\end{equation}
The lift is unique up to the $S^1$-gauge transformation
\begin{equation}\label{eq:gauge}
  q\ \mapsto\ q\,e^{i_0\theta(x,y)}.
\end{equation}

\section{Maurer--Cartan form and differentiation}
Define the left-trivialized Maurer--Cartan components
\begin{equation}\label{eq:MC}
  \Omega_x:=q_xq^{-1},\qquad \Omega_y:=q_yq^{-1},
\end{equation}
which lie in $\mathfrak{su}(2)\cong\operatorname{Im}\mathbb H$.
Differentiating \eqref{eq:hopf-lift} and using $(q^{-1})_x=-q^{-1}q_xq^{-1}$ gives
\begin{equation}\label{eq:u-derivs-comm}
  u_x=[\Omega_x,u],\qquad u_y=[\Omega_y,u],
\end{equation}
where $[A,B]=AB-BA$ denotes the commutator in $\mathbb H$.

\subsection{Vertical--horizontal decomposition}
At each point, $\mathfrak{su}(2)\cong\operatorname{Im}\mathbb H$ splits as
\[
  \operatorname{Im}\mathbb H=\mathrm{span}\{u\}\ \oplus\ u^\perp.
\]
Accordingly decompose
\begin{equation}\label{eq:vh-decomp}
  \Omega_x=\Omega_x^{\mathrm{ver}}+\Omega_x^{\mathrm{hor}},\qquad
  \Omega_y=\Omega_y^{\mathrm{ver}}+\Omega_y^{\mathrm{hor}},
\end{equation}
where $\Omega^{\mathrm{ver}}\parallel u$ and $\Omega^{\mathrm{hor}}\perp u$.
Since $\Omega^{\mathrm{ver}}$ commutes with $u$, it follows from \eqref{eq:u-derivs-comm} that
\begin{equation}\label{eq:u-derivs-hor}
  u_x=[\Omega_x^{\mathrm{hor}},u],\qquad u_y=[\Omega_y^{\mathrm{hor}},u].
\end{equation}
Under the gauge change \eqref{eq:gauge}, the Maurer--Cartan form changes by adding a
vertical term, so the horizontal components and the commutators in \eqref{eq:u-derivs-hor}
capture the gauge-invariant variation of $u$.

\section{Quaternionic form of the obstruction}
The obstruction $G=i_x+i\,i_y$ can be rewritten in terms of $u$.
Up to the scalar normalizations in \eqref{eq:u-def}, the intrinsic combination governing
deviations from the constant theory corresponds to
\begin{equation}\label{eq:uq-obstruction}
  \mathcal G:=u_x+u\,u_y.
\end{equation}
Using \eqref{eq:u-derivs-hor}, we may express $\mathcal G$ in Hopf variables as
\begin{equation}\label{eq:uq-obstruction-MC}
  \mathcal G
  =[\Omega_x^{\mathrm{hor}},u] \;+\; u\,[\Omega_y^{\mathrm{hor}},u].
\end{equation}
This expression is gauge-invariant (it is insensitive to vertical modifications of the lift).

\section{Quaternionic rigidity theorem}
\begin{theorem}[Quaternionic rigidity]\label{thm:quat-rigidity}
Let $u:\Omega\to S^2\subset\operatorname{Im}\mathbb H$ be the normalized generator
\eqref{eq:u-def} associated with an elliptic quadratic structure \eqref{eq:structure-poly}.
Let $q$ be a local Hopf lift \eqref{eq:hopf-lift} with Maurer--Cartan components
\eqref{eq:MC} and decomposition \eqref{eq:vh-decomp}.
Then the following are equivalent:
\begin{enumerate}
\item[\textup{(i)}] \textup{(Quaternionic rigidity)} $\ \mathcal G\equiv 0$, i.e.
\begin{equation}\label{eq:quat-rigidity}
  u_x+u\,u_y=0 \quad\text{on }\Omega.
\end{equation}
\item[\textup{(ii)}] \textup{(Horizontal Maurer--Cartan cancellation)}
\begin{equation}\label{eq:MC-cancel}
  [\Omega_x^{\mathrm{hor}},u] \;+\; u\,[\Omega_y^{\mathrm{hor}},u]=0 \quad\text{on }\Omega.
\end{equation}
\item[\textup{(iii)}] \textup{(Purely vertical transport along the characteristic direction)}
the pulled-back Maurer--Cartan form has no effective horizontal component along the
direction $\partial_x+u\,\partial_y$, i.e. $u$ undergoes only fiber (Hopf) gauge rotation
along that direction.
\end{enumerate}
Moreover, \eqref{eq:quat-rigidity} is the Hopf/quaternionic avatar of the original rigidity
condition $G=i_x+i\,i_y=0$.
\end{theorem}

\begin{proof}
By \eqref{eq:u-derivs-hor}, we have $u_x=[\Omega_x^{\mathrm{hor}},u]$ and
$u_y=[\Omega_y^{\mathrm{hor}},u]$. Substituting these into \eqref{eq:uq-obstruction}
yields \eqref{eq:uq-obstruction-MC}, and hence
$\mathcal G\equiv 0$ is equivalent to \eqref{eq:MC-cancel}. The gauge statement follows
from the fact that changing the lift by \eqref{eq:gauge} alters only the vertical part of the
Maurer--Cartan form, leaving the horizontal commutator expressions unchanged.
\end{proof}

\section{Interpretation}
Equation \eqref{eq:quat-rigidity} says that the $S^2$-valued imaginary direction field $u$
is transported compatibly with its own multiplication rule: along the characteristic direction
$\partial_x+u\,\partial_y$, the field exhibits no genuine horizontal twisting in $S^2$, only
Hopf fiber rotation. In this sense rigidity is a horizontality cancellation condition in the
Hopf fibration rather than an analytic assumption.

\chapter{Rigidization as a Computational Normal Form}
\label{app:rigidization-computational}

\section*{Purpose of this Appendix}

This appendix clarifies the intended scope of \emph{rigidization} within the
theory of variable elliptic structures.
Rigidization is not proposed as a replacement for global uniformization
(Beltrami theory), nor as a universal flattening mechanism.
Instead, it is designed as a \emph{computationally optimized normal form}
that exploits intrinsic transport structure when present, and that can be
used either as a standalone reduction or as a preprocessing step for
uniformization-based solvers.

The guiding philosophy is:
\begin{quote}
Rigidization removes exactly the analytic obstructions that prevent the use of
the constant-coefficient toolbox, while preserving geometric variability
that is irrelevant for computation but expensive to eliminate.
\end{quote}

\section{Rigidization versus Uniformization}

Let $\lambda(x,y)$ denote the spectral slope of a variable elliptic structure,
with intrinsic obstruction
\[
T := \lambda_x + \lambda\,\lambda_y .
\]

\subsection*{Uniformization (Beltrami)}

Uniformization seeks a diffeomorphism $\Phi$ such that the pullback structure
is \emph{constant}.  This solves the maximal geometric problem but requires
solving a globally coupled nonlinear elliptic PDE\@.
From a computational standpoint, uniformization is:
\begin{itemize}
\item global and nonlocal,
\item expensive to compute,
\item insensitive to existing transport structure,
\item often unnecessary when only analytic compatibility is required.
\end{itemize}

\subsection*{Rigidization}

Rigidization seeks a diffeomorphism $\Phi$ such that the pullback structure
satisfies
\[
\widetilde T = 0,
\qquad
\text{i.e.}\quad
\widetilde\lambda_X + \widetilde\lambda\,\widetilde\lambda_Y = 0,
\]
but does \emph{not} require $\widetilde\lambda$ to be constant.

This is a strictly weaker requirement than uniformization.
Nevertheless, once rigidity is achieved:
\begin{itemize}
\item the generalized Cauchy--Riemann system becomes homogeneous,
\item Cauchy--Pompeiu and similarity principles apply,
\item the full constant-structure analytic toolkit becomes available.
\end{itemize}

Thus rigidization reaches the minimal normal form required for analysis,
without over-solving the geometric problem.

\section{Why ``Weaker'' Means ``Easier''}

The rigidization condition is first-order in the structure and second-order
in the coordinates, and its linearization is an elliptic system whose
principal symbol is explicitly computable.
As shown in Chapter~\ref{ch:rigidization-convergence}, this enables:
\begin{itemize}
\item local Newton convergence under a quantitative small-obstruction condition
      (Theorem~\ref{thm:main-convergence}),
\item inner solvers based on transport--wave splitting (shifted ADI with
      Wachspress parameters),
\item direct control of analytic compatibility without global flattening.
\end{itemize}

In contrast, full uniformization removes curvature and transport simultaneously,
even when curvature plays no analytic role for the PDE being solved.

From a solver perspective, rigidization is therefore cheaper, more local,
and better aligned with problems that already possess a preferred transport
geometry.

\section{A Solver-Oriented Triage}

Rigidization is intended to be used selectively.
Three intrinsic indicators guide solver choice:
\begin{enumerate}
\item \textbf{Obstruction magnitude:}
\[
\rho_T := \frac{\|T\|_{C^{0,\alpha}}}{b_{\min}^2},
\qquad
b_{\min} := \inf \operatorname{Im}\lambda .
\]
\item \textbf{Metric curvature magnitude:} a scale-normalized curvature
indicator $\rho_K$ for the canonical metric.
\item \textbf{Transport dominance:} a problem-dependent measure of whether
a preferred first-order flow direction is present.
\end{enumerate}

These lead to the following computational policy:

\begin{itemize}
\item \textbf{Small obstruction ($\rho_T \ll 1$):}
Rigidize immediately using the Newton scheme of Chapter~\ref{ch:rigidization-convergence}.
This is the regime covered by Theorem~\ref{thm:main-convergence} with full convergence guarantees.

\item \textbf{Moderate obstruction, transport-dominated PDE:}
Attempt rigidization via a continuity strategy (interpolating from a rigid
reference structure $\lambda_0 = i$ via
$\lambda^t = (1-t)\lambda_0 + t\lambda$, taking incremental steps in~$t$
with each step small enough for the Newton iteration of
Theorem~\ref{thm:main-convergence} to converge).
If the obstruction is substantially reduced, proceed in rigid coordinates;
otherwise fall back to uniformization.

\item \textbf{Large obstruction, negligible curvature:}
Use Beltrami uniformization directly.

\item \textbf{Diffusion-dominated problems:}
Use uniformization; rigidization offers no advantage when no transport
structure is present.
\end{itemize}

This policy is not ad hoc: it reflects the analytic role played by the
obstruction in the failure of derivation properties, even though the specific
thresholds and decision boundaries are solver-dependent.
\section{Rigidization as a Preconditioner}

Even when rigidization does not succeed globally, it can still be valuable.

Chapter~\ref{ch:rigidization-convergence} shows that naive transport--wave splitting does not converge as a
standalone iteration, but is effective as an \emph{inner solver} for the
linearized elliptic problem.
This mirrors standard numerical practice:
\begin{itemize}
\item rigidization removes transport incompatibility first,
\item uniformization then operates on a better-conditioned structure,
\item overall solver stability and convergence improve.
\end{itemize}

In this sense, rigidization acts as a \emph{geometric preconditioner} for
uniformization-based methods.

More concretely, the incremental continuity strategy described above provides
a practical algorithm even without a closed-form global existence theorem:
one advances in~$t$ until the Newton step fails to converge, records the
partially rigidized structure, and hands the residual obstruction to a
uniformization solver operating on a problem whose effective $\rho_T$
has been reduced.

\section{The Global Rigidization Problem}
\label{sec:global-rigidization-open}

It is natural to ask whether rigidization can always be achieved on compact
domains, without any smallness condition on the obstruction~$T$.

\begin{conjecture}[Global rigidization]
\label{conj:global-rigidization}
Let $\Omega\subset\mathbb{R}^2$ be a compact simply connected domain with
$C^{2,\alpha}$ boundary.
Let $\lambda\in C^{2,\alpha}(\bar\Omega)$ with
$\operatorname{Im}\lambda \ge b_{\min}>0$ on~$\bar\Omega$.
Then there exists a $C^{2,\alpha}$ diffeomorphism
$\Phi:\bar\Omega\to\Phi(\bar\Omega)$ such that the pullback structure is rigid.
\end{conjecture}

A natural proof strategy is the continuity method: interpolating from a rigid
reference $\lambda_0 = i$ to the target~$\lambda$, the success set is nonempty
(at $t=0$) and open (by Theorem~\ref{thm:main-convergence}).
Closedness would follow from uniform a~priori estimates on the rigidizing
diffeomorphisms---specifically, a uniform $C^{2,\alpha}$ bound and a uniform
Jacobian lower bound along the family.

These a~priori estimates constitute the core difficulty.
Three coupled sub-problems are involved:
\begin{enumerate}
\item \textbf{Jacobian lower bound.}
Show that $\inf_{\bar\Omega} J_\Phi \ge c > 0$ for any rigidizing
diffeomorphism, with~$c$ depending only on $b_{\min}$,
$\|\lambda\|_{C^{2,\alpha}}$, and~$\Omega$.
A promising route is to derive a second-order elliptic equation (or
inequality) for~$\log J_\Phi$ from the ratio condition and apply the
maximum principle, using the boundary normalization
$\Phi|_{\partial\Omega}=\mathrm{id}$ (so $J=1$ on~$\partial\Omega$).

\item \textbf{Burgers constraint on compact domains.}
The pullback~$\widetilde\lambda$ solves conservative Burgers, whose
solutions are constant along characteristics.
On compact domains with $\operatorname{Im}\widetilde\lambda>0$, the
characteristic map from boundary data to the interior is expected to be
injective (no shock formation in the elliptic regime).
A correct proof requires a quantitative injectivity estimate on the
characteristic map---specifically, a determinant lower bound on the map
$(s,X)\mapsto(X,Y)$ involving $\partial\widetilde\lambda/\partial s$
and domain geometry, not merely the condition
$\operatorname{Im}\widetilde\lambda>0$.

\item \textbf{$C^{2,\alpha}$ a~priori bound.}
With the Jacobian controlled, the rigidization system is uniformly
elliptic.
Schauder estimates (Agmon--Douglis--Nirenberg) should yield global
$C^{2,\alpha}$ bounds, but the application to this specific quasilinear
$2\times2$ system requires verifying the structural conditions under
which elliptic regularity and boundary estimates hold.
\end{enumerate}

These sub-problems are mutually dependent: the Burgers constraint feeds into
the Jacobian bound, which in turn feeds into the Schauder estimates.
As of this writing, none of the three has been rigorously closed, and the
conjecture remains open.

We record it here both as an invitation to further work and as context for
the solver-oriented approach of this appendix: the computational triage
and preconditioner strategies above are designed to be effective regardless
of whether Conjecture~\ref{conj:global-rigidization} is eventually proved.

\section{What Is \emph{Not} Claimed}

To be explicit about the boundaries of the current theory, this monograph
does \emph{not} claim:
\begin{itemize}
\item global rigidization for arbitrary elliptic structures on compact
      domains (this remains an open problem;
      see Conjecture~\ref{conj:global-rigidization}),
\item a priori Jacobian lower bounds independent of the obstruction
      magnitude (the estimates in Chapter~\ref{ch:rigidization-convergence} depend on the smallness
      condition $\|T\|<C\,b_{\min}^2$),
\item replacement of Beltrami theory as a universal tool.
\end{itemize}

These are not limitations of design intent but reflections of the current
state of the analysis.
The global theory may well be within reach---the structural ingredients
(ellipticity of the linearized problem, Burgers characteristic geometry,
similarity-principle arguments) are all present---but assembling them into
a complete a~priori package requires new ideas, particularly for the
control of the Jacobian.

\section{Summary}

Rigidization should be viewed as:
\begin{itemize}
\item a maximal \emph{computable} normal form,
\item a solver-side exploitation of transport geometry,
\item a selective alternative or complement to uniformization,
\item a mechanism that preserves just enough geometry to be efficient.
\end{itemize}

When it applies, rigidization delivers the full analytic power of the
constant-coefficient theory at a fraction of the computational cost.
When its applicability has not been established, it still provides valuable
structural diagnostics and preconditioning for heavier methods.

In this sense, rigidization is not a competitor to uniformization but a
practical intermediary between raw variable geometry and full geometric
flattening---one whose global reach remains a compelling open question in
its own right.

\chapter[Symbolic Verification Second--Order Expansion]{Symbolic Verification of the Rigid Second--Order Expansion}
\label{app:sympy-verification}

\noindent
This appendix contains a complete symbolic verification of
Theorem~\ref{thm:rigid-dzdbar} using the computer algebra system
\texttt{SymPy}.

\medskip
\noindent
The script expands the operator $4\,\partial_z\partial_{\bar z}$ acting on
$f=u+v\,i$ using the defining structure relation
\[
i^2+\beta i+\alpha=0,
\]
substitutes the rigidity (Burgers) system
\[
\alpha_x=\alpha\,\beta_y,
\qquad
\beta_x+\alpha_y=\beta\,\beta_y,
\]
and simplifies the resulting expressions in the basis $\{1,i\}$.

\medskip
\noindent
It verifies that:
\begin{itemize}
\item the principal part is exactly $L_{\alpha,\beta}$ acting on $(u,v)$;
\item all second--order cross terms cancel under rigidity;
\item the remaining terms are purely first order and coincide with
     \eqref{eq:R0}--\eqref{eq:R1};
\item no zero--order terms occur.
\end{itemize}

\medskip
\noindent
The computation is purely algebraic and does not rely on any numerical
approximation.

\bigskip

\begin{lstlisting}[language=Python,
caption={SymPy verification of Theorem~\ref{thm:rigid-dzdbar}},
label={lst:sympy-verify}]
# --- begin SymPy verification script ---

#!/usr/bin/env python3
"""
verify_rigid_dzdbar.py

SymPy verification of the second-order expansion in the chapter
'Second-Order Operators and Factorization in the Rigid Regime'.

It verifies, under the rigidity (Burgers) system
   alpha_x = alpha*beta_y,
   beta_x  = beta*beta_y - alpha_y,
that for f = u + v i one has

   4 d_z d_{\bar z} f
     = (L_{alpha,beta} u + R0[u,v]) + (L_{alpha,beta} v + R1[u,v]) i,

with
   L_{alpha,beta} = d_x^2 - beta d_{xy} + alpha d_y^2,
   R0[u,v] = alpha_y u_y + alpha_y v_x - 2 alpha beta_y v_y,
   R1[u,v] = beta_y u_y + beta_y v_x + (2 alpha_y - 2 beta beta_y) v_y.

The computation is done in the 2D commutative algebra generated by i with
   i^2 + beta i + alpha = 0
and uses the identity (elliptic regime)
   (2i+beta)^{-1} = (-beta - 2i)/Delta,  Delta=4alpha-beta^2,
to express i_x and i_y from the differentiated structure polynomial.

Run:
   python3 verify_rigid_dzdbar.py
"""

import sympy as sp


def main():
   # Base variables and scalar coefficient functions
   x, y = sp.symbols("x y")
   alpha = sp.Function("alpha")(x, y)
   beta = sp.Function("beta")(x, y)
   u = sp.Function("u")(x, y)
   v = sp.Function("v")(x, y)

   Delta = 4 * alpha - beta**2

   # --- Algebra representation: pairs (a0, a1) correspond to a0 + a1*i ---
   def add(a, b):
       return (sp.simplify(a[0] + b[0]), sp.simplify(a[1] + b[1]))

   def smul(s, a):
       return (sp.simplify(s * a[0]), sp.simplify(s * a[1]))

   def mul(a, b):
       # (a0 + a1 i)(b0 + b1 i) with i^2 = -beta i - alpha
       a0, a1 = a
       b0, b1 = b
       c0 = a0 * b0 - alpha * a1 * b1
       c1 = a0 * b1 + a1 * b0 - beta * a1 * b1
       return (sp.simplify(c0), sp.simplify(c1))

   I = (sp.Integer(0), sp.Integer(1))        # i
   hatI = (-beta, sp.Integer(-1))            # \hat i = -beta - i

   # (2i+beta)^{-1} = (-beta - 2i)/Delta
   inv_twoIplusbeta = (-beta / Delta, sp.Integer(-2) / Delta)

   # --- i_x and i_y from differentiated structure polynomial ---
   alpha_x = sp.diff(alpha, x)
   beta_x = sp.diff(beta, x)
   alpha_y = sp.diff(alpha, y)
   beta_y = sp.diff(beta, y)

   ix = mul((-alpha_x, -beta_x), inv_twoIplusbeta)  # -(alpha_x + beta_x i)/(2i+beta)
   iy = mul((-alpha_y, -beta_y), inv_twoIplusbeta)  # -(alpha_y + beta_y i)/(2i+beta)

   # --- Derivatives of a section f = u + v i, i.e. pair (u,v) ---
   def dx(f):
       f0, f1 = f
       # (u_x + v_x i) + v i_x
       return add((sp.diff(f0, x), sp.diff(f1, x)), smul(f1, ix))

   def dy(f):
       f0, f1 = f
       # (u_y + v_y i) + v i_y
       return add((sp.diff(f0, y), sp.diff(f1, y)), smul(f1, iy))

   def dbar(f):
       # d_{\bar z} = 1/2 (d_x + i d_y)
       return smul(sp.Rational(1, 2), add(dx(f), mul(I, dy(f))))

   def dz(f):
       # d_z = 1/2 (d_x + \hat i d_y)
       return smul(sp.Rational(1, 2), add(dx(f), mul(hatI, dy(f))))

   f = (u, v)
   lhs = smul(4, dz(dbar(f)))  # pair (lhs0, lhs1)

   # --- Rigidity/Burgers substitutions (eliminate x-derivatives of alpha,beta) ---
   subs = {
       sp.diff(alpha, x): alpha * sp.diff(beta, y),
       sp.diff(beta, x): beta * sp.diff(beta, y) - sp.diff(alpha, y),
   }
   # mixed and second x-derivatives via differentiating the Burgers system
   subs.update(
       {
           sp.diff(alpha, x, y): sp.diff(alpha * sp.diff(beta, y), y),
           sp.diff(alpha, y, x): sp.diff(alpha * sp.diff(beta, y), y),
           sp.diff(alpha, x, x): sp.diff(alpha * sp.diff(beta, y), x),
           sp.diff(beta, x, y): sp.diff(beta * sp.diff(beta, y) - sp.diff(alpha, y), y),
           sp.diff(beta, y, x): sp.diff(beta * sp.diff(beta, y) - sp.diff(alpha, y), y),
           sp.diff(beta, x, x): sp.diff(beta * sp.diff(beta, y) - sp.diff(alpha, y), x),
       }
   )

   def enforce_rigidity(expr):
       e = expr
       for _ in range(6):
           e = sp.expand(e)
           e = e.xreplace(subs)
           e = sp.simplify(sp.expand(e))
       return sp.simplify(e)

   lhs0 = enforce_rigidity(lhs[0])
   lhs1 = enforce_rigidity(lhs[1])

   # --- RHS (principal part + first-order correction) ---
   L_u = sp.diff(u, x, 2) - beta * sp.diff(u, x, y) + alpha * sp.diff(u, y, 2)
   L_v = sp.diff(v, x, 2) - beta * sp.diff(v, x, y) + alpha * sp.diff(v, y, 2)

   R0 = alpha_y * sp.diff(u, y) + alpha_y * sp.diff(v, x) - 2 * alpha * beta_y * sp.diff(v, y)
   R1 = beta_y * sp.diff(u, y) + beta_y * sp.diff(v, x) + (2 * alpha_y - 2 * beta * beta_y) * sp.diff(v, y)

   rhs0 = enforce_rigidity(L_u + R0)
   rhs1 = enforce_rigidity(L_v + R1)

   diff0 = sp.simplify(lhs0 - rhs0)
   diff1 = sp.simplify(lhs1 - rhs1)

   print("Component (1) difference after rigidity simplification:")
   print(diff0)
   print("\nComponent (i) difference after rigidity simplification:")
   print(diff1)

   assert diff0 == 0, "Scalar component does not match."
   assert diff1 == 0, "i-component does not match."

   print("\nOK: Identity verified (both components simplify to 0).")


if __name__ == "__main__":
   main()


# --- end script ---
\end{lstlisting}

\chapter[Symbolic Verification Rigidity-Flatness]{Symbolic Verification for the Rigidity-Flatness Theorem Condition}

\UseRawInputEncoding
\begin{lstlisting}[language=Python,
caption={SymPy verification of Theorem~\ref{thm:rigid-dzdbar}},
label={lst:sympy-rig-flat-verify}]
# --- begin SymPy verification script ---
#!/usr/bin/env python3
"""
rigid_flat_vacuum_check.py

Reproducible SymPy computations for the "full vacuum" question in the
planar elliptic structure framework:

  - Define tau = p + i q (q>0), with canonical metric
        ds^2 = (1/q) (dx^2 + 2 p dx dy + (p^2 + q^2) dy^2).
  - Compute the Gaussian curvature K of this metric for general p(x,y), q(x,y).
  - Impose rigidity (torsion-free / conservative transport):
        tau_x + tau tau_y = 0
    i.e.
        p_x + p p_y - q q_y = 0
        q_x + p q_y + q p_y = 0
    and substitute the consequent relations for second derivatives
    to obtain K|_{rigid}.
  - Sanity-check with an explicit nontrivial rigid family:
        tau = (eps*y + i) / (1 + eps*x)

Outputs:
  - General K(p,q) as a symbolic expression
  - Rigid-reduced curvature K_rigid (expressed in y-derivatives only)
  - Curvature for the explicit rigid family

Run:
  python3 rigid_flat_vacuum_check.py

Dependencies:
  sympy
"""

import sympy as sp


def main():
    # Coordinates
    x, y = sp.symbols('x y', real=True)

    # Real/imag parts of tau = p + i q (q>0)
    p = sp.Function('p')(x, y)
    q = sp.Function('q')(x, y)

    # Canonical metric in (x,y):
    # ds^2 = (1/q) (dx^2 + 2 p dx dy + (p^2 + q^2) dy^2)
    g = sp.Matrix([
        [1/q,      p/q],
        [p/q, (p**2 + q**2)/q]
    ])
    g_inv = sp.simplify(g.inv())

    coords = [x, y]

    def d(expr, var):
        return sp.diff(expr, var)

    # Christoffel symbols Γ^k_{ij}
    Gamma = [[[0, 0], [0, 0]] for _ in range(2)]
    for k in range(2):
        for i in range(2):
            for j in range(2):
                s = 0
                for l in range(2):
                    s += g_inv[k, l] * (d(g[j, l], coords[i]) +
                                        d(g[i, l], coords[j]) -
                                        d(g[i, j], coords[l]))
                Gamma[k][i][j] = sp.simplify(sp.Rational(1, 2) * s)

    # Riemann tensor R^l_{ijk}
    Riem = [[[[0 for _ in range(2)] for _ in range(2)] for _ in range(2)] for _ in range(2)]
    for l in range(2):
        for i in range(2):
            for j in range(2):
                for k in range(2):
                    term = d(Gamma[l][i][k], coords[j]) - d(Gamma[l][i][j], coords[k])
                    for m in range(2):
                        term += Gamma[l][j][m] * Gamma[m][i][k] - Gamma[l][k][m] * Gamma[m][i][j]
                    Riem[l][i][j][k] = sp.simplify(term)

    # Ricci tensor and scalar curvature
    Ricci = sp.Matrix([[0, 0], [0, 0]])
    for i in range(2):
        for k in range(2):
            Ricci[i, k] = sp.simplify(sum(Riem[l][i][l][k] for l in range(2)))

    R_scalar = sp.simplify(sum(g_inv[i, k] * Ricci[i, k] for i in range(2) for k in range(2)))
    K = sp.simplify(R_scalar / 2)  # Gaussian curvature in 2D

    print("\n=== General Gaussian curvature K(p,q) ===")
    print(K)

    # --- Rigidity (torsion-free) equations: tau_x + tau tau_y = 0 ---
    # Split into real equations for p,q:
    #   p_x + p p_y - q q_y = 0
    #   q_x + p q_y + q p_y = 0
    px = sp.Derivative(p, x)
    py = sp.Derivative(p, y)
    qx = sp.Derivative(q, x)
    qy = sp.Derivative(q, y)

    # Solve for x-derivatives
    p_x_expr = -p * py + q * qy
    q_x_expr = -p * qy - q * py

    # Differentiate in y to get mixed derivatives
    p_xy_expr = sp.diff(p_x_expr, y)
    q_xy_expr = sp.diff(q_x_expr, y)

    # Differentiate in x and eliminate x-derivatives via rigidity
    p_xx_expr = sp.diff(p_x_expr, x).xreplace({px: p_x_expr, qx: q_x_expr})
    q_xx_expr = sp.diff(q_x_expr, x).xreplace({px: p_x_expr, qx: q_x_expr})

    # Substitute mixed derivatives to eliminate remaining x-dependence
    p_xx_expr = sp.simplify(p_xx_expr.subs({
        sp.Derivative(p, x, y): p_xy_expr,
        sp.Derivative(q, x, y): q_xy_expr
    }))
    q_xx_expr = sp.simplify(q_xx_expr.subs({
        sp.Derivative(p, x, y): p_xy_expr,
        sp.Derivative(q, x, y): q_xy_expr
    }))

    subs_rigid = {
        sp.Derivative(p, x): p_x_expr,
        sp.Derivative(q, x): q_x_expr,
        sp.Derivative(p, x, 2): p_xx_expr,
        sp.Derivative(q, x, 2): q_xx_expr,
        sp.Derivative(p, x, y): p_xy_expr,
        sp.Derivative(q, x, y): q_xy_expr,
    }

    K_rigid = sp.simplify(sp.together(K.subs(subs_rigid)))

    print("\n=== Curvature reduced under rigidity (T=0) ===")
    print(sp.factor(K_rigid))

    # --- Sanity check: explicit rigid family tau = (eps*y + i)/(1 + eps*x) ---
    eps = sp.symbols('eps', real=True)
    p_ex = eps * y / (1 + eps * x)
    q_ex = 1 / (1 + eps * x)

    K_family = sp.simplify(K.subs({p: p_ex, q: q_ex}))

    print("\n=== Curvature for rigid family tau=(eps*y+i)/(1+eps*x) ===")
    print(sp.factor(K_family))

    # Optional: show that K_family == 0 implies eps == 0 (generic statement)
    # This is not a full proof, but indicates nontrivial family is not flat.
    print("\nDone.")


if __name__ == "__main__":
    main()


\end{lstlisting}

\chapter*{Acknowledgements}

\noindent\textbf{Use of Generative AI Tools.}
Portions of the writing and editing of this manuscript were assisted by generative AI language tools. These tools were used to improve clarity of exposition, organization of material, and language presentation. All mathematical results, statements, proofs, and interpretations were developed, verified, and validated by the author. The author takes full responsibility for the accuracy, originality, and integrity of all content in this work, including any material produced with the assistance of AI tools. No generative AI system is listed as an author of this work.

\end{document}